\documentclass{amsart}
\usepackage{amsmath}
\usepackage{amssymb}
\usepackage{amsthm}
\usepackage{caption}
\usepackage[labelformat=simple,labelfont={}]{subcaption}
\usepackage{enumerate}
\usepackage{float}
\usepackage[letterpaper, margin = 1in]{geometry}
\usepackage{graphicx}
\usepackage[numbers]{natbib}
\usepackage{setspace}
\usepackage{subcaption}
\usepackage[all]{xy}
\usepackage{times}
\theoremstyle{plain}

\theoremstyle{plain}

\newtheorem{theorem}{Theorem}[section]

\newtheorem{lemma}[theorem]{Lemma}
\newtheorem{corollary}[theorem]{Corollary}

\theoremstyle{remark}
\newtheorem*{remark}{Remark}
\newtheorem*{notation}{Notation}

\newtheorem*{definition}{Definition}

\DeclareMathOperator{\spn}{span}
\newcommand{\bspn}[1]{\ensuremath{\spn\langle #1 \rangle}}

\numberwithin{equation}{section}

\newcommand{\sk}[1]{\raisebox{-5mm}{\includegraphics[height=36pt]{5skein#1.pdf}}}

\title{Multi-crossing Number for knots and the Kauffman Bracket Polynomial}
\author[C. Adams et al]
{Colin Adams, Orsola Capovilla-Searle, Jesse Freeman, Daniel Irvine, Samantha Petti, Daniel Vitek, Ashley Weber, Sicong Zhang}
\address{Colin Adams, Department of Mathematics, Williams College, Williamstown, MA 01267}
\email{cadams@williams.edu}
\address{Orsola Capovilla-Searle, Department of Mathematics, Bryn Mawr College,  Bryn Mawr, PA 19010-2899}
\email{ocapovilla@brynmawr.edu}
\address{Jesse Freeman, Department of Mathematics, Williams College, Williamstown, MA 01267}
\email{Jesse.B.Freeman@williams.edu}
\address{Daniel Irvine, Department of Mathematics, University of Michigan, 530 Church Street, Ann Arbor, MI 48109-1043}
\email{DIRVINE@umich.edu}
\address{Samantha Petti, Department of Mathematics, Williams College, Williamstown, MA 01267}
\email{snp1@williams.edu}
\address{Daniel Vitek, Department of Mathematics, Fine Hall, Princeton University, Princeton, NJ 08544-1000}
\email{dvitek@math.princeton.edu}
\address{Ashley Weber, Department of Mathematics, 151 Thayer Street, Brown University, Providence, RI 02912}
\email{aweber@math.brown.edu}
\address{Sicong Zhang, Department of Mathematics, Building 380, Stanford University, Stanford, CA, 94305}
\email{zhangsc91@gmail.com}

\begin{document}

\maketitle
\begin{abstract} 
A multi-crossing (or $n$-crossing) is a singular point in a projection at which $n$ strands cross so that each strand bisects the crossing. We generalize the classic result of Kauffman, Murasugi, and Thistlethwaite relating the span of the bracket polynomial to the double-crossing number of a link,  $\text{span}\langle K \rangle\leq 4c_2$, to the $n$-crossing number. In this paper we find the following lower bound on the $n$-crossing number in terms of the span of the bracket polynomial for any $n \geq 3$:
$$ \text{span} \langle K \rangle \leq \left(\left\lfloor\frac{n^2}{2}\right\rfloor + 4n - 8\right) c_n(K).$$ 
We also explore $n$-crossing additivity under composition, and find that for $n\geq4$ there are examples of knots $K_1$ and $K_2$ such that $c_n(K_1\#K_2)=c_n(K_1)+c_n(K_2)-1$. Further, we present the the first extensive list of calculations of $n$-crossing number knots. Finally, we explore the monotonicity of the sequence of $n$-crossings of a knot, which we call the crossing spectrum.\end{abstract}

\section{Introduction}

The classical projection of a link $K$ only considers crossings where two strands meet and bisect each other. In \citep{Triple},  an $n$-crossing, also known as multi-crossing, was introduced. (See \citep{PT} and \citep{TT} for multi-crossings as applied to graph projections.) A multi-crossing is defined to be a singular point in a projection at which $n$ strands cross, such that each strand bisects the crossing. In \citep{Triple} it was shown that for every $n\geq 2$, any knot or link $K$ has a projection with only $n$-crossings.  Hence, one can define $c_{n}(K)$ to be the minimal number of $n$-crossings in a projection with only $n$-crossings. An $n$-crossing has $n$ strands that are labeled top to bottom $1,2,... ,n$ respectively. There are various types of $n$-crossings characterized by where the strands with different heights are located. For all $n$-crossings, we read the height of the strands clockwise around the crossings, always beginning with the top strand. 

The purpose of this paper is to find a lower bound on multi-crossing number in terms of the span of the bracket polynomial. Independently (see \citep{Kauffman87, Kauffman88, Murasugi, Thistlethwaite})Kauffman, Murasugi, and Thistletwaite proved that $\text{span}\langle K \rangle\leq 4c_2(K)$. In \citep{Triple} and \citep{Quadruple}, Adams used a similar approach to find a lower bound on triple and quadruple crossing number: $\text{span}\langle K \rangle \leq 8c_3(K)$ and $\text{span}\langle K \rangle \leq 16 c_4(K)$. 

Our main result shows that  for all $n \geq 3$, $$ \text{span}\langle K \rangle \leq \left(\left\lfloor\frac{n^2}{2}\right\rfloor + 4n - 8\right) c_n(K).$$  This bound agrees with the previous bounds for $n=3,4$.

To obtain the bound, we must determine how to compute the bracket polynomial directly from an $n$-crossing projection. We compute the bracket polynomial of a knot from its double crossing projection by resolving the crossings as some combination of A-splits and B-splits. Each unique way of resolving all of the crossings in a projection is called a state. Each state $s$ contributes a summand of $A^{a(s)}A^{-b(s)}(-A^{2}-A^{-2})^{|s|-1}$ to the bracket polynomial, where $|s|$ is the number of disjoint connected components, $a(s)$ is the number of A-splits and $b(s)$ is the number of B-splits. Thus, $<K>=\sum A^{a(s)}A^{-b(s)}(-A^{2}-A^{-2})^{|s|-1}$, where the sum is over all states $s$ of the projection. 

As shown in the traditional skein relation,  A-splits and B-splits are the only ways to resolve double crossings. However, a multi-crossing can be resolved in a variety of ways. To calculate the bracket polynomial from an $n$-crossing projection we need an $n$-skein relation. The 3 and 4-skein relations appear in \citep{Triple} and \citep{Quadruple}, respectively. In Figure 1 we provide an example of the 5-skein relation associated to the 12345 type 5-crossing. 

        ~ 

\begin{figure}[H]
\begin{align*}
<12345> &=A^2\left(\sk{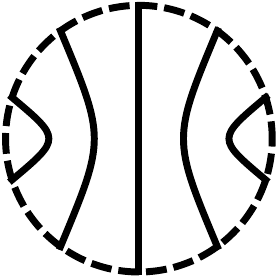}+\sk{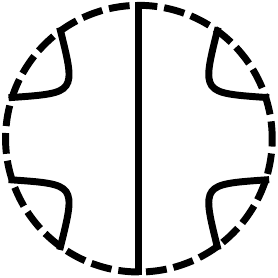}+\sk{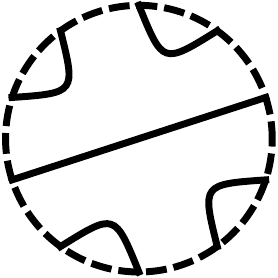}+\sk{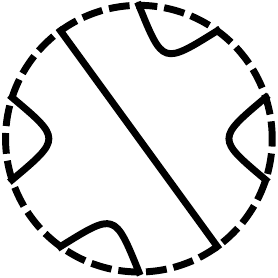}+\sk{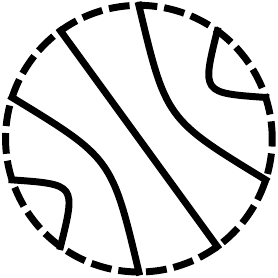}\right)\\
&\quad+A^0\left(\sk{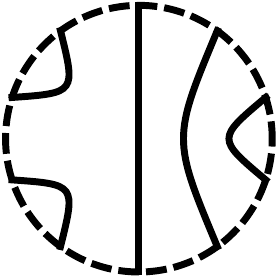}+\sk{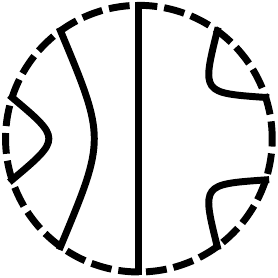}+\sk{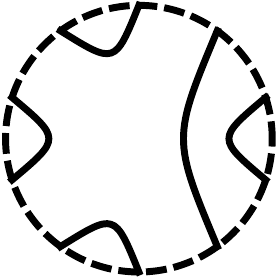}+\sk{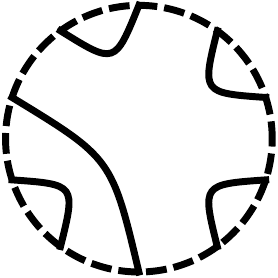}+\sk{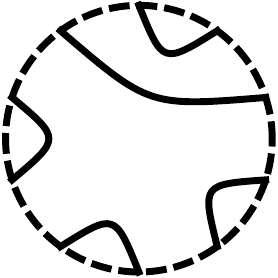}+\sk{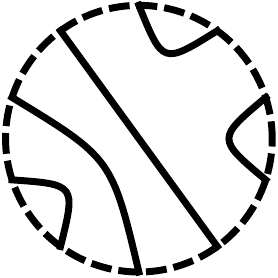}\right.\\
&\quad\left.+\ \sk{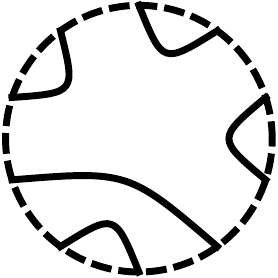}+\sk{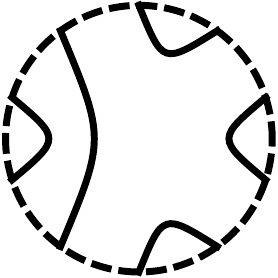}+\sk{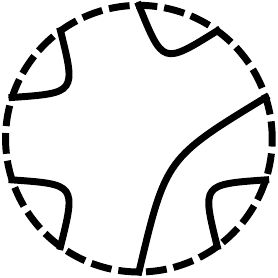}+\sk{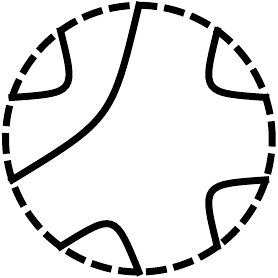}+\sk{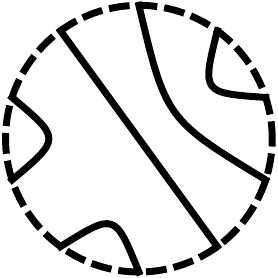}+\sk{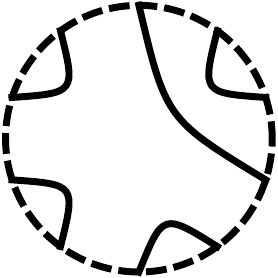}\right)\\
&\quad+A^{-2}\left(\sk{20}+\sk{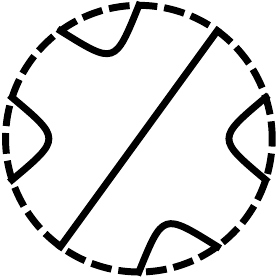}+\sk{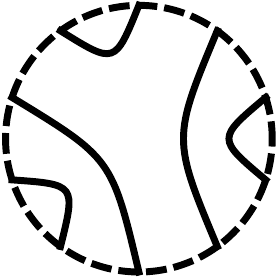}+\sk{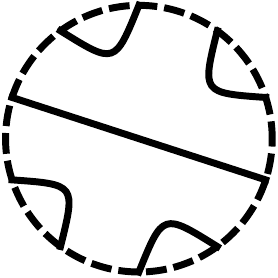}+2\ \sk{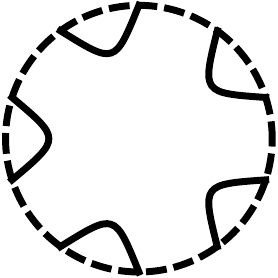}+\sk{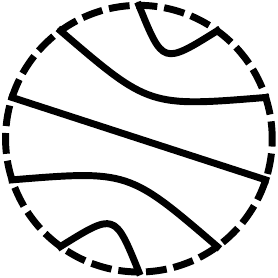}+\sk{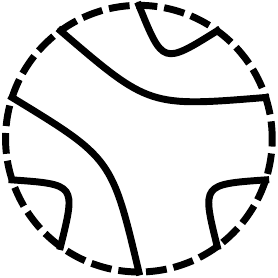}\right.\\
&\quad\left.+\ \sk{4}+2\ \sk{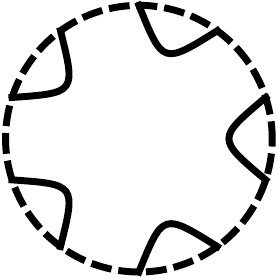}+\sk{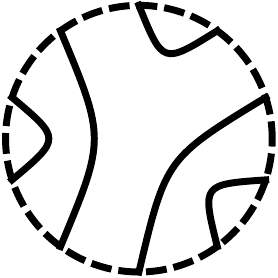}+\sk{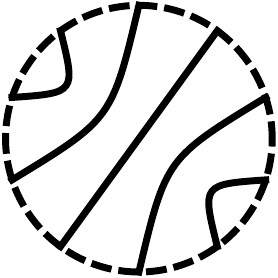}+\sk{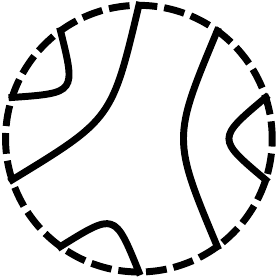}+\sk{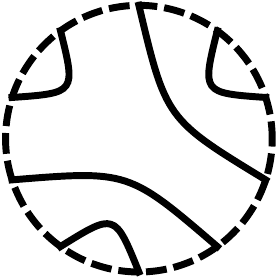}+\sk{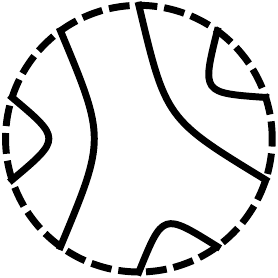}\right)\\
&\quad+A^{-4}\left(\sk{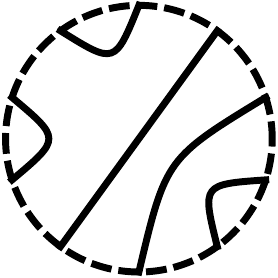}+\sk{34}+\sk{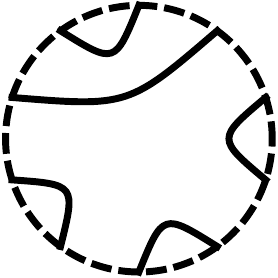}+\sk{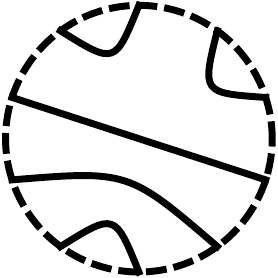}+\sk{30}+\sk{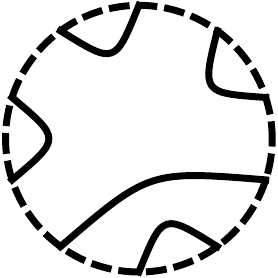}+\sk{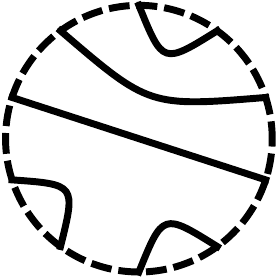}\right.\\
&\quad\left.+\ \sk{10}+\sk{3}+\sk{2}+\sk{6}+\sk{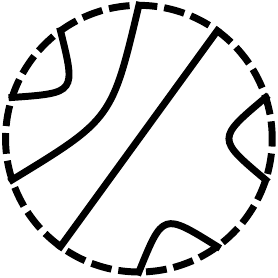}+\sk{24}+\sk{15}\right)\\
&\quad+A^{-6}\left(\sk{36}+\sk{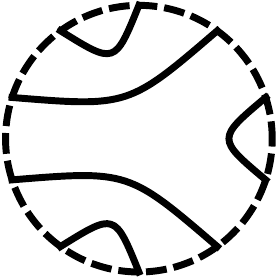}+\sk{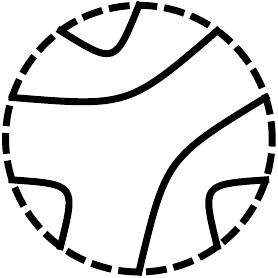}+\sk{32}+\sk{29}+\sk{8}+\sk{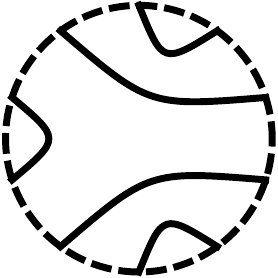}\right.\\
&\quad\left.+\ \sk{1}+\sk{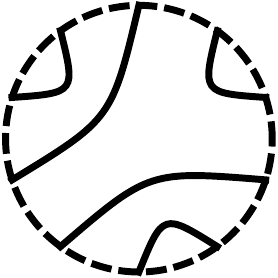}\right) + A^{-8}\left(\sk{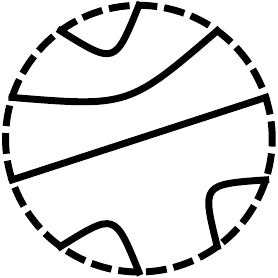}+\sk{38}+\sk{31}+\sk{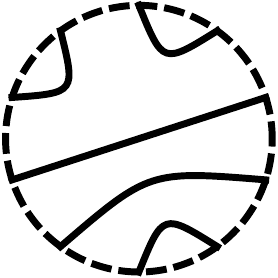}\right) + A^{-10}\ \sk{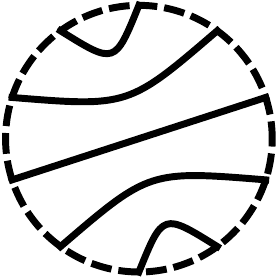}.
\end{align*}

\caption{The skein relation for the crossing $\langle12345\rangle$, with width 12.}\label{fig:5skein}
\end{figure}

\begin{definition}A \textit{term} $S$ of an $n$-skein relation consists of a particular split and the \textit{coefficient} associated to it. The coefficient is of the form $J A^t$, where $J$ and $t$ are integers. We define the power of $S$, denoted as $P(S)$, to be this value $t$.  The \textit{width} of a skein relation $R$, denoted $w(R)$,  is the difference between the highest power and lowest power appearing in the skein relation. 
\end{definition}

Note that there are multiple $n$-skein relations because each type of $n$-crossing generates a unique $n$-skein relation. For a fixed $n=2,3,4$, all $n$-skein relations for a given $n$ have the same width. This pattern does not hold, however, for $n=5$. There are twenty-two 5-skein relations with a width of 12 and two 5-skein relations with a width of 8. Similarly, for $n=6,7,8$, some $n$-skein relations have different widths as illustrated in the following table.

\begin{center}
    \begin{tabular}{ | c | c | c | p{5cm} |}
    \hline
  $ n $&    Realized Widths \\ \hline
    5&8,\text{ }12  \\ \hline
  6&14,\text{}16,\text{}18  \\ \hline
7&16,\text{}20,\text{}24   \\ \hline
8&24,\text{}26,\text{}28,\text{}30,\text{}32 \\ \hline
    \end{tabular}
\end{center}

The terms of an $(n+1)$-skein relation are obtained by identically adding a strand on top of the split associated to each term in the corresponding $n$-skein relation and then resolving the resulting double crossings as A-splits or B-splits. We then adjust the powers of the coefficients of the terms according to the rules in the 2-skein relation. For each split in the $n$-skein relation, we resolve the crossings in all possible combinations of $A$-splits and $B$-splits. The resulting terms form the $(n+1)$-skein relation.

 
To obtain the bracket polynomial we must split apart all the $n$-crossings in a projection in every possible way according to the terms of the $n$-skein relation.  We want to understand how terms in the $n$-skein relation are related to each other so as to understand which terms  contribute the highest and lowest powers to the bracket polynomial. To do so, we define split moves, split distance, high states and low states.

\begin{definition} A \textit{split move} on a split $S$ replaces two adjacent arcs in the split by the only other two arcs that connect their endpoints without forming additional crossings. 
\end{definition}

 \begin{figure}[h]
\begin{center}
\includegraphics[width=0.2\columnwidth,angle=0] {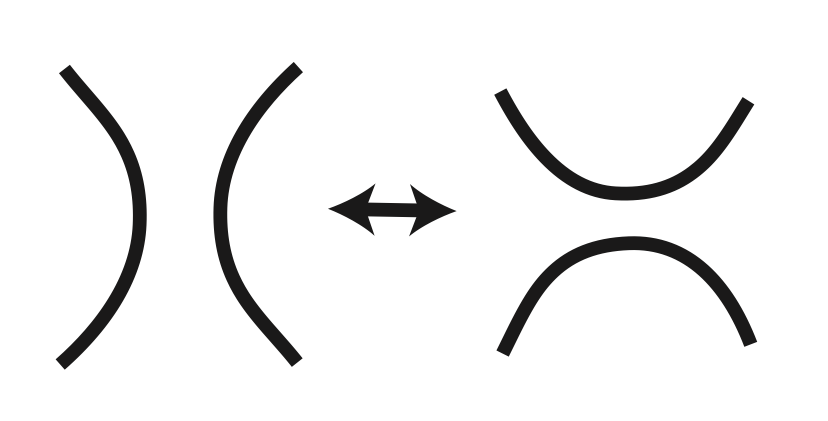} 
\caption{\label{sm} A split move.}
\end{center}
\end{figure}

Note that each split can be obtained from any other split by a sequence of split moves. Each move changes the number of components in a state by $\pm 1$.
 
\begin{definition} The \textit{split distance, $d(S,T)$} between two terms $S$ and $T$ is the minimum number of split moves required to change from the split of $S$ to the split of $T$. Note that this defines a metric on the set of splits.
\end{definition}

Given a multi-crossing with a fixed strand order, define a partial order on the terms in the corresponding skein relation as follows. Let $S$ and $T$ be terms in an n-skein relation. Define the covering relation $S \prec T$ if the splits of the terms $S$ and $T$ are related by one split move, and $P(S)=P(T)-2$. Extend ``$\prec $" to a partial order ``$<$" via transitivity, i.e. $S<T$ if and only if $S=S_0 \prec S_1 \prec \dots \prec S_k=T$ for a chain of splittings $\{S_i\}|_{0 \le i \le k}$.

\begin{definition} Given an n-skein relation, we define a \textit{high split} (resp. \textit{low split}) to be a split appearing in a term that is maximal (resp. minimal) with respect to the partial order ``$<$". \end{definition}

The following lemma, appearing in \citep{Quadruple}, generalizes here.

\begin{lemma}\label{high split is high}
Let $s$ be a state such that for a given multi-crossing, denoted as $x$, it is resolved with the split of a term $S$. If $T<S$ (resp. $T>S$), then changing the split at $x$ from the split of $S$ to the split of $T$ will not increase (resp. decrease) the highest (resp. lowest) power in the polynomial associated to $s$.
\end{lemma}

\begin{proof}
If $T \prec S$, the split of $S$ and the split of $T$ are 1 split move apart.  Thus replacing $S$ with $T$ changes $|s|$ by at most 1. Since $T \prec S$, $P(T)=P(S)-2$, and so the summand from the state corresponding to the  $T$ split cannot have a power higher than the summand from the state corresponding to the $S$ split. Extending by transitivity, if $T<S$, we can change from the split of $S$ to that of $T$ by a sequence of split moves that does not increase the highest power of the state. The proof for $T>S$ follows similarly.
\end{proof}

\begin{definition}
A \textit{maximal state}, denoted $s_{max}$, is a state that contributes a summand with the highest power of $A$ to the bracket polynomial, and such that all its splits are high splits. Let $M$ denote this highest power of $A$. A \textit{minimal state}, denoted $s_{min}$, is a state that contributes a summand with the lowest power of $A$ to the bracket polynomial, and such that all its splits are low splits. Let $m$ denote this lowest power of $A$. We refer to the number of connected components in the state $s_{max}$ as $|s_{max}|$ and the number in $s_{min}$ as $|s_{min}|$.
\end{definition}

We always require that $s_{max}$ states and $s_{min}$ states be the result of exclusively high splits or exclusively low splits. This is always possible by Lemma \ref{high split is high}.  Starting from any state $s$ that contributes the highest (resp. lowest) power to the bracket polynomial, we can always change each split to a high (resp. low) split to obtain a state $s'$ that also contributes the highest (resp. lowest) power. 

Given a projection of a link $K$ with $c_{n}(K)$ $n$-crossings, label the different $n$-crossings $x_{i}$, each of which is associated to a skein relation $R_{i}$, where $1\leq i \leq c_{n}$. Denote the highest power of $A$ appearing in the skein relation $R_{i}$ of $x_{i}$ as $h_{i}$ and the lowest power of $A$ as $l_{i}$. The span of $K$ is determined by the highest power of $A$ and the lowest power $A$. Then, the width of the skein relation $R_i$ associated to $x_{i}$ is $w(R_{i})=h_{i}-l_{i}$.

From the skein relations,
$$M \leq \left( \sum^{c_n}_{i=1} h_{i} \right) +2(|s_{max}|-1)$$
$$m \geq \left(\sum^{c_n}_{i=1} l_{i}\right) -2(|s_{min}|-1)$$

The above statements are inequalities rather than equalities because terms with high or low splits do not necessarily have the coefficients with the highest or lowest power of $A$. In all,
\begin{align*}
\text{\text{span}}\langle K \rangle& = M-m\leq \left( \sum^{c_n}_{i=1} h_{i} -\sum^{c_n}_{i=1} l_{i}\right)+2(|s_{max}|-1)+2(|s_{min}|-1)\\
&=\underbrace{\left(\sum^{c_n}_{i=1} w(R_{i})\right)}_{\text{Contribution from width}}+\underbrace{2(|s_{max}|+|s_{min}|-2)}_{\text{Contribution from components}}.\end{align*}

We  investigate the contribution from width and the contribution from the number of components to obtain a bound on the span of the bracket polynomial. In Section \ref{Floor Proof} we prove Theorem \ref{Upper Bound}, which states that for $n\geq 2$, the maximum width of an $n$-skein relation is $\lfloor\frac{n^{2}}{2}\rfloor$. We further prove that this bound is realized for all $n \geq 2$. In Section \ref{components} we  prove Theorem \ref{componentsthm}, which states that for $n \geq 3$, given any $c_{n}$ $n$-crossing link diagram with an $s_{max}$ state and an $s_{min}$ state, $|s_{max}|+|s_{min}|\leq (2n-4)c_{n}+2$. In Section \ref{together} we prove the Main Theorem:  $\text{span}\langle K \rangle \leq (\left\lfloor\frac{n^2}{2}\right\rfloor+4n-8)c_{n}$ for any $c_{n}$ $n$-crossing link with $n\geq 3$. Further, we explore how we can obtain a tighter bound given specific conditions on $K$. In Section \ref{spectrum}, we define the crossing spectrum of a knot $K$ to be the sequence of $n$-crossing numbers of $K$. We then explore the crossing spectrum as a knot invariant and evaluate the relationship between multi-crossing numbers for the same knot as we vary $n$.  In Section \ref{additivity} we explore the additivity of $n$-crossing knots and links under composition. Finally, in Section \ref{chart} we illustrate and explain the methods to calculate the first extensive list of $n$-crossing numbers for many prime and composite knots.

\section{An upper bound on the width of an $n$-skein relation}\label{Floor Proof}

As described in \citep{Quadruple}, a \textit{parallel split} in an n-skein relation is a splitting that consists of n parallel arcs. The top row in Figure 1 contains two parallel splits.

In this section, we  prove the upper bound on the width of an $n$-skein relation is $\lfloor \frac{n^2}{2} \rfloor$. To do so we must understand how width changes when we construct an $(n+1)$-skein relation from an $n$-skein relation by adding an overstrand.

The width of the skein relation  could increase by more than $$\left\lfloor \frac{(n+1)^2}{2} \right\rfloor- \left\lfloor \frac{n^2}{2} \right\rfloor= \begin{cases} n \textrm{  when $n$ is even} \\n+1 \textrm{  when $n$ is odd}\end{cases}$$ if there exist two terms $S$ and $T$ that meet the following criteria:

\begin{enumerate}
\item The number of intersections of the overstrand with $S$ and the overstrand with $T$ is greater than  \\
$\left\lfloor \frac{(n+1)^2}{2} \right\rfloor- \left\lfloor \frac{n^2}{2} \right\rfloor$.
If we resolve the intersections between the overstrand and the term with a higher exponent of $A$ as A-splits and resolve the intersections between the overstrand and the term with a lower exponent of A as B-splits, the difference in powers of the new terms increases by the total number of intersections of the overstrand with the two terms $S$ and $T$.  
\item The exponents of the coefficients of the terms are already far apart. The ability to separate the exponents of $A$ attached to two terms by a large amount is not worrisome if the exponents of those terms are already very close together. 
\end{enumerate}

The essence of the proof is that no pair of terms in a single skein relation will possess both of these properties. If a single overstrand can intersect many strands in the splits of a pair of terms, then the difference in the exponents of $A$ between the two terms is sufficiently smaller than the upper bound of $\left\lfloor \frac{n^2}{2} \right\rfloor$. \\

To rigorously prove this fact, we need to consider how the difference in power of the two terms relates to the maximum number of intersections an overstrand can create. Lemma \ref{biglem} provides an explicit relationship between these values. The width theorem follows almost directly from Lemma \ref{biglem}. We will first introduce notation and definitions that allow us to discuss the the process of adding an overstrand and increasing the width of the subsequent skein relation. We then prove Lemmas \ref{first}- \ref{istar}, which are all necessary to prove Lemma \ref{biglem}. 

\medskip

\begin{notation} Let $T$ be a term in an $n$-skein relation $R$  corresponding to a crossing $x$. Let $O$ be an overstrand placed over $x$ to create a crossing $x'$ with $(n+1)$-skein relation $R'$. Let $T'$ be a term in $R'$ obtained by resolving the intersections of $T$ with $O$ according to a particular combination of A-splits and B-splits. Let $O'$ be an overstrand placed over $x'$ to create a crossing $x''$ with  $(n+2)$-skein relation $R''$. We may also consider $O'$ as a overstrand placed on the term $T$ before the intersections of $O$ with $T$ are resolved (and in this case $O'$ will intersect $O$). Let $P(T)$ represent the power of $A$ in the coefficient of the term $T$.  Let $I(O,T)$ be the set of arcs in the split diagram of $T$ that $O$ intersects; $|I(O,T)|$ denotes its cardinality. Let $I(O,O',T)$ denote the arcs that are in the split diagram of $T$ that both $O$ and $O'$ intersect. We abbreviate its cardinality by $T^*$. When we say that we are performing a geometric operation, such as a split move or an arc surgery, on a term $T$, we mean that we are performing this operation on the split of the term $T$. The power of the term $T$ remains unchanged.
\end{notation}

Finally, we define a quantity $I^*$ that will be the key to the main theorem.

 \begin{definition}
 Given two terms $S$ and $T$ of an n-skein relation for a given crossing $x$ and two overstrands $O$ and $O'$, we define $I^*=\min\{S^*,T^*\}$. 
 \end{definition}
 
\begin{definition} \label{separated} Let $O$ be an overstrand on a split $T$.  We call the endpoints of any arc that intersects $O$ \textit{separated} endpoints.
\end{definition}

\begin{definition} 

Here we describe \textit{arc surgery}. This is a procedure for generating a new split $T_1$ from an existing split $T_0$, given a strand $O$ overlaid. The resulting split $T_1$ has the same number of arcs as $T_0$, but $T_1$ has two fewer intersections with $O$. Although arc surgery relates splits within a $n$-skein relation, arc surgery is performed with respect to an overstrand $O$. For simplicity we rotate the split so that $O$ is vertical.

\begin{enumerate} 
\item Choose an \textit{east pair} of endpoints and a \textit{west pair} of endpoints  in $T_0$ satisfying the following three conditions:

\begin{enumerate}
\item The two endpoints of the west pair are to the left of $O$.  The two endpoints of the east pair are to the right of $O$. 
\item When one considers the set of complementary regions bounded by the arcs of $T_0$ and the circle connecting all the endpoints, the endpoints of a given pair are on the boundary of the same region.
\item All four endpoints are separated endpoints.
\end{enumerate}

\item Delete all the arcs that intersect $O$. Connect the east pair of endpoints by an arc. Next, connect the west pair of endpoints by an arc.  All other arcs are left unchanged.

\item There may be endpoints with no arcs extending from them as a result of the deletion in step (2). There is one and only one way to connect these endpoints so that the newly created arcs all intersect $O$ and do not intersect each other. Connect the endpoints in this way.
\end{enumerate}
\end{definition}

 \begin{figure}[H]
        \centering
        \begin{subfigure}[b]{0.19\textwidth}
                \centering
                \includegraphics[width=\textwidth]{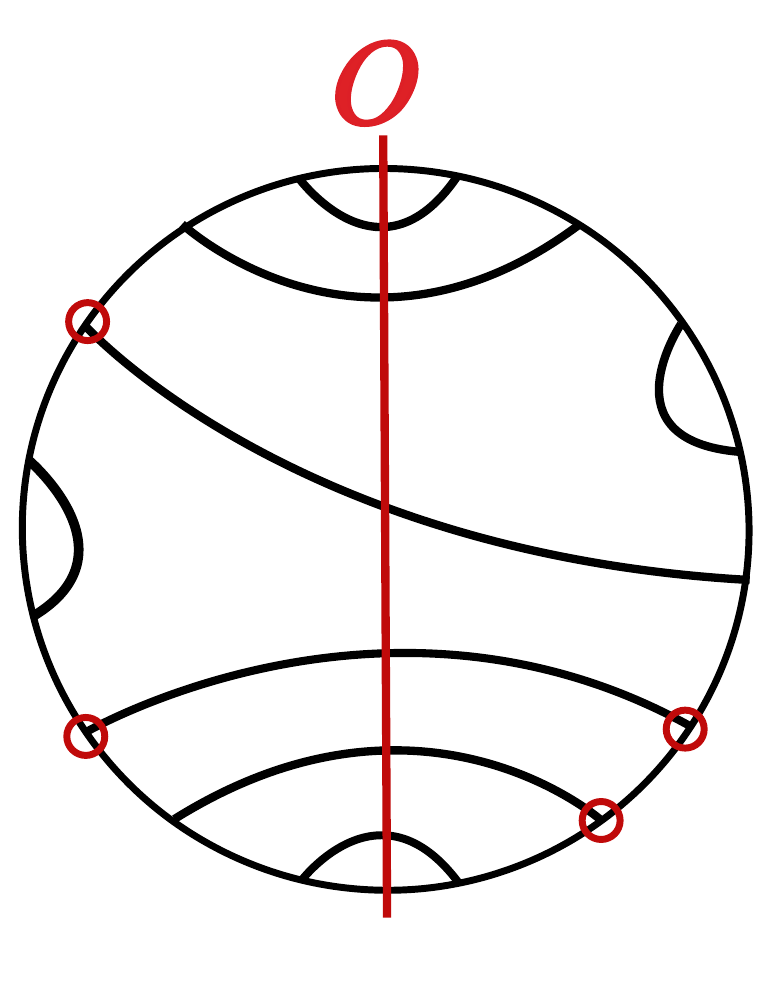}
                \caption{ }
                \label{fig:sample before}
        \end{subfigure}%
        \qquad
        ~ 
        \begin{subfigure}[b]{0.2\textwidth}
                \centering
                \includegraphics[width=\textwidth]{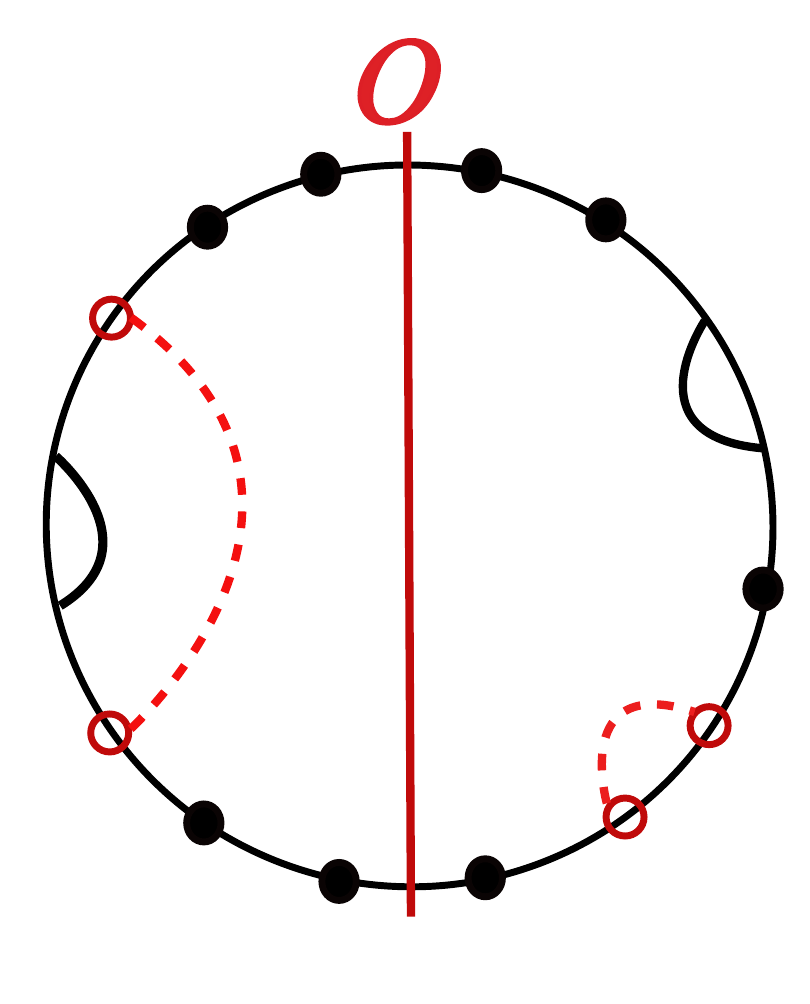}
                \caption{ }
                \label{fig:sample mid}
        \end{subfigure}
        \qquad
        ~ 
        \begin{subfigure}[b]{0.2\textwidth}
                \centering
                \includegraphics[width=\textwidth]{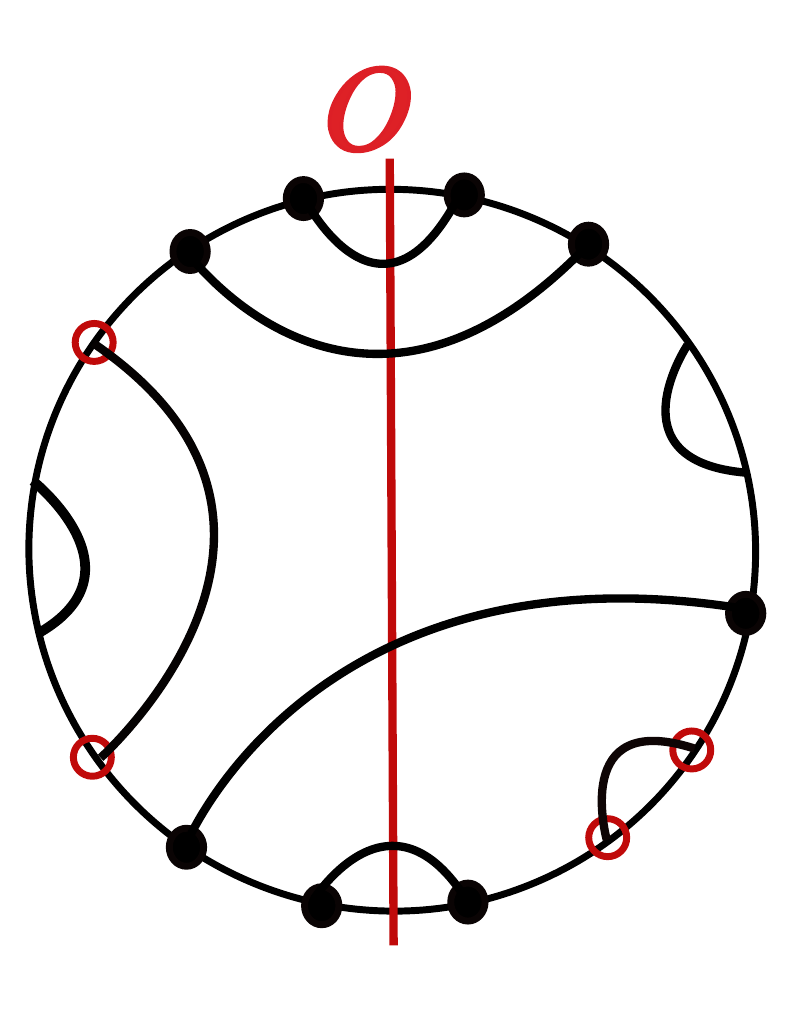}
                \caption{ }
                \label{fig:sample post}
        \end{subfigure}
        \caption{(a) Step 1: Select an east and west pair from the endpoints of $T_0$. (b) Step 2: Delete the arcs extending from the endpoints of the east and west pair and then connect east and west pairs. (c) Step 3: Connect the remaining endpoints with arcs intersecting O to create $T_1$.}\label{fig:arcsurgery}
\end{figure}

\begin{remark}
A split move is a special type of arc surgery that occurs when the west and east pairs of endpoints are all part of one polygon. An arc surgery is more powerful than a split move. Figure \ref{fig:foreman} illustrates two splits that are separated by one arc surgery, but they are separated by at least 7 split moves.  Note that an arc surgery is a directed operation, unlike a split move. Lemma \ref{before} will introduce the concept of inverse arc surgeries. 

\end{remark}

\begin{figure}[H]
        \centering
        \begin{subfigure}[b]{0.2\textwidth}
                \centering
                \includegraphics[width=\textwidth]{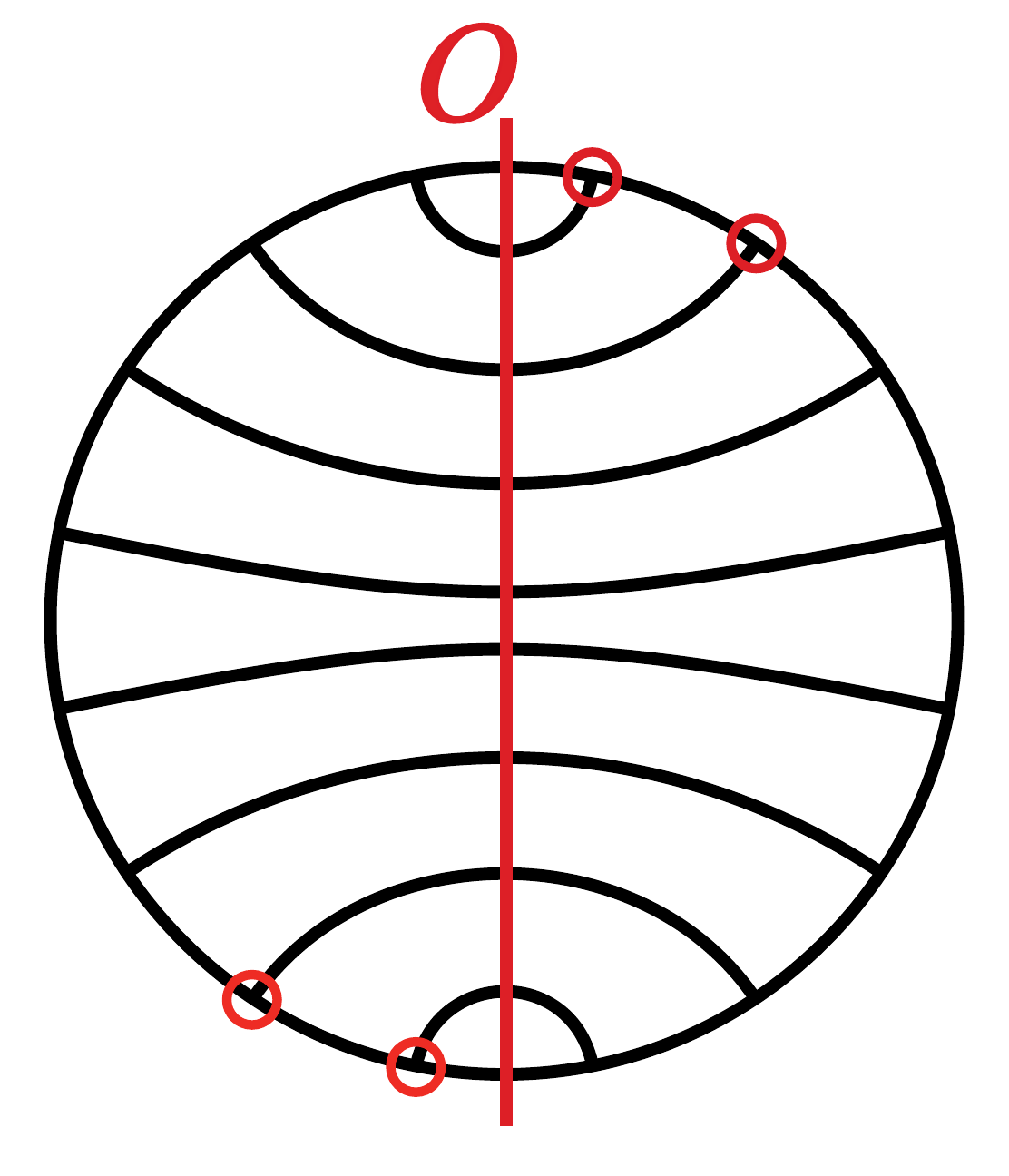}
                \caption{Before arc surgery}
                \label{fig:beforeman}
        \end{subfigure}%
        \qquad
        ~ 
        \begin{subfigure}[b]{0.2\textwidth}
                \centering
                \includegraphics[width=\textwidth]{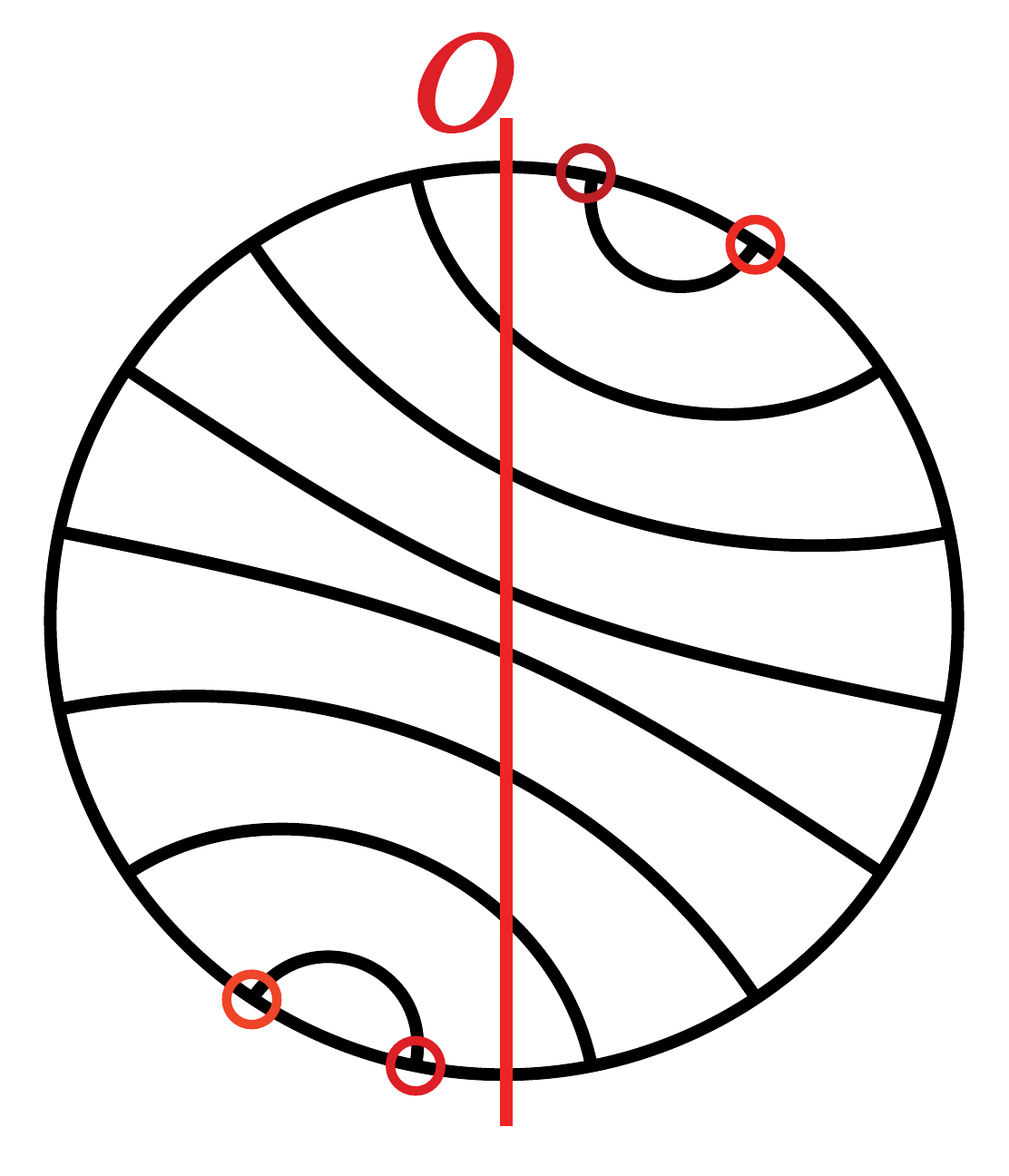}
                \caption{Post arc surgery}
                \label{fig:afterman}
        \end{subfigure}
        ~ 
        \caption{These two splits are one arc surgery apart, but 7 split moves apart.}\label{fig:foreman}
\end{figure}

\begin{definition} Let $O'$ be an overstrand on a spilt $T$. We call $r_i$ a \textit{rotation of $O'$} obtained by rotating $O'$ clockwise past $i$ endpoints of the split. The rotation of $O'$ in the counterclockwise direction will be denoted $\overline{r_i}$.  We note that the set  $\{\overline{r_{\lfloor\frac{n}{2}\rfloor}}, \overline{r_{\lfloor\frac{n}{2}\rfloor - 1}}, \overline{ r_{\lfloor\frac{n}{2}\rfloor - 2}} \ldots r_0 \ldots r_{\lceil\frac{n}{2}\rceil-1}, r_{\lceil\frac{n}{2}\rceil}\}$ contains all the unique rotations of $O'$. When $n$ is even, $\overline{r_{\left \lfloor \frac{n}{2} \right \rfloor}} = r_{\left\lceil \frac{n}{2} \right \rceil}$ and when $n$ is odd, all elements in the aforementioned set are distinct. 

 When $r_i \in \{ r_0, r_1, ... , r_{\lceil \frac{n}{2} \rceil}\}$, we say that $r_i$ is \textit{clockwise} of $O'$, and when $\overline{r_i} \in \{\overline{r_{\lfloor\frac{n}{2}\rfloor}},...,\overline{r_1} \}$, we say that $\overline{r_i}$ is \textit{counterclockwise} of $O'$. To make notation clearer, we will denote $\overline{r_{\lfloor\frac{n}{2}\rfloor}}$ as $\overline{r_{\frac{n}{2}}}$ and ${r_{\lceil\frac{n}{2}\rceil}}$ as  ${r_{\frac{n}{2}}}$.
\end{definition}

 \begin{definition}\label{cdl}Let $j$ be an integer such that $1 \leq j \leq\lceil \frac{n}{2}\rceil$. The following definitions describe the change in the number of intersections with $T$  as we rotate $O'$ clockwise from $r_{j-1}$ to $r_j$. We call a rotation from  $r_{j-1}$ to $r_j$ \textit{increasing} if the number of intersections with arcs in $T$ increases. We define a \textit{decreasing} rotation analogously. A rotation is \textit{level} if the number of intersections does not change. \\
 \begin{itemize}
\item   Let $c$ denote the number of increasing rotations, i.e., the number of distinct $j$ for which\\ $I(r_j,T) - I(r_{j-1},T) = 2$.
  \item  Let $l$ denote the number of level rotations, i.e., the number of distinct $j$ for which\\ $I(r_j,T) - I(r_{j-1},T) = 0$. 
   \item Let $d$ denote the number of decreasing rotations, i.e., the number of distinct $j$ for which \\ $I(r_j,T) - I(r_{j-1},T) = -2$. 
\end{itemize}
   
Define $\overline{c}$, $\overline{l}$, $\overline{d}$ analogously. These values represent  how the number of intersections with $T$ changes as we rotate the position of $O'$ counterclockwise from $\overline{r_{j-1}}$ to $\overline{r_{j}}$.

 \end{definition} 
 
We are interested in rotations of $O'$ because we will later enumerate the possible positions of $O$ as rotations of $O'$. For the next lemma, remember that $T'$ is a term in the $(n+1)$-skein relation obtained by resolving all intersections of $O$ with $T$ in some manner.
 
 \begin{lemma} \label{first} 
 Given a term T and two overstrands $O$ and $O'$, \[|I(O',T')|  \leq |I(O',T) |+ 1,\] with equality if and only if exactly one of the following holds. 


 \begin{enumerate} 
 \item $O$ does not intersect any arcs in the split diagram of $T$  
 \item If $I(O,O',T)$ is nonempty, $O$ is counterclockwise of $O'$, and all intersections of $O$ with arcs in $I(O,O',T)$ are resolved as A-splits. 
 \item If $I(O,O',T)$ is nonempty, $O$ is clockwise of $O'$, and all intersections of $O$ with arcs in $I(O,O',T)$ are resolved as B-splits.
 \end{enumerate} 

\end{lemma}

\begin{proof} 
We prove that if any of the three conditions hold then $|I(O',T')|  = |I(O',T)| + 1$. We then show how any other way of resolving the intersections of $O$ and $T$ creates strictly fewer intersections between $O'$ and $T'$. 

\medskip 

\noindent \textbf{Condition 1:} Assume $O$ does not intersect any arcs in $T$. In this case, resolving the intersections with $O$ does not affect the number of arcs that intersect $O'$. Because both $O$ and $O'$ bisect the split diagram, $O'$ will also intersect $O.$ Therefore $|I(O',T')|  = |I(O',T)| + 1.$ Note that $n$ must be even for this to occur.

\medskip 

\noindent \textbf{Condition 2:} We will ignore arcs that $O'$ intersects and $O$ does not intersect, i.e. $I(O',T) - I(O,T)$ Select east because resolving the intersections of $O$ with $T$ will not change the fact that these arcs intersect $O'$. Assume $I(O,O',T)$ is nonempty and $O$ is counterclockwise of $O'$. Resolve the intersections of $O$ with $I(O,T) - I(O',T)$. Note that we are not yet resolving the intersections of $O$ with $I(O,O',T)$. This operation will create a strand, $E$, that has one endpoint on each side of $O'$ and intersects exactly the arcs $I(O,O',T)$. At this point, the number of intersections with $O'$ has not changed because we have only resolved the intersections of $O$ with arcs that $O'$ does not intersect. We now consider the crossings of $E$ with $I(O,O',T)$, since $O$ has ``become" $E$. Resolving $k$ crossings with any overstrand requires the re-pairing of $2(k+1)$ endpoints. These are the $k$ endpoints on each side of the overstrand, plus the two endpoints of the overstrand. Here we are resolving $T^*$ crossings with $E$ and changing arcs extending from $T^*+1$ endpoints on each side of $O'$. One endpoint on each side comes from $E$; the others were endpoints in $I(O,O',T)$. Because this is a split diagram, these arcs may intersect $O'$, but not each other.  There is one and only one way to create $|I(O',T)| + 1$ intersections with $O'$. See Figure \ref{fig:lemma1}. This unique splitting is the ``homogeneous'' splitting composed of all A-splits, because we want to connect all the ``A-regions'' in the pre-resolved diagram in order to maximize intersections with $O'$. Any other splitting of the intersections of $E$ with arcs in $I(O,O',T)$ creates at least one arc between two of the $T^*+1$ endpoints on the same side of $O'$. This means there are at most $|I(O',T)| - 1$ intersections between $O'$ and $T'$.  

\medskip 

\noindent \textbf{Condition 3:} 
This follows similarly to the proof of Condition 2. 

\end{proof}

 \begin{figure}[H]
        \centering
        \begin{subfigure}[b]{0.15\textwidth}
                \centering
                \includegraphics[width=\textwidth]{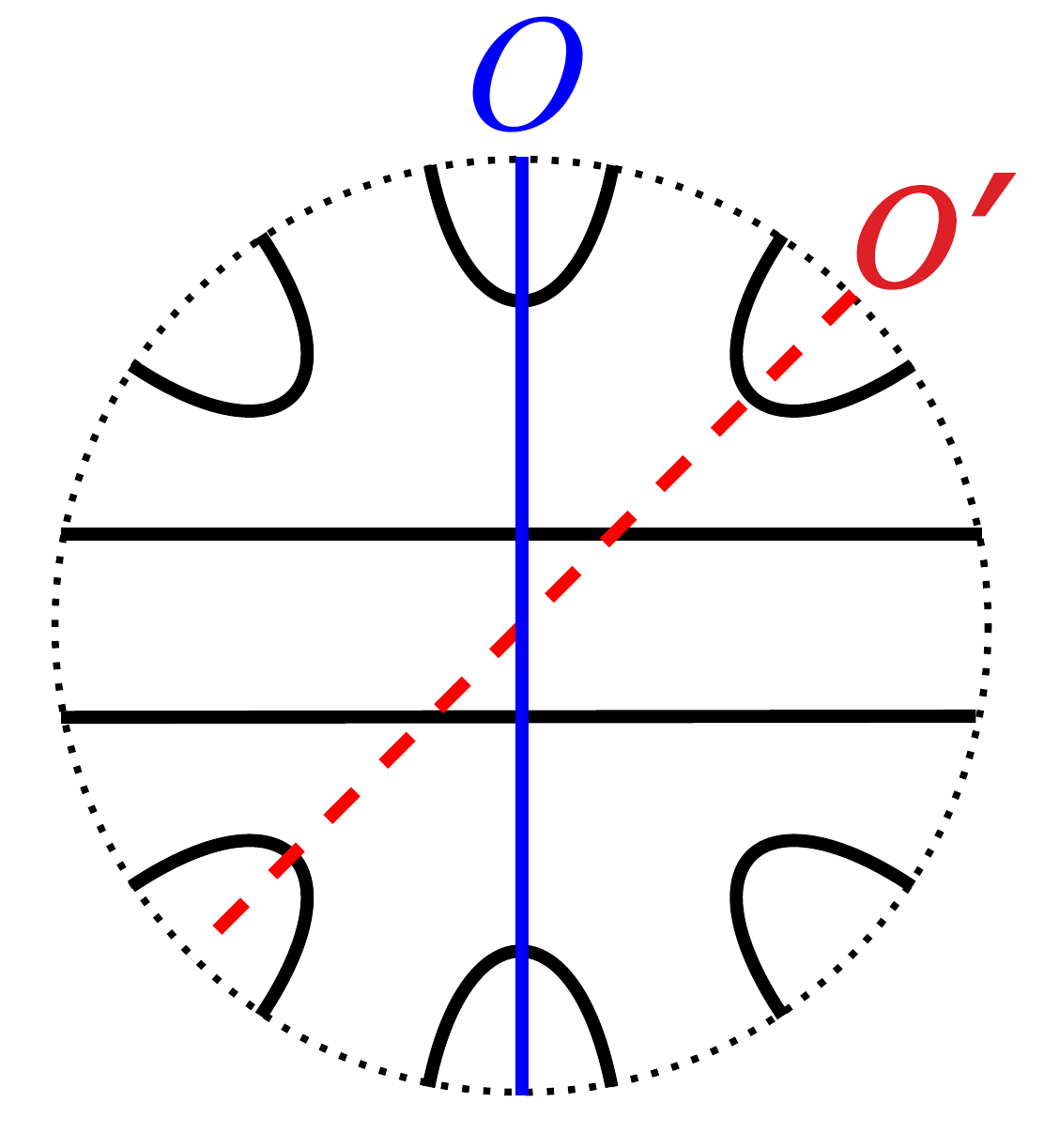}
                \caption{ }
                \label{fig:pre}
        \end{subfigure}%
        ~ 
        \begin{subfigure}[b]{0.15\textwidth}
                \centering
                \includegraphics[width=\textwidth]{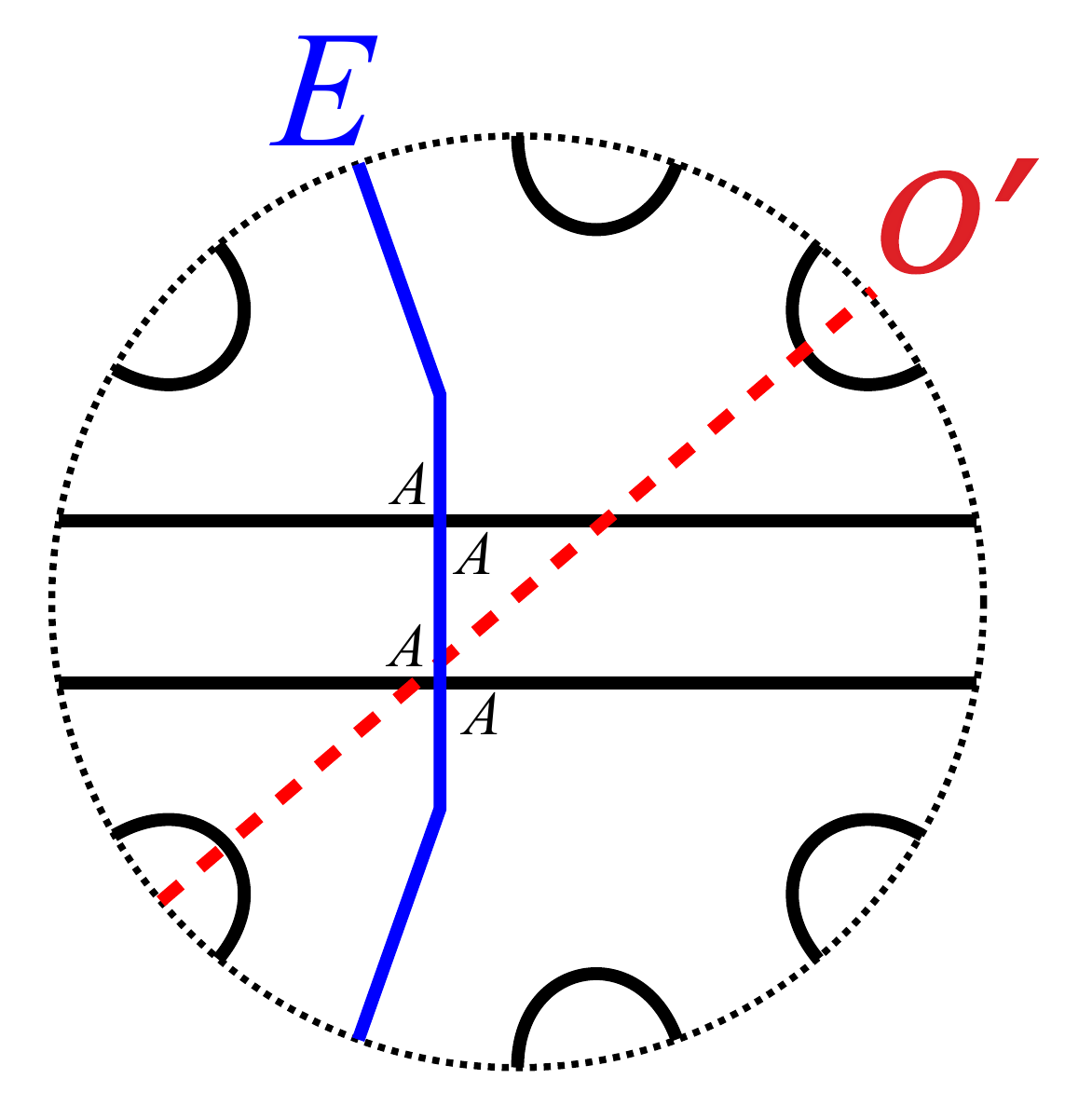}
                \caption{ }
                \label{fig:mid}
        \end{subfigure}
        ~ 
           \begin{subfigure}[b]{0.15\textwidth}
                \centering
                \includegraphics[width=\textwidth]{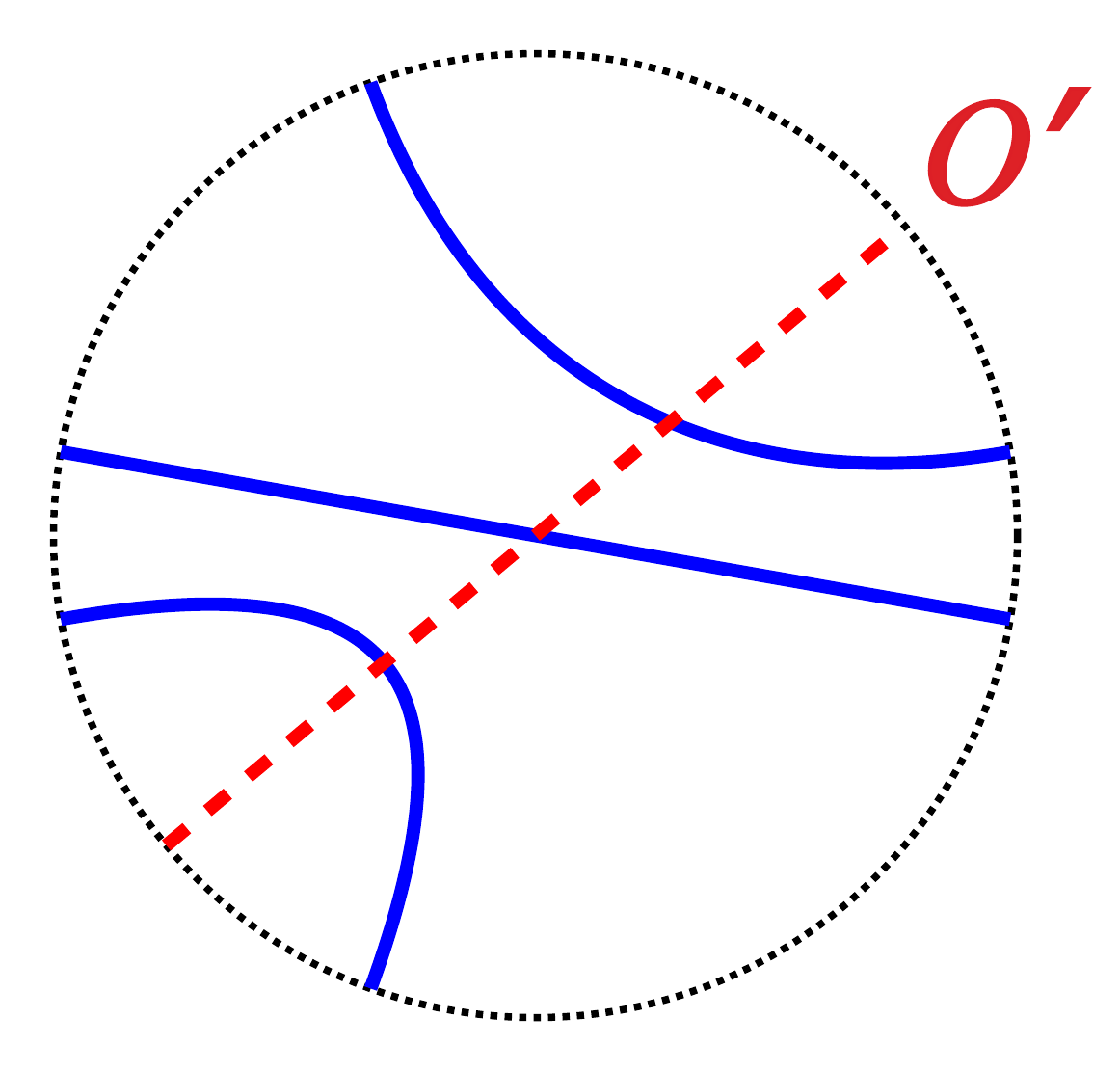}
                \caption{ }
                \label{fig:post}
        \end{subfigure}
        \begin{subfigure}[b]{0.15\textwidth}
                \centering
                \includegraphics[width=\textwidth]{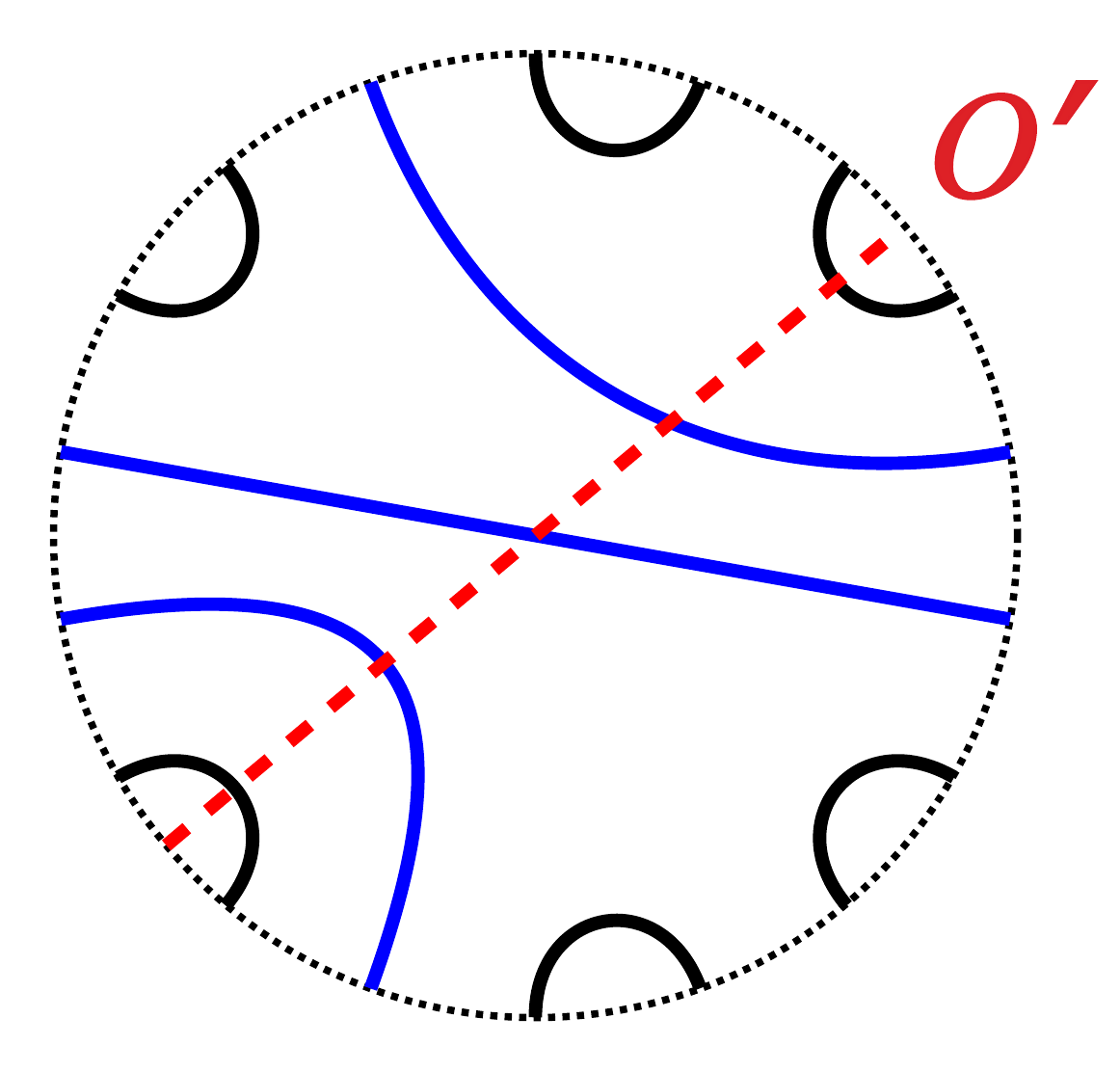}
                \caption{ }
                \label{fig:bigpost}
        \end{subfigure}
        \caption{ (a) The split diagram with $O$ and $O'$ before any crossings have been resolved. (b) Resolve the intersections of $O$ with $I(O,O',T)$ to create $E$ (c) The all $A$ split resolution of $E$  maximizes value of $|I(O',T')|$ by pairing all $2(T^*+1)$ endpoints so all arcs intersect $O'$. (d) This resolution in context of the entire split diagram.}\label{fig:lemma1}
\end{figure}

\begin{lemma}\label{before} Given a split $T_1$ in an $n$-skein relation and an overstrand $O$ such that\\ $|I(O, T_1)|=n-2m$ for any integer $m$ satisfying $1\leq m\leq \frac{n}{2}$, there exists a split $T_0$ in the $n$-skein relation such that one arc surgery performed on $T_0$ yields $T_1$ and $|I(O, T_0)|=n-2(m-1)=n-2m+2$.
\end{lemma}

\[
\xymatrix{\ar @{} [dr] |{}
T_0 \ar[d]_{}^{\text{1 arc surgery}} & |I(O,T_0)| = n-2m+2 \\ 
T_1  & |I(O,T_1)| = n-2m	}
\]

\begin{proof} We present a method of constructing such a split $T_0$. We can think of this as inverse arc surgery (see Figure \ref{fig:lemma2}).  Again assume that $O$ is vertical. First, delete an arc of $T_1$ contained completely on the east (right-hand) side of $O$ such that there are no arcs between this arc and $O$. Delete another arc of $T_1$ contained completely on the left-hand side of $O$ such that there are no arcs between this arc and $O$. Next, delete all other arcs that cross $O$. Reconnect all unpaired endpoints so that each of the new arcs intersects $O$, but none of the other newly-created arcs. This is $T_0$.  Note that we can perform one arc surgery on $T_0$ to obtain $T_1$. Choose the endpoints corresponding to the arcs previously deleted from $T_1$ as the east pair and the west pair. Performing the arc surgery with this choice of east pair and west pair creates $T_1$. Furthermore, recall that a single arc surgery decreases the number of intersections with $O$ by 2. This gives $|I(O,T_0)|=n-2m+2$. If there are multiple arcs that satisfy the conditions for deletion in the first step, there will be many possible splits $T_0$, all of which are one arc surgery away from $T_1$. 

 \begin{figure}[H]
        \centering
        \begin{subfigure}[b]{0.15\textwidth}
                \centering
                \includegraphics[width=\textwidth]{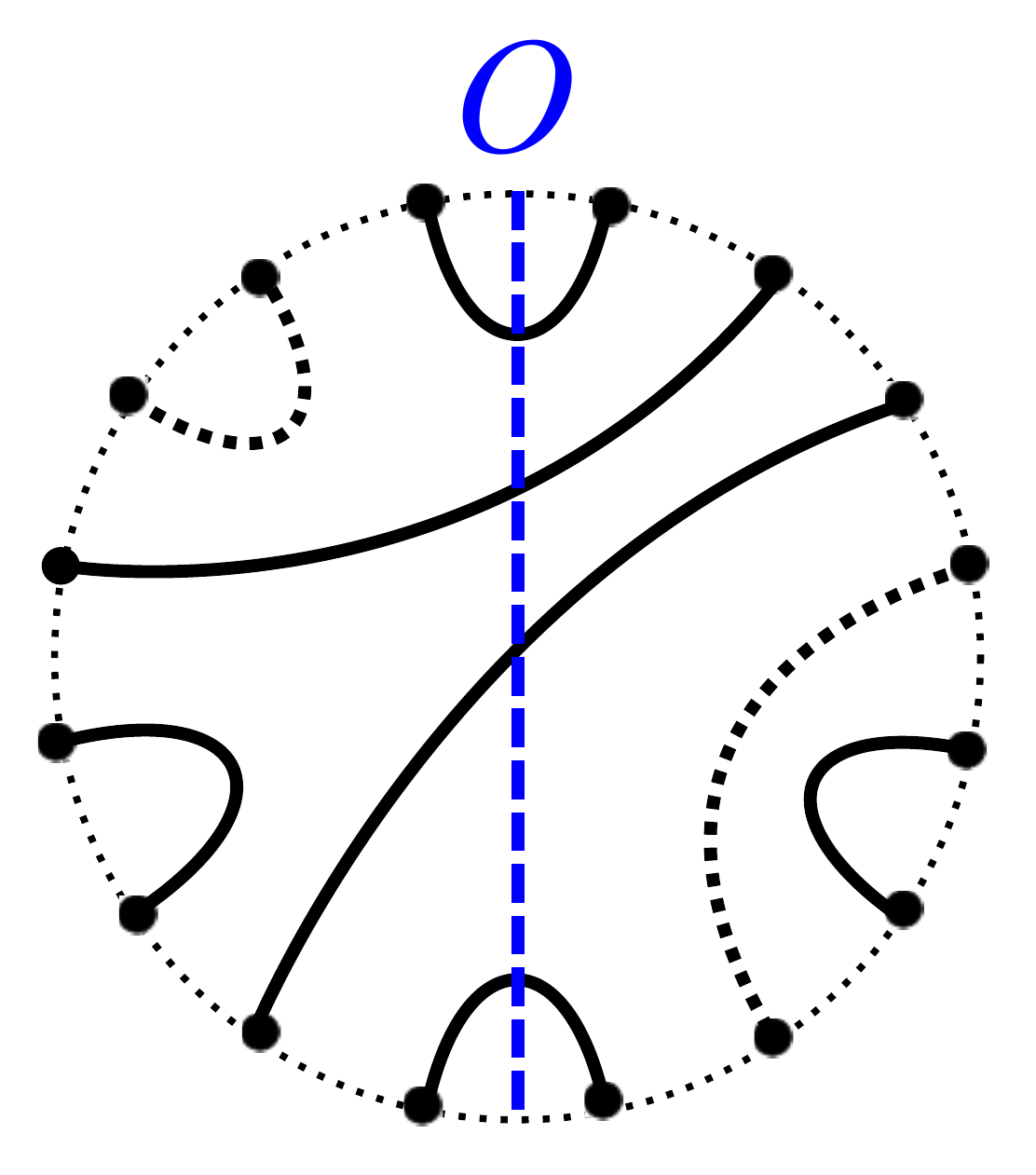}
                \caption{ }
                \label{fig:T}
        \end{subfigure}%
        ~ 
        \begin{subfigure}[b]{0.15\textwidth}
                \centering
                \includegraphics[width=\textwidth]{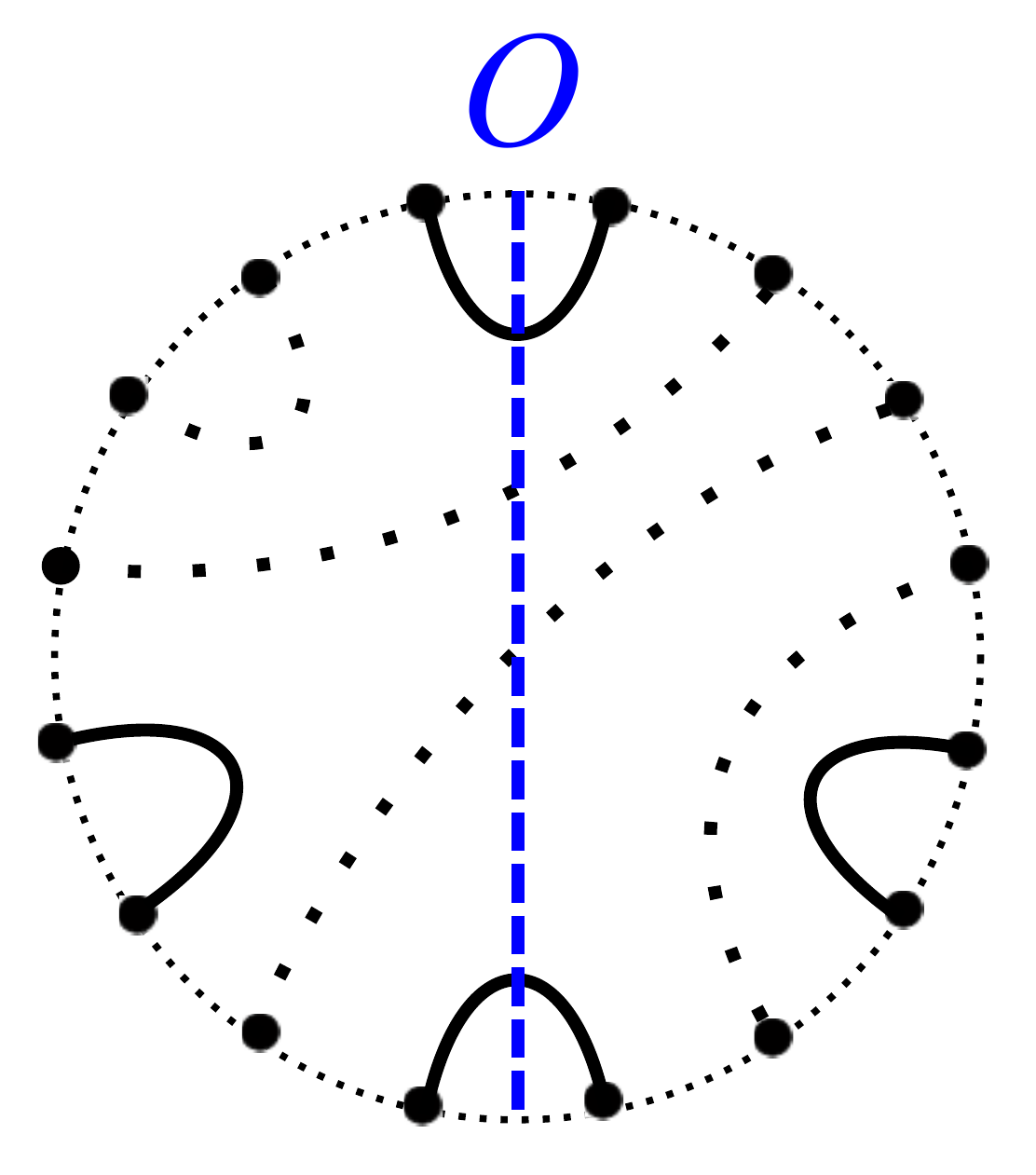}
                \caption{ }
                \label{fig:transition}
        \end{subfigure}
        ~ 
        \begin{subfigure}[b]{0.15\textwidth}
                \centering
                \includegraphics[width=\textwidth]{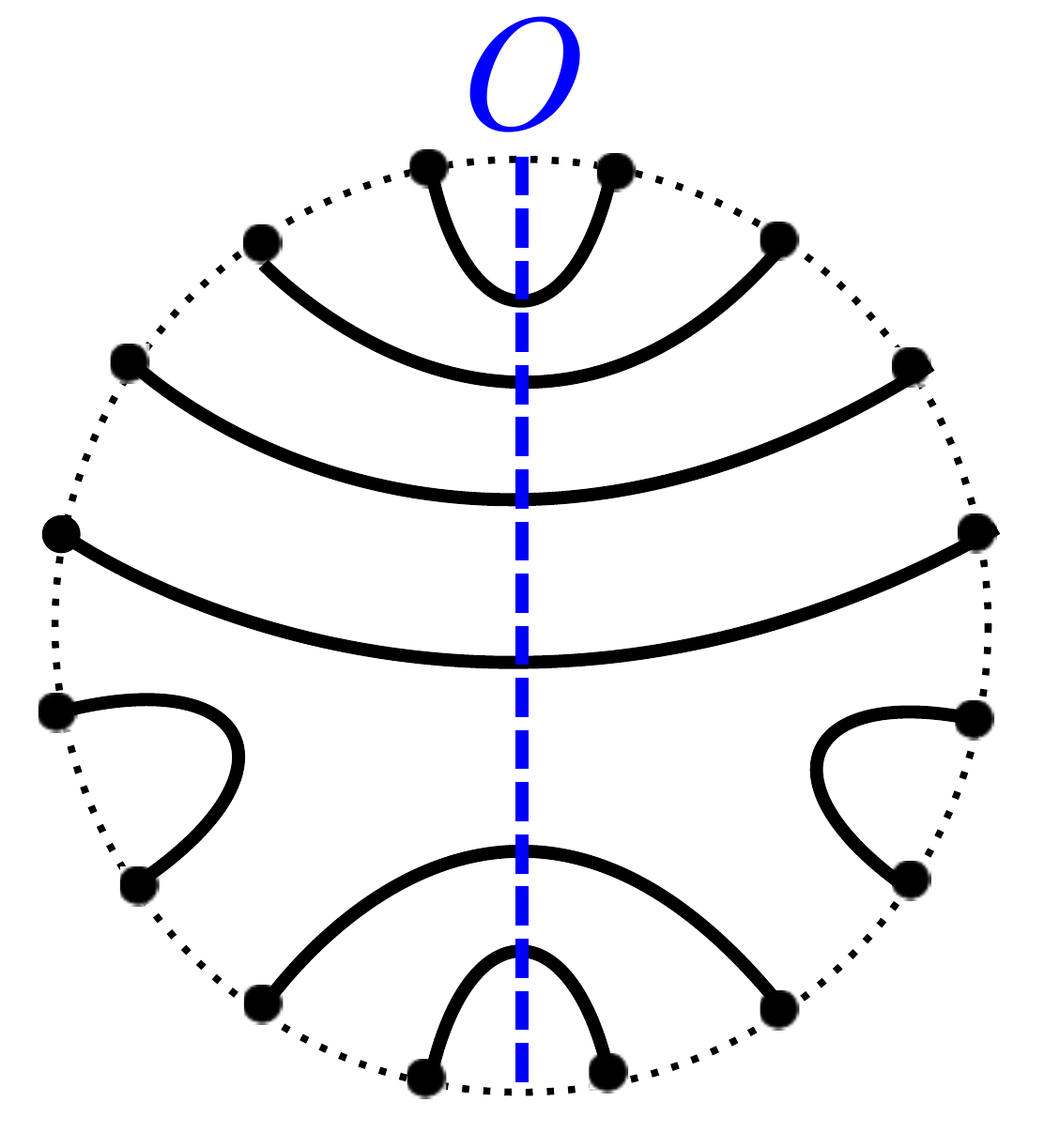}
                \caption{ }
                \label{fig:$T_0$}
        \end{subfigure}
        \caption{(a) Select east and west arcs of $T_1$ for `inverse' arc surgery.  (b) The process of `inverse' arc surgery. (c) One possible split $T_0$.}\label{fig:lemma2}
\end{figure}

\end{proof}

Recall the definitions of $c, l, \overline{c} $ and $\overline{l}$.

 \begin{lemma} \label{rotations}
  Let $T$ be a split in an $n$-skein relation and let $m$ be an integer such that $0\leq m\leq \frac{n}{2}$.   If $|I(O,T)| = n - 2m$, then $2c + l \leq 2m$ and $2\overline{c} +  \overline{l} \leq 2m$ . 
 \end{lemma}

 \begin{proof} 
We  prove that the inequality holds for $c$ and $l$. The case for $\overline{c}$ and $\overline{l}$ follows similarly. We proceed by induction on $m$. We increase $m$ by performing arc surgeries. Throughout the course of this argument, we  assume $n$ is even. The arguments presented below generalize to the odd case. Note that for the duration of this proof, the position of $O$ is fixed.\\ 


\textbf{Base Case:} Suppose $|I(O,T)| = n - 2(0)$. If $O$ intersects $n$ arcs of a split $T$ of an $n$-skein relation, then $T$ must be a parallel split and $O$ must be the strand that intersects all arcs of $T$. Each rotation of $O$ clockwise up to the position $r_\frac{n}{2}$ intersects $2$ fewer strands than the previous position.  Therefore, $c= l = 0$ and $2c + l =0 \leq 2(0)$, as desired. \\ 

\textbf{Induction:} Assume $2c + l \leq 2m$ for all $T$ such that $|I(O,T)| = n - 2m$. Let $T_1$ be a new split such that $|I(O,T_1)| = n - 2(m+1)$. Let $c_1$ and $l_1$ denote the quantities $c$ and $l$ for $T_1$. We will show that $2c_1 + l_1 \leq 2(m+1)$. 

By Lemma \ref{before}, for any $T_1$ such that $|I(O,T_1)| = n - 2(m+1)$, there exists a split $T_0$ such that $|I(O,T_0)|=n-2m$ and $T_0$ is one arc surgery away from $T_1$.  Given the splits $T_1$ and $T_0$, draw $O$ vertically and draw a strand at the position $r_{\frac{n}{2}}$.  Note that $O$ and $r_{\frac{n}{2}}$ divide the split into quadrants. Label the quadrants NE, NW, SE, SW according to the intercardinal directions. The quadrants are the same size when $n$ is even because $r_{\frac{n}{2}}$ is orthogonal to $O$ when $n$ is even. In the odd case, the NE and SW quadrants will each contain one more endpoint than the SE and NW quadrants. We call an endpoint in the NE quadrant a NE endpoint and similarly for the other directions. 

We show that $|I(r_\frac{n}{2},T_1)|  \leq |I(r_\frac{n}{2},T_0)|+2$, and then use this statement to complete the inductive step. The intersections of $r_\frac{n}{2}$ with $T_1$ are of two forms: side arcs and diagonal arcs of $T_1$. A side arc connects a NE endpoint to a SE endpoint or a NW endpoint to a SW endpoint. A diagonal arc connects a NE endpoint to a SW endpoint or a NW endpoint to a SE endpoint. Let $SA(T)$ and $DA(T)$  represent the number of side arcs and diagonal arcs in the split $T$ respectively. Then $|I(r_\frac{n}{2},T)| =SA(T)+DA(T)$. 

The number of diagonal arcs depends on the difference in the number of separated endpoints in the NW and NE quadrants. Recall from definition~\ref{separated} that separated endpoints belong to arcs that intersect $O$. Every separated endpoint in the NE quadrant that cannot connect to an endpoint in the NW quadrant must connect to an endpoint in the SW quadrant and form a diagonal arc. Thus, $DA(T)$ is the positive difference in the number of separated endpoints in the in NW and NE quadrants for the split $T$, 
\[DA(T)=|\#\text{ of separated NE endpoints }  -\#\text{ of separated NW endpoints }|.\]
We will examine how one arc surgery affects $SA(T_0)+DA(T_0)$. Note that before surgery, the east and west pair endpoints are separated endpoints. After surgery they are no longer separated endpoints. Recall that arc surgery can only be performed on separated endpoints. Arc surgery can therefore create, but not remove, side arcs because side arcs do not intersect $O$. Without loss of generality, we need only to consider four cases for the location of the east pair and the west pair. 

\medskip

\noindent \textbf{Case 1:} Both west pair endpoints and both east pair endpoints are in the NW and NE quadrants respectively. 

Arc surgery on $T_0$ decreases the number of separated NW endpoints by $2$ and the number of separated NE endpoints by $2$. Thus,  $DA(T_0)=DA(T_1)$. Note that no new side arcs are created in the surgery so $SA(T_0)=SA(T_1)$. We obtain, $SA(T_1)+DA(T_1)=SA(T_0)+DA(T_0)$. 

\medskip

\noindent \textbf{Case 2:}  One endpoint of the west pair is in the NW quadrant and the other is in the SW quadrant. Both endpoints of the east pair are in the NE quadrant.

Arc surgery on $T_0$ decreases the number of separated NW endpoints by $1$ and the number of separated NE endpoints by $2$. Thus,  $DA(T_1) = DA(T_0) \pm 1$. The arc surgery creates one new side arc so $SA(T_1)= SA(T_0) +1$. We obtain $SA(T_1)+DA(T_1)\leq SA(T_0)+DA(T_0)+2$.  

\medskip

\noindent \textbf{Case 3:} One endpoint of the west pair is in the NW quadrant and the other is in the SW quadrant. One endpoint of the east pair is in the NE quadrant and the other is in the SE quadrant.

Arc surgery on $T_0$ decreases the number of separated NW endpoints by $1$ and the number of separated NE endpoints by $1$. Thus,  $DA(T_1)= DA(T_0)$. The arc surgery creates two new side arcs so $SA(T_1)= SA(T_0) +2$. We obtain, $SA(T_1)+DA(T_1)\leq SA(T_0)+DA(T_0)+2$. 

\medskip

\noindent \textbf{Case 4:} Both west pair endpoints and both east pair endpoints are in the NW and SE quadrants respectively. 

Arc surgery on $T_0$ decreases the number of separated NW endpoints by $2$ and does not change the number of separated NE endpoints. Thus, $DA(T_1)= DA(T_0)\pm 2$. The arc surgery does not create new side arcs so $SA(T_1)= SA(T_0) $. We obtain, $SA(T_1)+DA(T_1)\leq SA(T_0)+DA(T_0)+2$. 

\medskip

Recall  $SA(T_0)+DA(T_0)=|I(r_\frac{n}{2},T)|$. Since in each of the four cases, $SA(T_1)+DA(T_1)\leq SA(T_0)+DA(T_0)+2$, we have shown $|I(r_\frac{n}{2},T_1)|  - |I(r_\frac{n}{2},T_0)| \leq2$. We will now use this fact to complete the inductive step.

 \begin{figure}[H]
        \centering
        \begin{subfigure}[b]{0.22\textwidth}
                \centering
                \includegraphics[width=\textwidth]{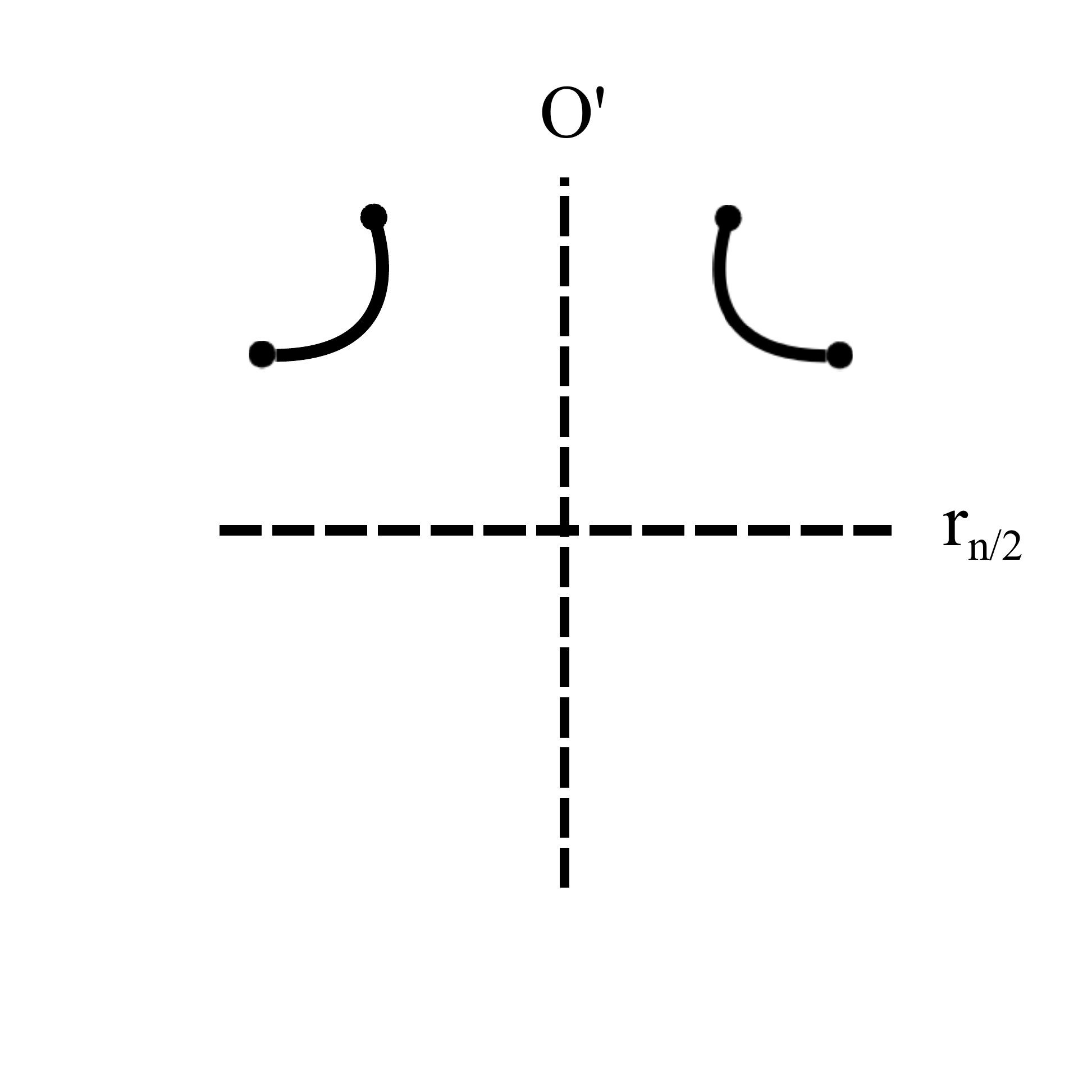}
                \caption{Case 1}
                \label{fig:case1}
        \end{subfigure}%
        ~ 
        \begin{subfigure}[b]{.22\textwidth}
                \centering
                \includegraphics[width=\textwidth]{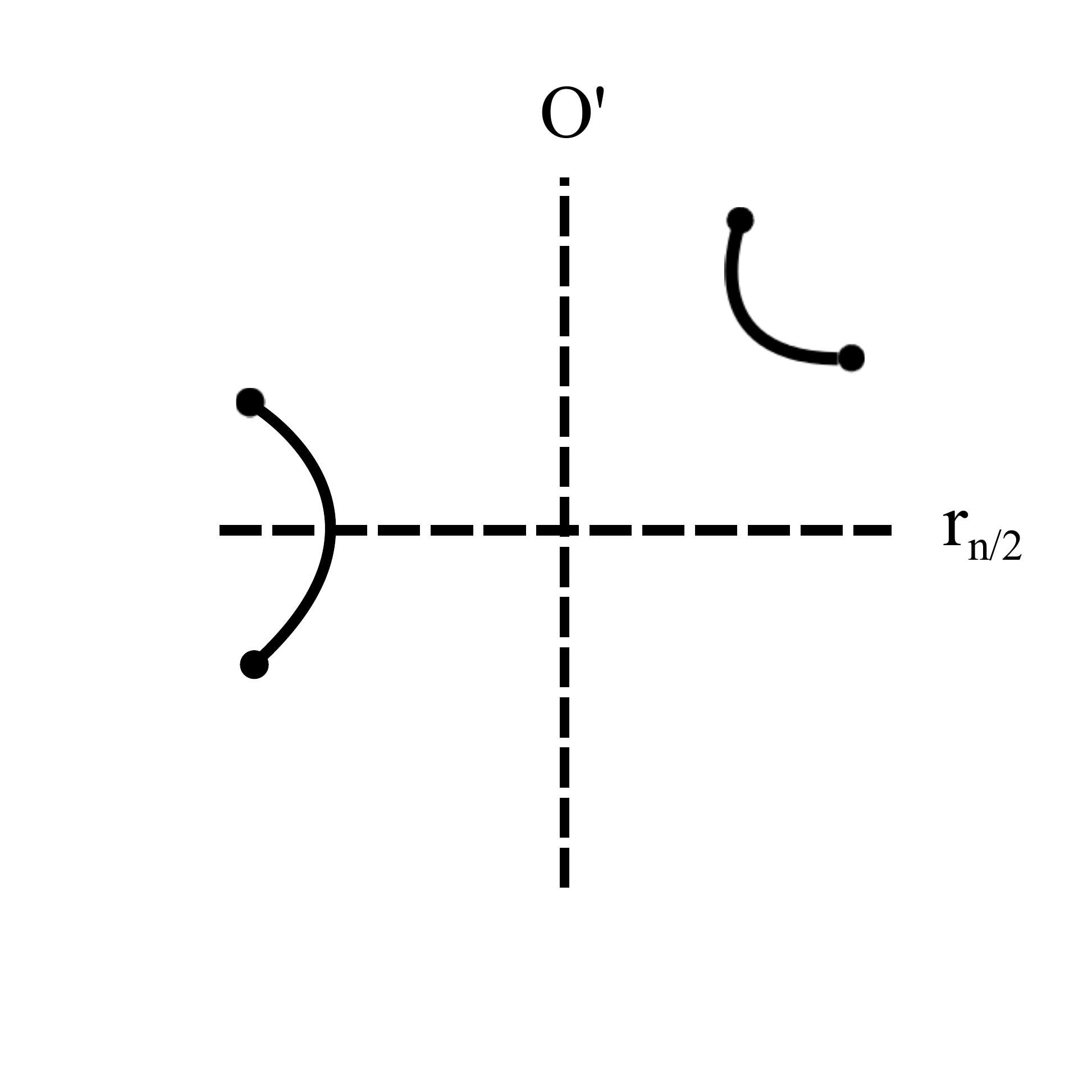}
                \caption{Case 2}
                \label{fig:case2}
        \end{subfigure}
        \begin{subfigure}[b]{0.22\textwidth}
                \centering
                \includegraphics[width=\textwidth]{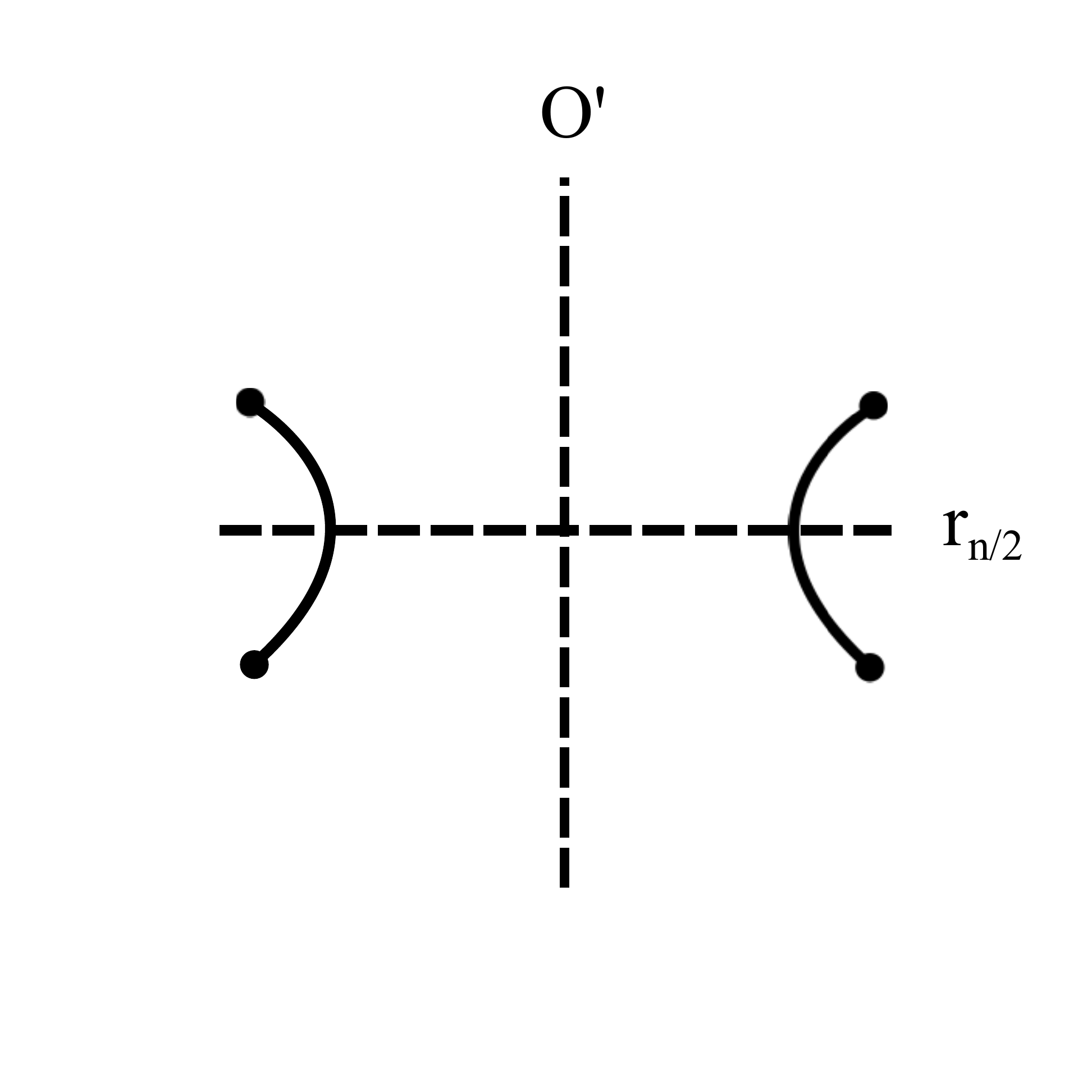}
                \caption{Case 3}
                \label{fig:case3}
        \end{subfigure}
        ~ 
        \begin{subfigure}[b]{0.22\textwidth}
                \centering
                \includegraphics[width=\textwidth]{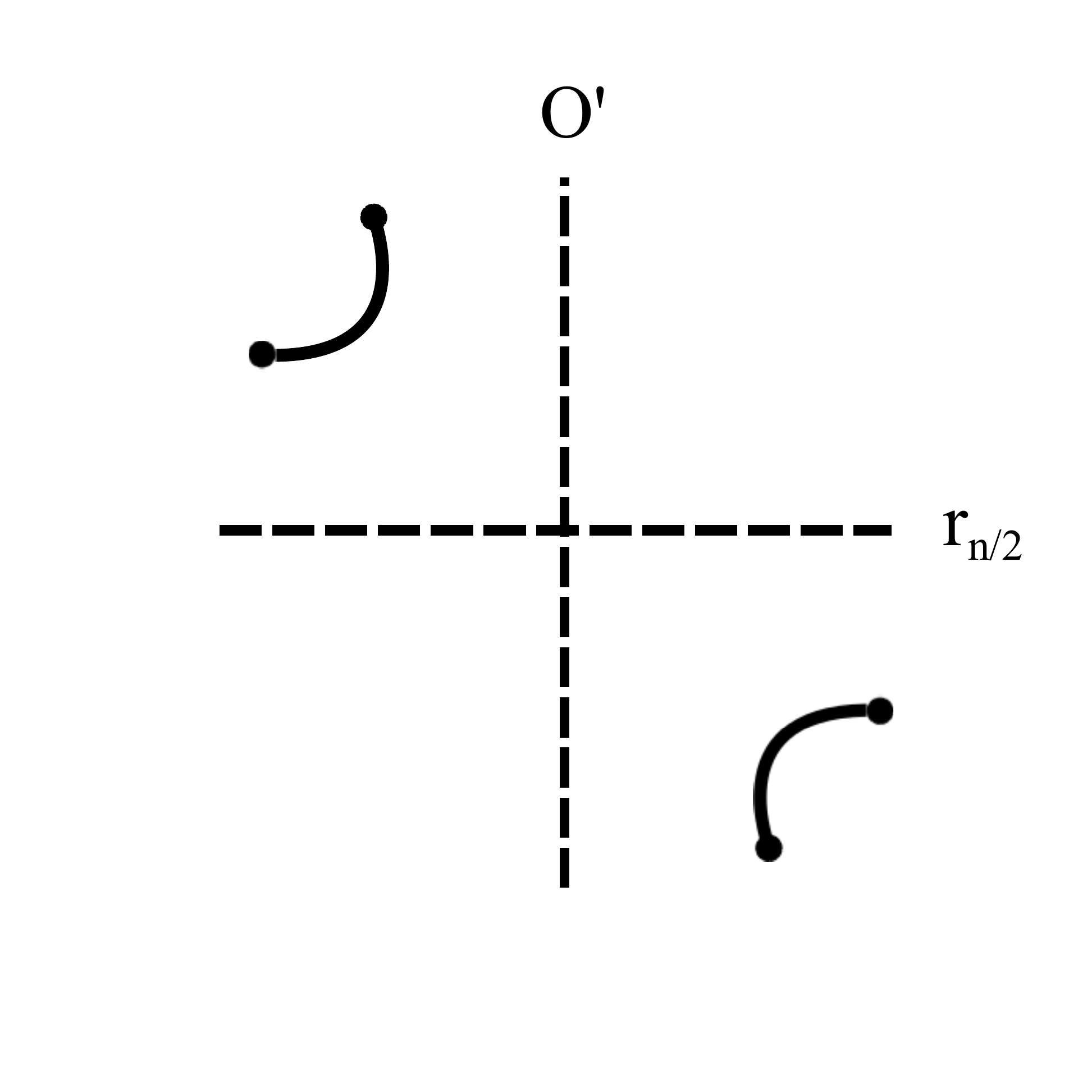}
                \caption{Case 4}
                \label{fig:case4}
        \end{subfigure}
        \caption{The different locations for the east and west pair.}\label{fig:list}
\end{figure} 

Since each increasing rotation increases the the number of intersections with $T_0$ by $2$ and each decreasing rotation decreases the number of intersections with $T_0$ by $2$, we obtain

\[|I(r_\frac{n}{2},T_1)| = |I(O,T_1)| + 2c_1 - 2d_1\]
\[|I(r_\frac{n}{2},T_0)| = |I(O,T_0)| + 2c_0 - 2d_0.\]

We substitute the above equations into the inequality $|I(r_\frac{n}{2},T_1)|  - |I(r_\frac{n}{2},T_0)|\leq 2$, and obtain

\[|I(O,T_1)| + 2c_1 - 2d_1 - (|I(O,T_0)| + 2c_0 - 2d_0) \leq 2. \]

Recall that by our choice of $T_0$, $|I(O,T_1)| - |I(O,T_0)| = -2$. So

\[2c_1 - 2d_1 - 2c_0 + 2d_0 \leq 4.\]

There must be a total of $\frac{n}{2}$ rotations to arrive at the position $r_{\frac{n}{2}}$, so

\[c_1 + l_1 + d_1 = \frac{n}{2} \ \textrm{and} \  c_0 + l_0 + d_0 = \frac{n}{2}, \]
which is equivalent to 
\[-2d_1 = 2c_1 + 2l_1 - n \ \textrm{and} \ 2d_0 = n - 2c_0 - 2l_0. \]

We obtain 
\begin{center}
\begin{align*}
  &2c_1 + (2c_1 + 2l_1 - n) - 2c_0 + (n - 2c_0 - 2l_0) \leq 4 \\
  \Rightarrow & 4(c_1 - c_0) + 2(l_1 - l_0) \leq 4 \\
  \Rightarrow &2c_1 + l_1 -( 2c_0 +l_0) \leq 2\\
   \Rightarrow & 2c_1 + l_1\leq 2+(2c_0 + l_0).\\
\end{align*}
\end{center}
By the inductive hypothesis,  $2c_0 + l_0\leq 2m$. Therefore
$2c_1 + l_1  \leq 2(m+1),$ as desired. 

\end{proof}

\begin{lemma}\label{t2 bound}
  Let $T$ be a split in an $n$-skein relation and $O'$ an overstrand. Let $m$ and $ j$ be integers such that $0\leq m \leq\frac{n}{2}$ and $m+j\leq \left\lceil \frac{n}{2} \right\rceil$.  If $|I(O',T)| = n - 2m$, then $|I(r_{m+j},T)| \leq n - 2j$.
\end{lemma}


\begin{proof}
  The overstrand $r_{m+j}$ denotes the rotation of $O'$ past $m+j$ endpoints. The case for counterclockwise rotations $r_{\overline{m+j}}$ follows similarly, except that we require $m+j\leq \left\lfloor \frac{n}{2} \right\rfloor$. Each rotation can be classified as increasing, decreasing, or level, as described previously. We  define $c^*$ as the number of increasing rotations in the first $m+j$ rotations of $O'$. More rigorously, $c^*$ is the number of distinct $p \leq m + j$ such that $|I(r_p,T)| - |I(r_{p-1}, T)| = 2$. We define $l^*$ and $d^*$ analogously. We note that $|I(r_{m+j},T)| = |I(O',T)| + 2c^* - 2d^*=n-2m+ 2c^* - 2d^*$. 
   
 We aim to find an upper bound on $|I(r_{m+j},T)|$  under the following constraints, the first following from the definitions and the second following from the definitions and Lemma \ref{rotations}:
  \begin{equation}\label{Constraint 1}c^* + l^* + d^* = m + j \end{equation}
  \begin{equation} \label{Constraint 2}2c^* + l^* \leq 2c+l \leq 2m. \end{equation}
  
  \medskip
  
Doubling both constraints and subtracting Constraint (\ref{Constraint 1})  from Constraint (\ref{Constraint 2}) yields
$$2c^* - 2d^* \leq 2m -2j.$$

Therefore $|I(r_{m+j},T)| = n-2m+ 2c^* - 2d^* \leq n-2m + 2m - 2j = n - 2j$ as desired.\\
\end{proof}

\begin{lemma}\label{istar} Let $T$ and $S$ be terms of an $n$-skein relation and let $O$ and $O'$ be overstrands. Let $p$ be an integer such that $p\leq n$. If $|I(O',S)|+|I(O',T)|=n+p$, then $|I(O,S)|+|I(O,T)|-2I^*\leq n-p$.

\end{lemma}
\begin{proof} Fix the position of $O'$ and let $|I(O',S)|=p+2m$ and $|I(O', T)|=n-2m$. We note that the value of $|I(O,S)|-2I^*$ depends on the position of $O$.  Without loss of generality, assume that $O'$ is counterclockwise of $O$ and $I^*$ is realized in the diagram of $S$, meaning that  $|I(O,O',S)| = S^* = I^*$. We may assume this because $|I(O', T)|= n-2m$ can be expressed in the form $p+2m$ for some $p$. Hence $0\leq m\leq \frac{n-p}{2}$. We also note that each  position of $O$ can be denoted as $r_i$, where $r_i$ is the rotation of $O'$ clockwise past $i$ endpoints. We observe that $r_i$ intersects 2 more, 2 fewer, or the same number of arcs as $r_{i-1}$ by rotating in and out of arcs (see Figure \ref{fig:Rotations}). To tabulate $\Delta I^*$, we must consider whether the arcs we have rotated in and out of were arcs in $I(O',S)$. Rotating past one endpoint, $O$ may intersect 1 more, 2 more, 1 fewer, 2 fewer, or the same number of arcs in $I(O',S)$ as before. 
\begin{figure}[H]
        \centering
        \begin{subfigure}[b]{0.2\textwidth}
                \centering
                \includegraphics[width=\textwidth]{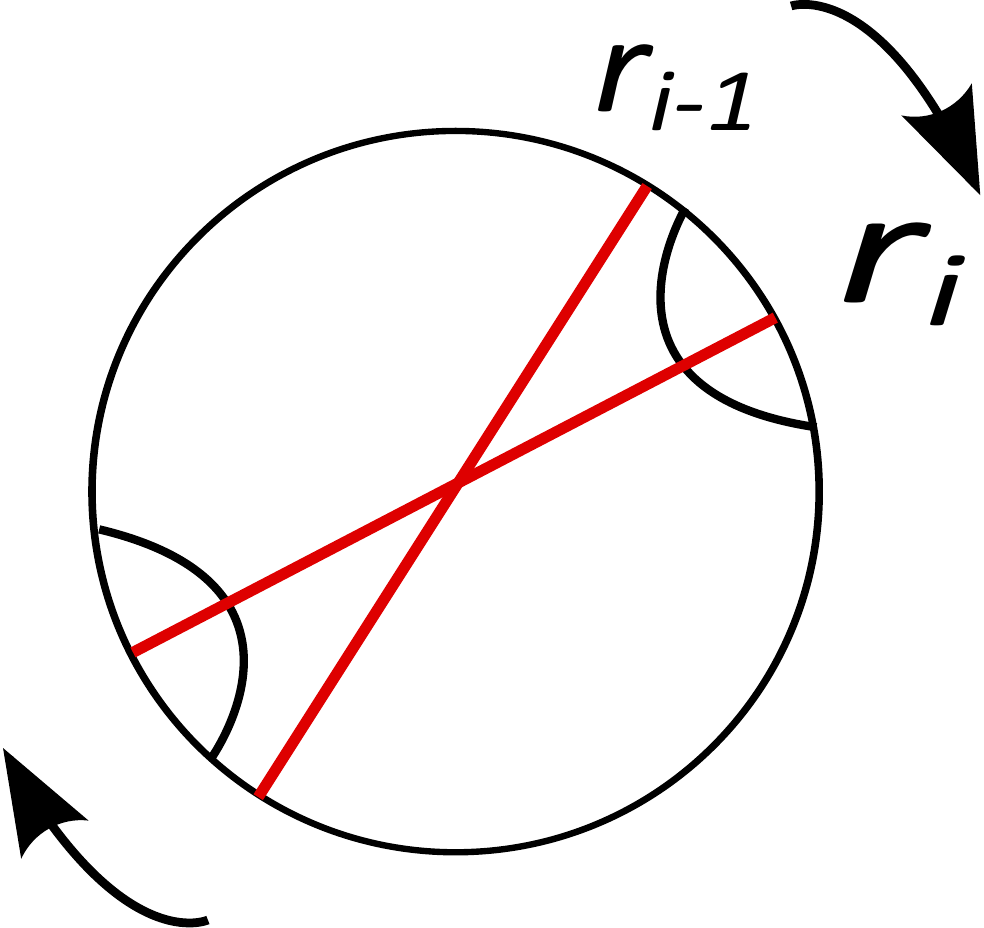}
                \caption{ }
                \label{fig:IncRot}
        \end{subfigure}
        \begin{subfigure}[b]{0.2\textwidth}
                \centering
                \includegraphics[width=\textwidth]{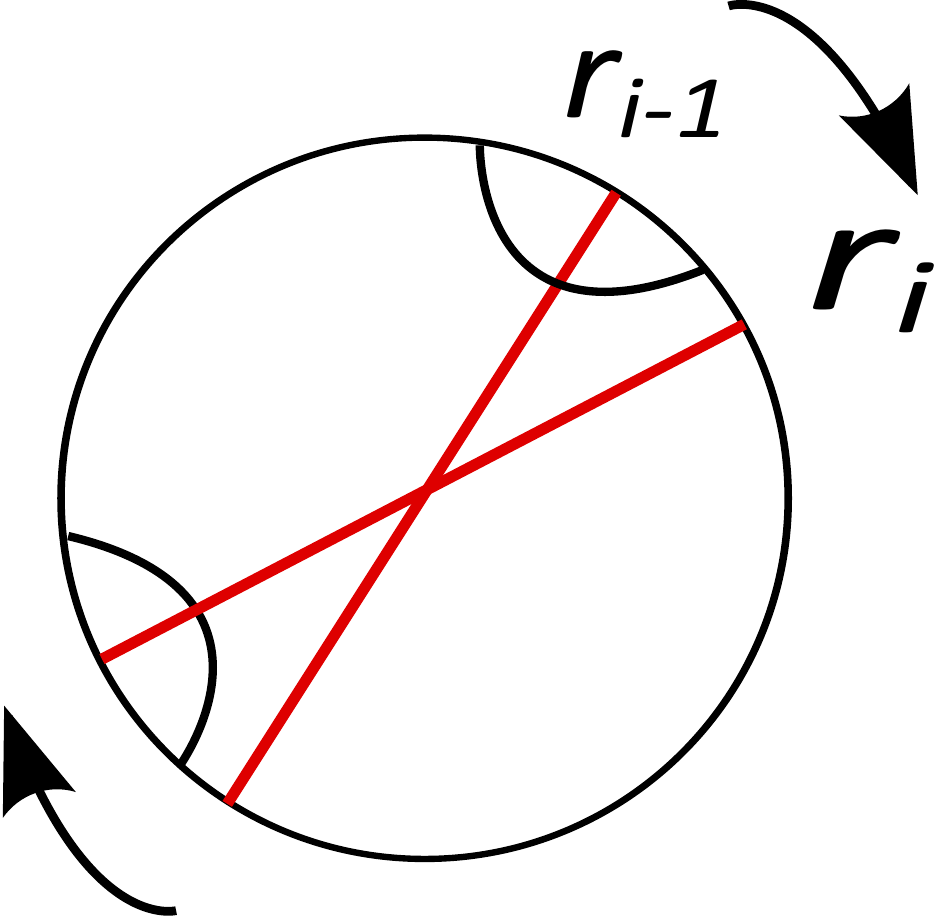}
                \caption{ }
                \label{fig:LevelRot}
        \end{subfigure}
        ~ 
        \begin{subfigure}[b]{0.2\textwidth}
                \centering
                \includegraphics[width=\textwidth]{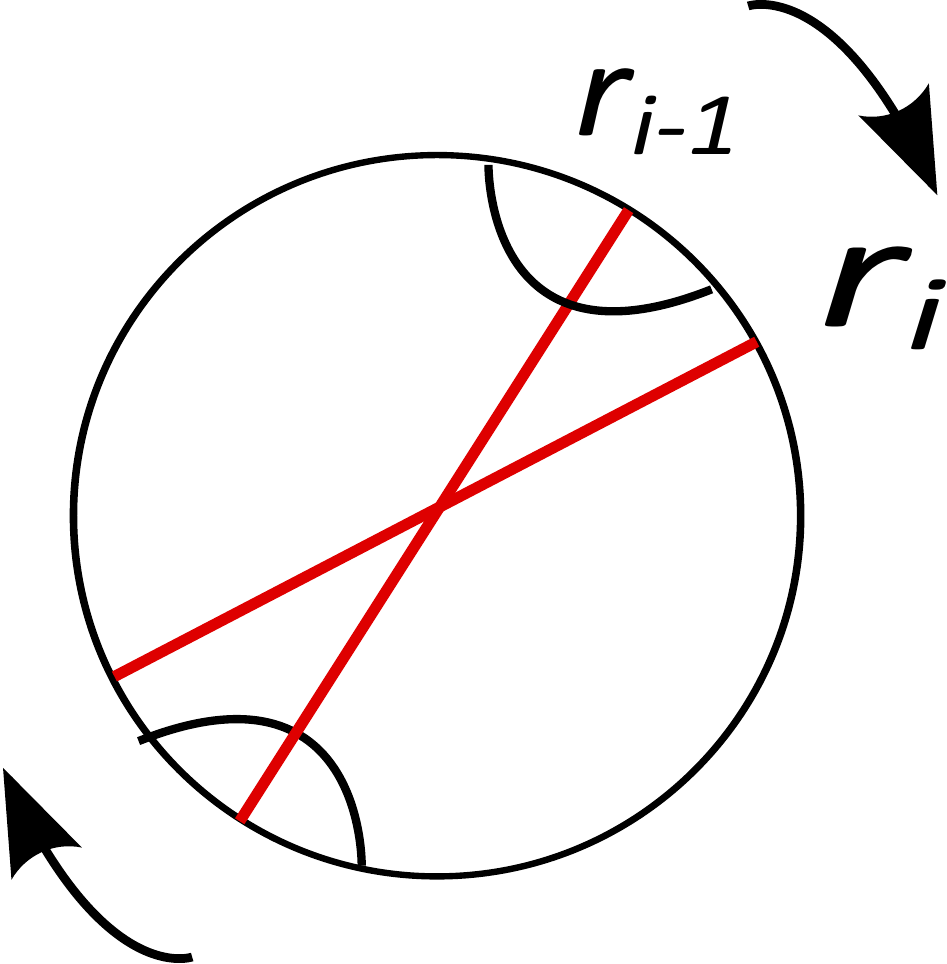}
                \caption{ }
                \label{fig:DecRot}
        \end{subfigure}
        \caption{(a) An increasing rotation (b) A level rotation (c) A decreasing rotation} \label{fig:Rotations}
\end{figure}

  The table below lists all possible changes to $|I(O,S)|$ and $I^*$ due to rotation past one endpoint. The last column displays the net effect on $|I(O,S)| - 2I^*$

\begin{center}
    \begin{tabular}{ | c | c | c | p{5cm} |}
    \hline
  $ \Delta |I(O,S)| $&    $ \Delta I^*$ & $ \Delta (|I(O,S)|-2I^*)$ \\ \hline
    +2 &+2 & -2  \\ \hline
  +2 &+1 & 0  \\ \hline
+2 &0 & +2  \\ \hline
0 &+1 & -2 \\ \hline
0&-1& +2  \\ \hline
-2 &0 & -2  \\ \hline
-2&-1& 0  \\ \hline
-2&-2& +2  \\ \hline
    \end{tabular}
\end{center}

 According to the table, any rotation that moves $O$ from $r_{i-1}$ to $r_{i}$ can increase the quantity $ |I(O,S)|-2|I(O,O',S)|$ by at most 2. Thus
 \[|I(r_{m+j},S)| - 2|I(r_{m+j},O',S)| - (|I(r_0,S)| - 2 |I(r_0,O',S)|) \leq 2(m +j).\]
  Since $r_0$ and $O'$ intersect the same arcs of S, $I(O', S)=I(r_0,S)=I(r_0,O',S)$ which implies
 
 \[(|I(r_0,S)| - 2 |I(r_0,O',S)|) = -|I(O',S)|=-(p + 2m).\] 

It follows that 

\[|I(r_{m+j},S)| - 2I^*=|I(r_{m+j},S)| - 2|I(r_{m+j},O,S)| \leq 2(m + j) - (p +2m) = -p + 2j.\] 

We apply Lemma \ref{t2 bound}, which states if  $|I(O',T)| = n - 2m$, then $|I(r_{m+j},T)| \leq n - 2j$ 

\[ |I(r_{m+j},T)| + |I(r_{m+j},S)| - 2I^* \leq (n-2j)- p+2j=n-p. \] 

Regardless of which rotation $r_{m+j}$ from $O'$ we choose for $O$, we always have 
$$|I(O,S)| + |I(O,T)| - 2I^* \leq n- p.$$
\end{proof}



\begin{lemma}\label{biglem} Let $S$ and $T$ be terms in an $n$-skein relation and let $O$ be any single overstrand placed on the split diagrams of $S$ and $T$. Then 

\begin{enumerate} 

\item If $n$ is even and $\max_{O} \{|I(O,S)| + |I(O,T)| \}= n + 2k$, then $|P(T) - P(S)| \leq \lfloor \frac{n^2}{2} \rfloor - 2k^2$; 
\item If $n$ is odd and $\max_{O} \{|I(O,S)| + |I(O,T)| \}= n + 1 +  2k$, then $|P(T)-P(S)| \leq \lfloor \frac{n^2}{2} \rfloor - 2k^2 - 2k$.
\end{enumerate} Here $k$ is an integer such that $k\leq \lfloor\frac{n}{2}\rfloor$. 
\end{lemma}

\begin{proof} We proceed by induction on $n$. 
\medskip

\noindent \textbf{Base Case:} The Lemma holds for $n=3$ and $n=4$.

\medskip

\noindent \textbf{Induction:} We assume the lemma holds when an overstrand $O$ is added to an $n$-skein relation. Suppose now an additional overstrand $O'$ is added. Without loss of generality, assume $O'$ is a counterclockwise rotation of $O$. We divide the inductive step into two cases and prove each case for $n$ even and for $n$ odd. The notation $\max_{O'}$ denotes the maximum possible value over all choices of $O'$.

\medskip

\noindent \textbf{Case 1:} Assume that $\max_{O'} \{ |I(O',S')| + |I(O',T')| \} \leq \max_{O'} \{ |I(O',S)| + |I(O',T)| \}$. This occurs when resolving the intersections of $O$ with $T$ and with $S$ does not increase the maximum number of arcs a consistently placed overstrand ($O'$) can intersect between both diagrams. 

\medskip

\noindent \textbf{Case 2:} Assume that $\max_{O'} \{|I(O',S')| + |I(O',T')| \}=\max_{O'} \{|I(O',S)| + |I(O',T)| \}+ 2$. This occurs when resolving the intersections of $O$ with $T$ and with $S$ increases the maximum number of arcs a consistently placed overstrand ($O'$) can intersect between both diagrams. 

\medskip

We note that for any $O'$, $|I(O',S')|$ has the same parity as $|I(O',T')|$ and $|I(O',S)|$ has the same parity as $|I(O',T)|$. So it is not possible to have \\ $\max_{O'}\{|I(O',S')| + |I(O',T')|\}=\max_{O'}\{ |I(O',S)| + |I(O',T)| \}+ 1$. In addition, by Lemma \ref{first} $\max_{O'}\{ |I(O',S')| + |I(O',T')| \}\leq \max_{O'} \{ |I(O',S)| + |I(O',T)| \}+ 2$. Therefore, the two cases above are exhaustive. 

\medskip

\noindent \textbf{Proof of Case 1:} We restate our inductive hypothesis:
\[ \max_{O}\{|I(O,S)| + |I(O,T)| \}= \begin{cases} n + 2k \Longrightarrow |P(T)-P(S)| \leq \lfloor \frac{n^2}{2} \rfloor - 2k^2 \ \textrm{if $n$ is even}. \\
n + 1 + 2k \Longrightarrow |P(T)-P(S)| \leq \lfloor \frac{n^2}{2} \rfloor - 2k^2 - 2k \ \textrm{if $n$ is odd.}
\end{cases} \]
We assume the conditions of Case 1: \[\max_{O'}\{|I(O',S')| + |I(O',T')|\}\leq \max_O \{|I(O,S)| + |I(O,T)|\} =
 \begin{cases}
   n+2k \ \textrm{if $n$ is even} \\
   n + 1 + 2k \ \textrm{if $n$ is odd.}
 \end{cases} \]

We will show that
$$ \max_{O'} \{|I(O',S')| +  |I(O',T')|\} \leq \max_O \{|I(O,S)| + |I(O,T)|\}$$
\begin{center}
 $\Longrightarrow$
\end{center}
$$\begin{cases}
  |P(T')-P(S')| \leq \left\lfloor \frac{(n+1)^2}{2} \right\rfloor - 2(k-1)^2 - 2(k-1) \ \textrm{if $n$ is even} \\
\\
|P(T')-P(S')| \leq \left\lfloor \frac{(n+1)^2}{2} \right\rfloor - 2k^2 \ \textrm{if $n$ is odd.}
 \end{cases}$$

Let $\delta$ denote the change in power resulting from resolving the intersections of the overstrand $O$ with splits $S$ and $T$, i.e. $$|P(T')-P(S')|=|P(T)-P(S)| +\delta.$$ 

Without a loss of generality, assume $P(T)\geq P(S)$. Every A-split between $O$ and $T$ increases $\delta$ by one while every B-split between $O$ and $T$ decreases $\delta$ by one. Every B-split between $O$ and $S$ increases $\delta$ by one while every A-split between $O$ and $S$ decreases $\delta$ by one. Therefore, $\delta \leq\max_{O}\{ |I(O,S)| + |I(O,T)|\} = n + 2k$. This maximum occurs when the intersections of $O$ and $T$ are resolved as A-splits and the intersections of $O$ and $S$ are resolved as B-splits. Using this fact and the inductive hypothesis that $|P(T)-P(S)| \leq \left\lfloor \frac{n^2}{2} \right\rfloor - 2k^2$, we obtain  

\begin{align*} 
 |P(T')-P(S')| &\leq
 \begin{cases}
   \left\lfloor \frac{n^2}{2} \right\rfloor - 2k^2 + n + 2k \ \textrm{if $n$ is even} \\
\\
\left\lfloor \frac{n^2}{2} \right\rfloor - 2k^2 - 2k + n + 1 + 2k \ \textrm{if $n$ is odd}
 \end{cases}
 \\ 
&=
\begin{cases}
  \left\lfloor \frac{(n+1)^2}{2} \right\rfloor - 2k^2 + 2k \ \textrm{if $n$ is even} \\
\\
\left\lfloor \frac{(n+1)^2}{2} \right\rfloor - 2k^2 \ \textrm{if $n$ is odd} \\
\end{cases}
\\
&=
\begin{cases}
  \left\lfloor \frac{(n+1)^2}{2} \right\rfloor - 2(k-1)^2 - 2(k-1) \ \textrm{if $n$ is even} \\
\\
\left\lfloor \frac{(n+1)^2}{2} \right\rfloor - 2k^2 \ \textrm{if $n$ is odd,}
\end{cases}
\end{align*} 

as desired. 

\medskip

\noindent \textbf{Proof of Case 2} We restate our inductive hypothesis: $$\max_{O}  \{  |I(O,S)| + |I(O,T)| \} =
\begin{cases}
   n + 2k \Longrightarrow |P(T)-P(S)| \leq \left\lfloor \frac{n^2}{2} \right\rfloor - 2k^2 \ \textrm{if $n$ is even} \\
   \\
   n + 1 + 2k \Longrightarrow |P(T)-P(S)| \leq \left\lfloor \frac{n^2}{2} \right\rfloor - 2k^2 - 2k \ \textrm{if $n$ is odd.}
\end{cases}$$

We assume the conditions of Case 2:
\begin{equation}\label{increasing}
\max_{O'} \{|I(O',S')| + |I(O',T')|\}=\max_O \{|I(O,S)| + |I(O,T)|\}+2=
\begin{cases}
  n+2k+2 \ \textrm{if $n$ even} \\
  n + 1 + 2k + 2 \ \textrm{if $n$ odd.}
\end{cases}
\end{equation}
We will show that 
\[\max_O\{|I(O',S')| + |I(O',T')|\} =
\begin{cases}
 (n + 1) +1 + 2k \Rightarrow |P(T')-P(S')| \leq \left\lfloor \frac{(n+1)^2}{2} \right\rfloor - 2k^2 - 2k, \ \textrm{$n$ even}\\
 \\
n+1+2k + 2 \Rightarrow |P(T')-P(S')| \leq \left\lfloor \frac{(n+1)^2}{2} \right\rfloor - 2k^2, \ \textrm{$n$ odd.}
\end{cases}\]

This case occurs when $|I(O',S')|=|I(O',S)|+1$ and $|I(O',T')|=|I(O',T)|+1$. Recall that by Lemma \ref{first}, this occurs precisely when all the intersections of $O$ with $S^*$ and all the intersections of $O$ with $I(O',S)$ and all the intersections of $O$ with $I(O',T)$ are resolved as A-splits. To maximize $\delta$ under these conditions, we split all intersections of $O$ with $I(O,T) - I(O',T)$ as as A-splits and all the intersections of intersections of $O$ with $I(O,S) - I(O',S)$ as B-splits. Ordinarily, $\delta \leq n + 2k$. But Lemma \ref{first} informs us that the assumption in Equation \ref{increasing} necessitates performing `inefficient' A-splits on $S$. Each of these $A$ splits reduces the maximum value of $\delta$ by 2. Retaining the assumption that $I(O,O',S) = I^*$, we have $\delta \leq I(O,S)+I(O,T)-2I^*$. By Lemma \ref{istar}, $I(O,S)+I(O,T)-2I^* \leq n - 2k$.

We observe
\begin{align*} 
 |P(T')-P(S')|&=|P(T)-P(S)| +\delta\\
&\leq
\begin{cases}
  \left\lfloor \frac{n^2}{2} \right\rfloor - 2k^2 + n - 2k \ \textrm{if $n$ is even} \\
  \\
\left\lfloor \frac{n^2}{2} \right\rfloor - 2k^2 - 2k + n - 2k - 1 \ \textrm{if $n$ is odd}
\end{cases}
\\
&= 
\begin{cases}
   \left\lfloor \frac{(n+1)^2}{2} \right\rfloor - 2k^2 - 2k \ \textrm{if $n$ is even} \\
   \\
 \left\lfloor \frac{(n+1)^2}{2} \right\rfloor - 2(k+1)^2 \ \textrm{if $n$ is odd,}
\end{cases}
\end{align*} 
 as desired.
\end{proof}

\begin{theorem}\label{Upper Bound}
  Let $R$ be an $n-$skein relation for $n\geq 2$. Then, $w(R) \leq \left\lfloor \frac{n^2}{2}\right\rfloor$.
\end{theorem}

\begin{proof}
  We proceed by induction on $n$, showing that for any terms $S$ and $T$ in $R$, 
  \[|P(T)-P(S)| \leq \left\lfloor \frac{n^2}{2} \right\rfloor. \]

\medskip

\noindent \textbf{Base Case:} Observe that this theorem holds for $n=2,3,4$ by \citep{Kauffman87, Triple, Quadruple}. \\


\medskip

\noindent \textbf{Induction:} Suppose for all $S$ and $T$ in an $n$-skein, $| P(T)-P(S)| \leq \left\lfloor \frac{n^2}{2} \right\rfloor$. We will show that $| P(T')-P(S')| \leq \left\lfloor \frac{(n+1)^2}{2} \right\rfloor$ where $S'$ and $T'$ are any terms in in the $(n+1)$-skein relation obtained by adding $O'$ to the $n$-skein and resolving the crossings. Since $S'$ and $T'$ are the results of resolving intersections of $O$ with $S$ and the intersections of $O$ with $T$, we obtain, 
$$|P(T')-P(S')|\leq |P(T)-P(S)|+|I(O,S)|+|I(O,T)|.$$

\medskip

\noindent \textbf{Case 1:} Suppose  $ \max_O \{ |I(O,S)| + |I(O,T)| \} \leq n$. \\
 We use the inductive hypothesis and obtain, 
 $$ |P(T')-P(S')|\leq |P(T)-P(S)|+|I(O,S)|+|I(O,T)|\leq \left\lfloor\frac{n^2}{2}\right\rfloor+n\leq  \left\lfloor \frac{(n+1)^2}{2} \right\rfloor.$$

\medskip

\noindent \textbf{Case 2:} Suppose  $\max_O \{ |I(O,S)| + |I(O,T)| \} > n.$

Assume $n$ is even and $\max_O \{ |I(O,S)| + |I(O,T)|\} =n + 2k$ for $k$ an integer such that $1 \leq k \leq \frac{n}{2}$. By the result of Lemma \ref{biglem},  \[ |P(T)-P(S)| \leq \lfloor \frac{n^2}{2} \rfloor - 2k^2. \] We obtain
\[|P(T')-P(S')|\leq |P(T)-P(S)|+|I(O,S)|+|I(O,T)|\leq\left\lfloor \frac{n^2}{2} \right\rfloor - 2k^2 + n + 2k \leq \left\lfloor \frac{(n+1)^2}{2} \right\rfloor. \]

Assume $n$ is odd and $\max_O \{|I(O,S)| + |I(O,T)|\} = n +1+ 2k$ for $k$ an integer such that $1 \leq k \leq \frac{n}{2}$. By the result of Lemma \ref{biglem}, $\max_O\{|I(O,S)| + |I(O,T)|\} = n + 1+2k $ implies $|P(T)-P(S)| \leq \lfloor \frac{n^2}{2} \rfloor - 2k^2-2k$. We obtain
\[ |P(T')-P(S')|\leq |P(T)-P(S)|+|I(O,S)|+|I(O,T)| \leq\left\lfloor \frac{n^2}{2} \right\rfloor - 2k^2 -2k+ n + 1+ 2k \leq \left\lfloor \frac{(n+1)^2}{2} \right\rfloor. \]

\end{proof}

We now demonstrate that this upper bound on the width of an $n$-skein relation is the best possible.  Let $<12 \ldots n>$ denote the skein relation for the crossing in which the $k^{th}$ highest strand is immediately clockwise of the $({k-1})^{st}$ highest strand for $2 \leq k \leq n$.

\begin{corollary}For $n\geq 2$, $w(<12 \ldots n>) = \left\lfloor \frac{n^2}{2} \right \rfloor$. 
\end{corollary}

\begin{proof}

	For convenience, we will denote the lowest power of $A$ appearing in the skein relation for $<12 \ldots n>$ by $l(<12 \ldots n>)$. We denote the highest power of $A$ appearing in the skein relation by $h(<12 \ldots n>)$. \\
    
	We will use an inductive argument to show that $l(<12 \ldots n+1>) = l(<12 \ldots n>) - n$. We refer to the overstrand that transforms $<12 \ldots n>$ into $<12 \ldots n+1>$ as $O_n$ \\
    
   \medskip

\noindent \textbf{Base Case:} For $<12>$, the lowest power of $A$ is realized by exactly one term, $T_2$: a parallel split that realizes $|I(O_2, T_2)| = 2$. \\
    
   \medskip

\noindent  \textbf{Induction:} Assume that there is only one term on the lowest power of $<12 \ldots n>$, a term $T_n$ such that $|I(O_n, T_n)| = n$. An overstrand placed over $<12 \ldots n>$ can intersect at most $n$ arcs per term. So, if $T_n$ had the unique lowest power of $A$ in $<12 \ldots n>$, the term obtained from performing all B-splits on this term will be the only term on the lowest level of the skein relation $<12 \ldots n+1>$. The power of this term is exactly $n$ lower than $P(T_n)$. Moreover, the new split is in fact the parallel split that realizes $n+1$ intersections with the overstrand $O_{n+1}$. \\
    
    \begin{figure}[H]
        \centering
        \begin{subfigure}[b]{0.2\textwidth}
                \centering
                \includegraphics[width=\textwidth]{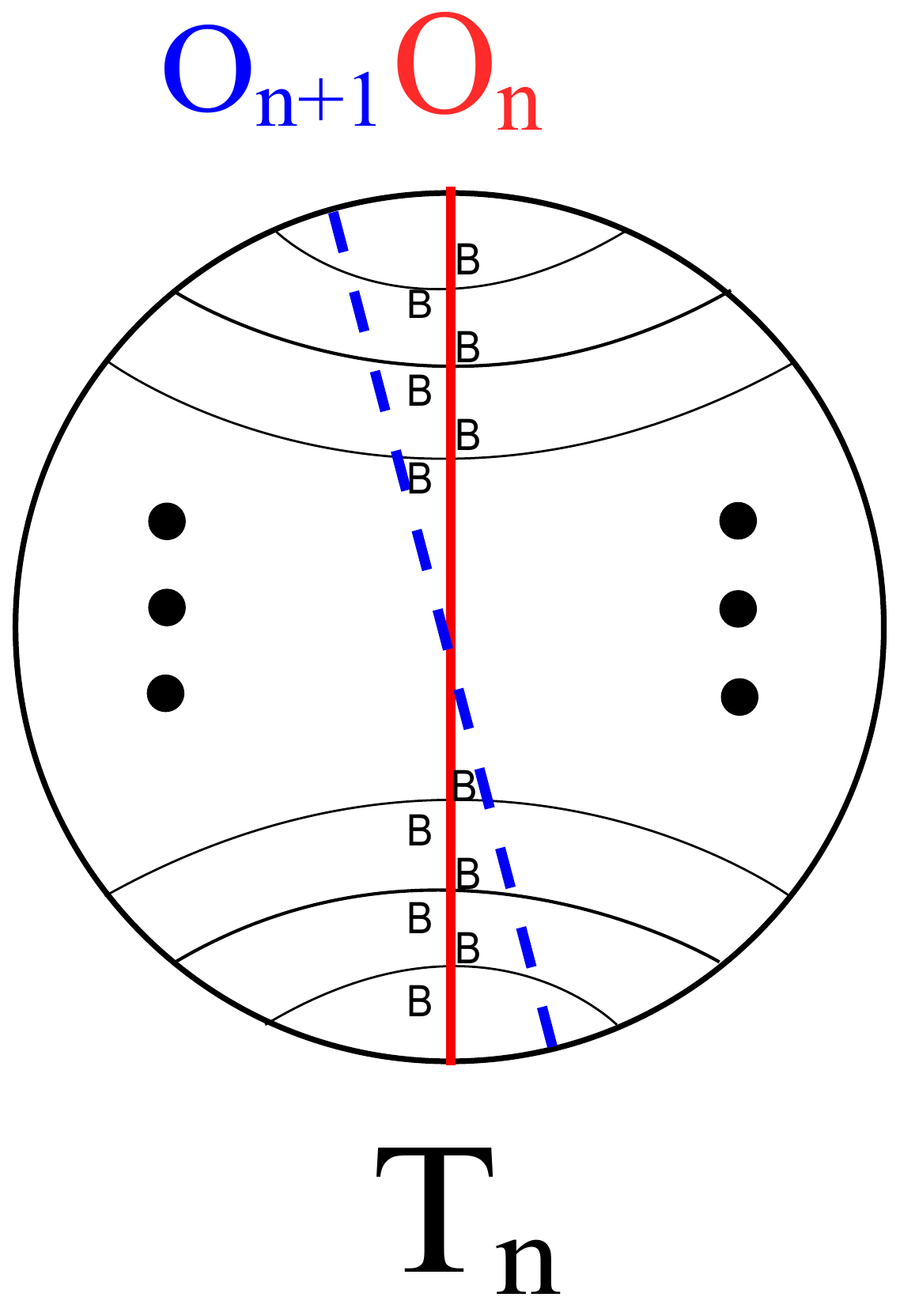}
                \caption{ }
                \label{fig:Tn}
        \end{subfigure}
        \begin{subfigure}[b]{0.2\textwidth}
                \centering
                \includegraphics[width=\textwidth]{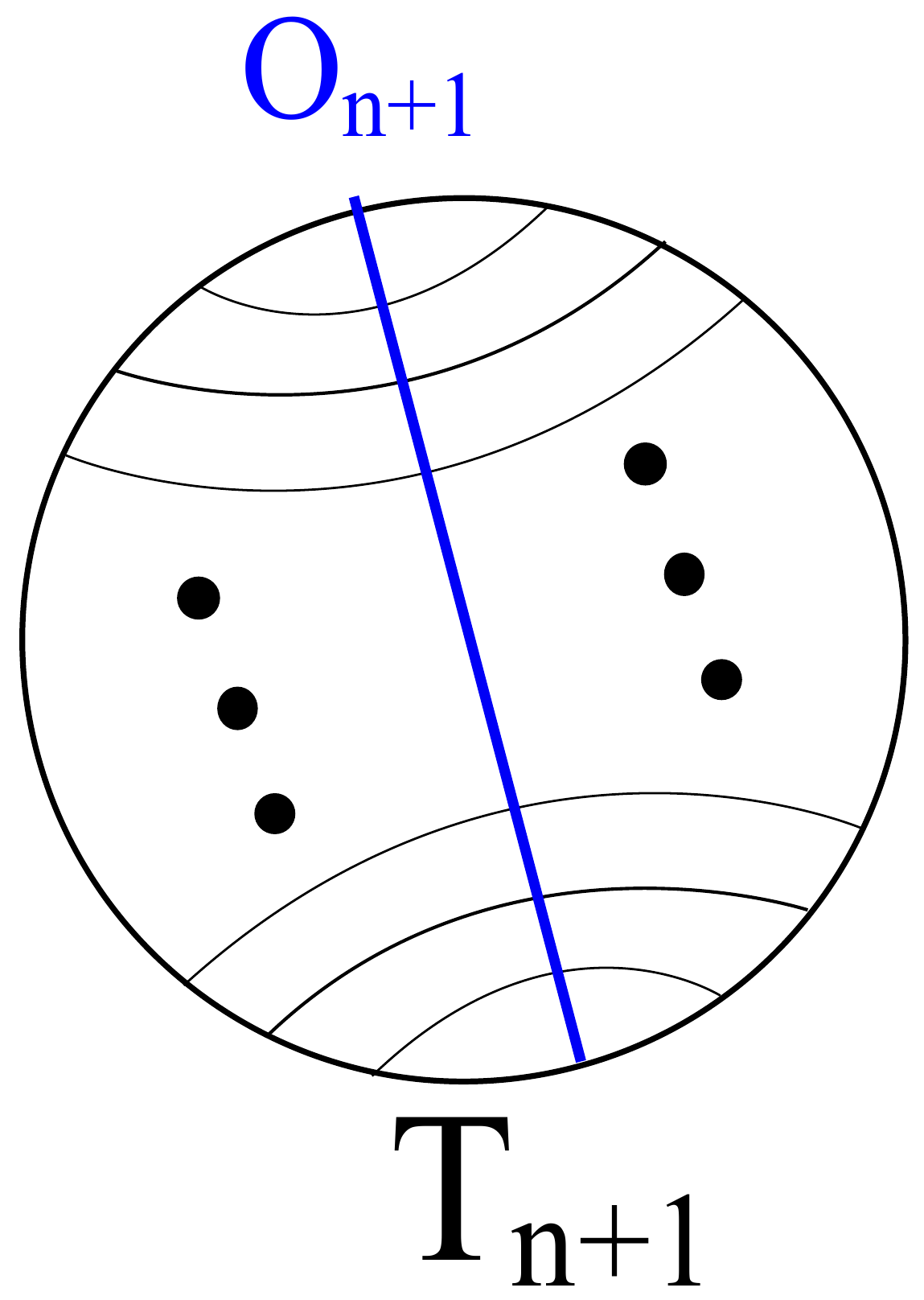}
                \caption{ }
                \label{fig:Tn+1}
        \end{subfigure}
        \caption{(a) $T_n$ with $O_n$, $O_{n+1}$ and $B$ regions marked (b) The result of performing all $B$ splits on the intersections of $O_n$ with $I(O_n, T_n)$} \label{fig:tightpic}
\end{figure}
  
  We have shown that $l(<12 \ldots n+1>)  = l(<12 \ldots n>) - n$. Using this fact, we perform a second inductive argument to show that 
  
  \[ w(<12 \ldots n>) \geq \left\lfloor \frac{n^2}{2} \right\rfloor \Longrightarrow w(<12 \ldots n+1>) \geq \left\lfloor \frac{(n+1)^2}{2} \right\rfloor \] 
    
   \medskip

\noindent \textbf{Base Case:} Our base case comes from the observation that $w(<12>) \geq \left\lfloor \frac{2^2}{2} \right\rfloor = 2$.
    
   \medskip

\noindent \textbf{Induction:} Assume $w(<12 \ldots n>) \geq \left\lfloor \frac{n^2}{2} \right\rfloor$. 
    
    First, we observe that any overstrand intersects an odd number of arcs when $n$ is odd. So, when $n$ is odd, we can perform at least one $A$ split on a term that realizes the highest power of $A$, resulting in  $$h(<12 \ldots n+1>) \geq h(<12 \ldots n>) + 1$$ 

    When $n$ is even, we can still assert that $h(<12 \ldots n+1>) \geq h(<12 \ldots n>)$ \\
    
    We have
    \begin{enumerate}
    \item \hspace{1cm} $\displaystyle{h(<12 \ldots n+1>) \geq h(<12 \ldots n>) + \begin{cases} 1\ \textrm{when $n$ is odd} \\
    0 \ \textrm{when $n$ is even}
    \end{cases} }$\\
    \item \hspace{1cm} $\displaystyle{l(<12 \ldots n+1>) = l(<12 \ldots n>) - n.}$
    \end{enumerate}
    
    \medskip
    
   Noting that $w(<12 \ldots n>) = h(<12 \ldots n>) - l(<12 \ldots n>)$, we subtract (2) from (1) to obtain 
    \begin{center}
 \begin{eqnarray*}   
w(<12 \ldots n+1>) =  w(<12 \ldots n+1>) + \begin{cases}
n+1 \ \textrm{when $n$ is odd} \\
n \ \textrm{when $n$ is even}
\end{cases} \\
w(<12 \ldots n+1>)\geq \left\lfloor \frac{n^2}{2} \right \rfloor + \begin{cases}
n+1 \ \textrm{when $n$ is odd} \\
n \ \textrm{when $n$ is even}
\end{cases}\\
w(<12 \ldots n+1>) \geq \left\lfloor \frac{(n+1)^2}{2} \right \rfloor
\end{eqnarray*}
\end{center}

This completes our second inductive proof, showing that for all $n$, $w(<12 \ldots n>) \geq \left\lfloor \frac{n^2}{2} \right\rfloor$. \\

By Theorem \ref{Upper Bound}, $w(<12 \ldots n>) \leq \left\lfloor \frac{n^2}{2} \right\rfloor$. So, $w(<12 \ldots n>) = \left\lfloor \frac{n^2}{2} \right\rfloor$.
 
\end{proof}


\section{Bound on the number of components}\label{components}

In this section we aim to prove the following theorem about the number of  components in maximal and minimal states.

\begin{theorem} \label{componentsthm}
Let $n \geq 3$ and let $K$ be a connected multi-crossing projection with $c_n$ $n$-crossings. Let $s_{max}$ be a maximal state and $s_{min}$ be a minimal state. Then $$|s_{max}|+|s_{min}| \le (2n-4)c_n+2.$$
\end{theorem}

We  reduce the proof of Theorem~\ref{componentsthm} to finding a lower bound for the split distance between any high and low split. Lemmas \ref{sd} and \ref{sdvsnc} give the crucial connection between the split distance between two splits and the total number of components formed when we replace a given crossing with each split.  When we replace a crossing with a split to calculate the number of connected components, we call the split we use the \textit{interior split} $S$. Imagine the interior split being bounded by an ``invisible" circle $C$, as shown in Figure \ref{inside_outside}. When we follow the knot away from one of the endpoints on $C$, we  eventually return to a different endpoint on $C$. Connect all such pairs of endpoints by a collection of disjoint arcs outside $C$. We call this connection of endpoints  the \textit{exterior split} $T$. Let $||S,T||$ denote the number of topological circles formed by connecting the interior split $S$ with the exterior split $T$. The split distance $d(S,T)$ is defined to be the split distance between $S$ and the reflection of $T$ through $C$.

\begin{figure}[H]
        \centering
        \begin{subfigure}[b]{0.3\textwidth}
                \centering
                \includegraphics[width=\textwidth]{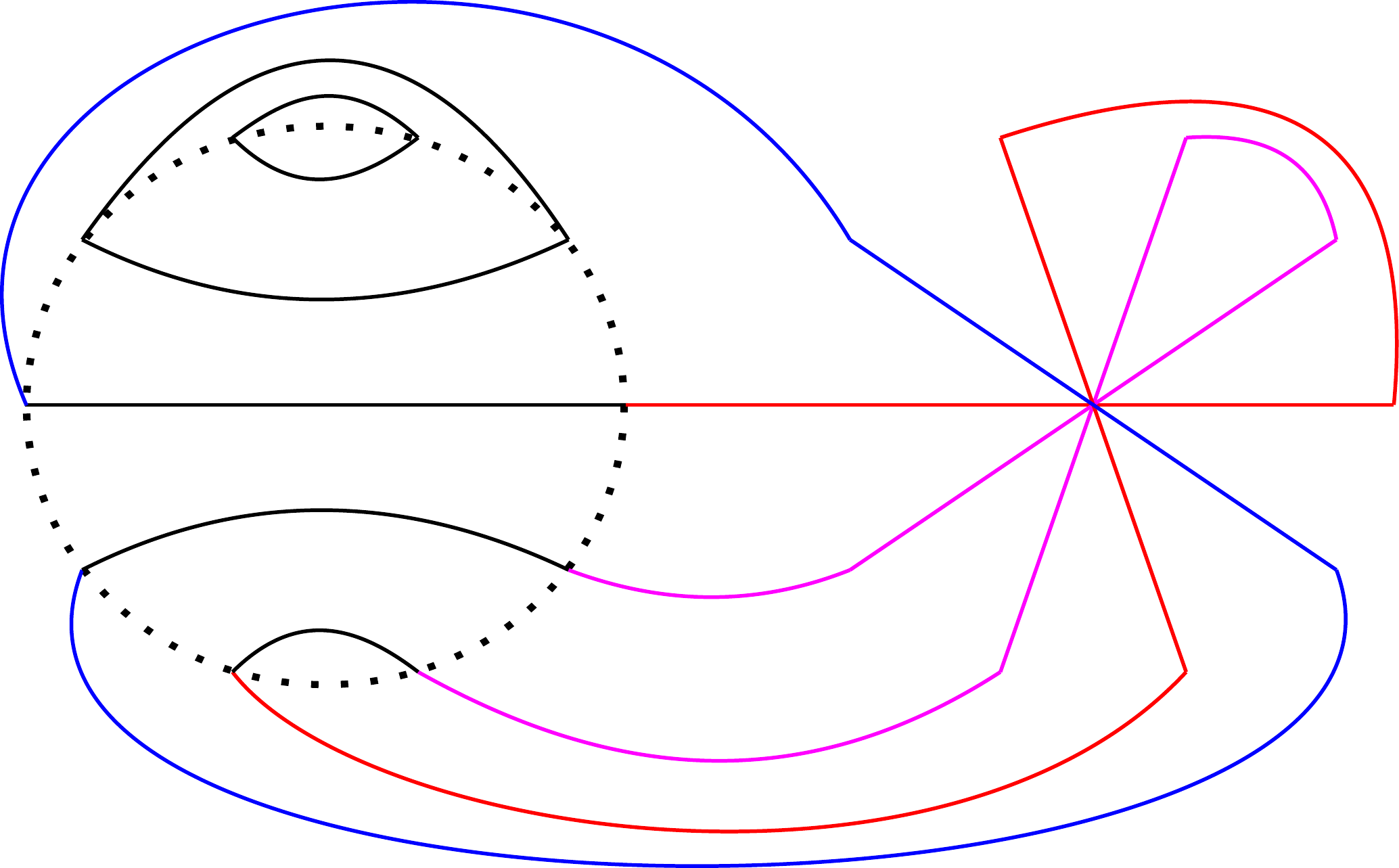}
                \label{fig:getextsplit1}
        \end{subfigure}
\qquad
        \begin{subfigure}[b]{0.15\textwidth}
                \centering
                \includegraphics[width=\textwidth]{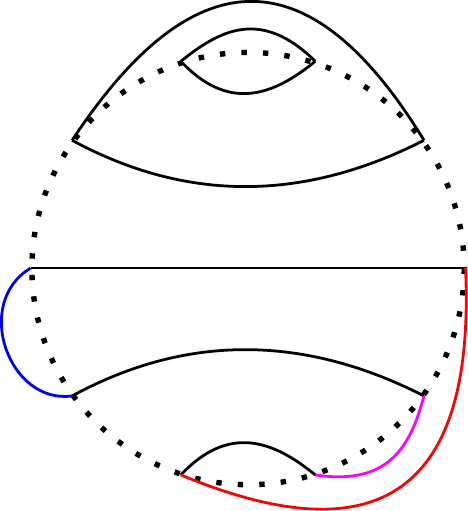}
                \label{fig:getextsplit2}
        \end{subfigure}
        \caption{Constructing the exterior split. In this instance $||S,T||=3$.} \label{fig:extsplit}
\end{figure}

\begin{lemma}
\label{sd}
Let $x$ be a crossing in an $n$-crossing projection with $n \geq 2$, and 
let $S$ be an interior split of $x$ and $T$  the corresponding  exterior split. Then
$$d(S,T) = n-||S,T||.$$
\end{lemma}

\begin{proof}
Both $S$ and $T$ have $n$ arcs. For the duration of this proof, we define the length of a connected component to be the number of times the connected component intersects the ``invisible" circle $C$ separating both splits. Assume each of the $||S,T||$ connected components is $2k_i$ arcs long such that $1 \leq i \leq ||S,T||$, so that  $\sum^{||S,T||}_{i=1} k_i=n$.

\begin{figure}[h]
\begin{center}
\includegraphics[width=0.4\columnwidth,angle=0] {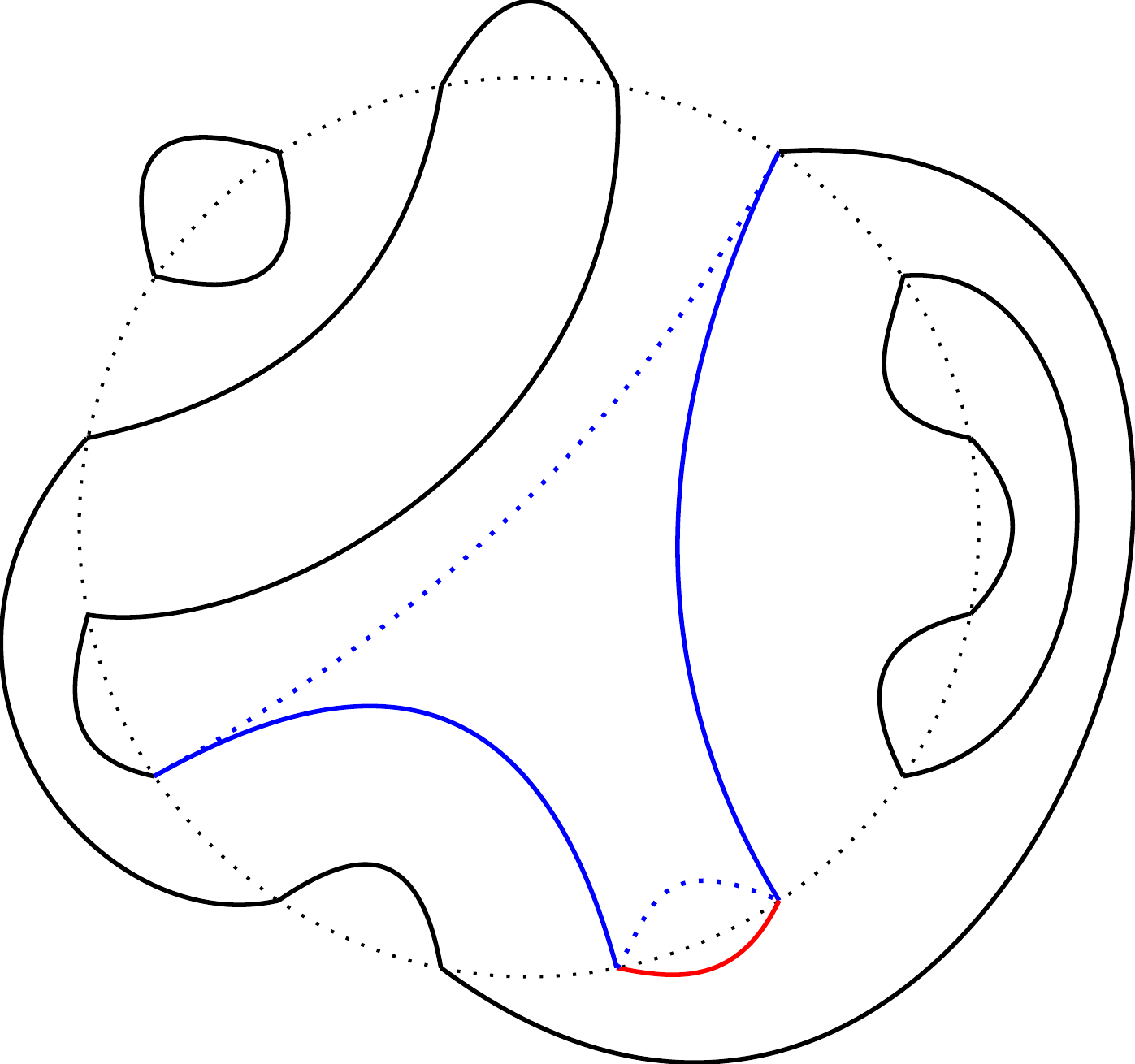} 
\caption{The red arc in $T$ is a non-nesting arc. Performing a split move on the blue arcs in $S$ changes a component of length $10=2(5)$ into two components of length $8=2(5-1)$ and 2.\label{inside_outside} }
\end{center}
\end{figure}

We want to perform a number of split moves on $S$ until the number of connected components is $n$ and each connected component has a length of two. In other words, we perform split moves until the interior split $S$ has turned into $T$. When performing split moves, we ignore any connected components of length 2 that are formed by common arcs in $S$ and $T$. Instead, choose a non-nesting arc in $T$, which is to say, an arc such that there exists a disk  bounded by that arc and an arc on the invisible circle that contains no other arcs of $T$ other than those we are ignoring. We perform a split move on the two distinct arcs in $S$ sharing the endpoints. Note that there at are least two non-nested arcs in $T$. Performing the split move changes a component of length $2k_i$ into two components of length $2(k_i-1)$ and 2, as shown in Figure \ref{inside_outside}. Continue the process until all components have length 2, i.e. the interior split $S$ has turned into $T$. Inducting on $k_i$, each length $2k_i$ component has required $k_i-1$ split moves in total. So the split distance between $S$ and $T$ satisfies $$d(S,T) \leq \sum^{||S,T||}_{i=1} (k_i-1)=\sum^{||S,T||}_{i=1} k_i-||S,T||=n-||S,T||.$$ 

We can find a lower bound on $d(S,T)$ by considering that each split move can increase the number of components by at most 1. Therefore, given a fixed exterior split, changing from an interior split that creates $n$ components to an interior split that creates $||S,T||$ components requires at least $n-||S,T||$ split moves, i.e. $d(S,T) \geq n-||S,T||$. Thus, $$d(S,T) = n-||S,T||.$$
\end{proof}

In the case that the knot or link has a single $n$-crossing, $||S,T||$ is the number of components of the state corresponding to the split $S$. However, in knot projections with more than one multi-crossing, a state consists of more than one split, as each multi-crossing is split in a specific way. We need to know the number of components that result when all crossing are split. However, for induction we will consider splitting one crossing at a time.  Split one crossing of the knot. The resultant projection is a link. We call a topological circle in this projection a link component. We define a projection component to be a connected component of the projection when it is considered as an object on the plane. In other words, all link components that pass through the same multi-crossing are part of the same projection component.  Each projection component occupies a disjoint region in the plane. Clearly, the number of link components is always greater than or equal to the number of projection components since each projection component contains at least one link component. 

\begin{figure}[H]
        \centering
        \begin{subfigure}[b]{0.3\textwidth}
                \centering
                \includegraphics[width=\textwidth]{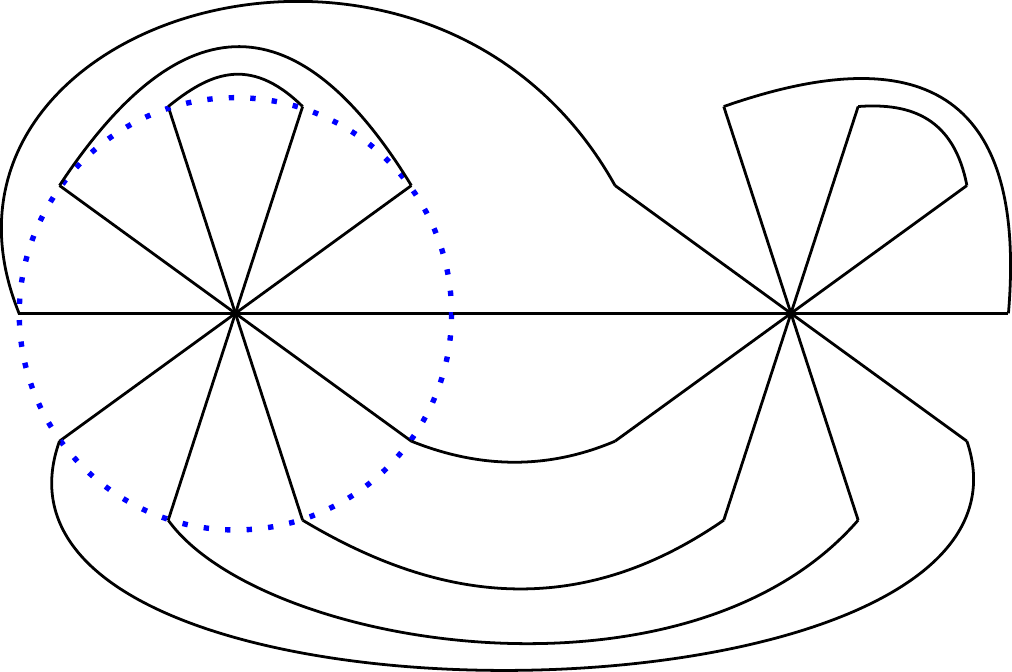}
                \caption{}
                \label{fig:split before}
        \end{subfigure}
        \begin{subfigure}[b]{0.3\textwidth}
                \centering
                \includegraphics[width=\textwidth]{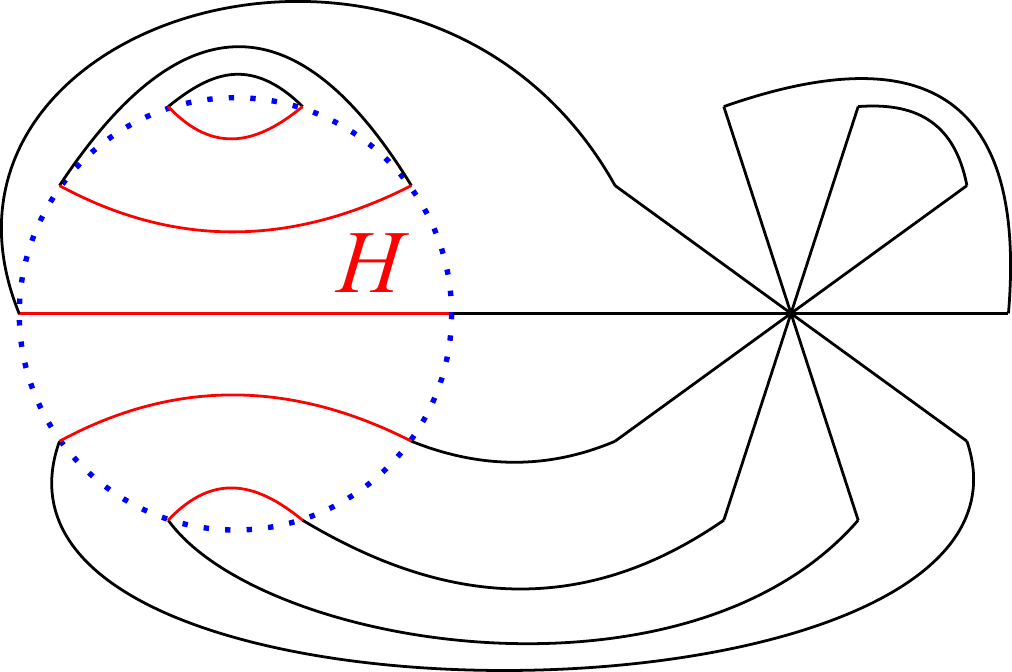}
                \caption{}
                \label{fig:split after}
        \end{subfigure}
\begin{subfigure}[b]{0.3\textwidth}
                \centering
                \includegraphics[width=\textwidth]{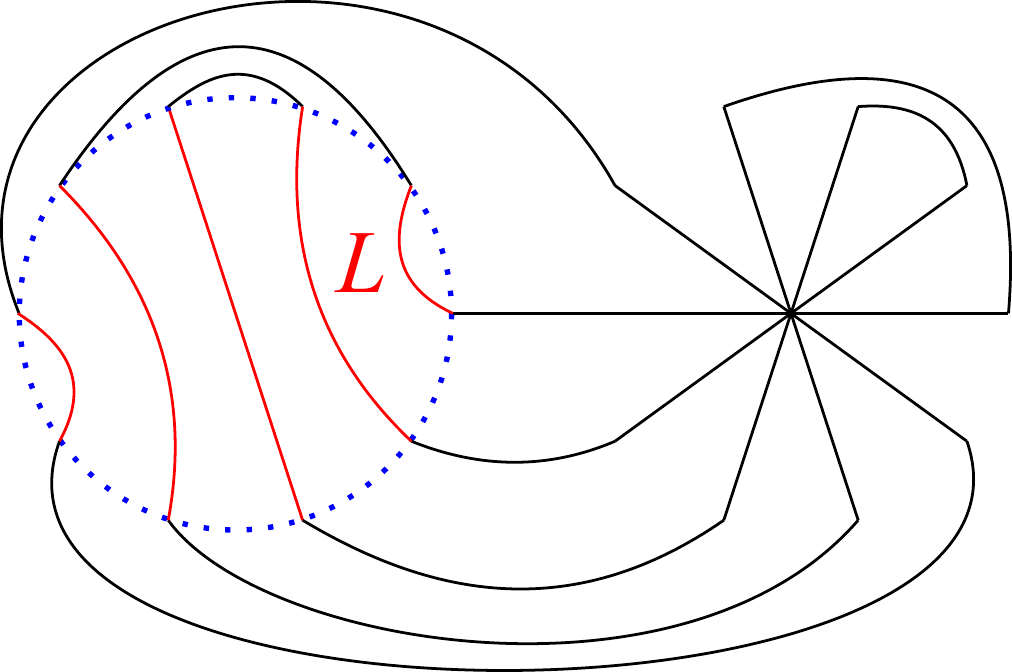}
                \caption{}
                \label{fig:split after 2}
        \end{subfigure}
        \caption{(a) A projection with two 5-crossings. (b) The result of splitting a 5-crossing as $H$ has 3 link components and 3 projection components. (c) The result of splitting the 5-crossing as $L$ has 3 link components and 1 projection component.} \label{fig: split at crossing}
\end{figure}

\begin{lemma}
\label{sdvsnc}
Let $n \geq 2$ and let $U$ and $V$ be two different splits at a specified $n$-crossing $x$ in a projection $P$, that split $P$ into $k_U$ and $k_V$ projection components. Then $k_U+k_V\le 2n-d(U,V)$.
\end{lemma}
\begin{proof} Obtain the exterior split of $x$  as described in Figure \ref{fig:extsplit}. We call the exterior split $T$.
We know $d(U,T)=n-||U,T||$, and $d(V,T)=n-||V,T||$. By the triangle inequality,
$$||U,T||+||V,T||=(n-d(U,T))+(n-d(V,T)) \leq 2n-d(U,V).$$
Note that $||U,T||$ is the number of link components corresponding to the $U$ split. Thus $k_U\leq||U,T||$. Likewise, $k_V\leq||V,T||$. We obtain  $k_U+k_V\le 2n-d(U,V)$. 

\end{proof}

\begin{lemma}\label{general_components_theorem}
Let $P$ be a connected projection with $c_n$ $n$-crossings for $n \geq 2$. If $D \le d(H,L)$ for any high split $H$ and low split $L$, i.e. any high and low splits are at least $D$ split moves apart, then $$|s_{max}|+|s_{min}| \le (2n-D-2)c_n+2.$$
\end{lemma}

\begin{proof}

\noindent\textbf{Base Case:} First consider $c_n=1$. Let $T$ be the exterior split and $H$,$L$ be a high and a low split. By Lemma \ref{sdvsnc},
$$|s_{max}|+|s_{min}|=||H,T||+||L,T|| \leq 2n-d(H,L) \leq 2n-D.$$ This proves the base case.

\medskip

\noindent \textbf{Induction:}

Assume $|s_{max}|+|s_{min}| \le (2n-D-2)c+2$ for all non-splittable knot projections with $c_n$ $n$-crossings. Now consider any non-splittable $n$-crossing projection $K$ with $c_n + 1$ crossings. We want to show that $|s_{max}| + |s_{min}| \leq (2n -D-2)(c+1) + 2$. 
 
Pick one crossing $x$ and split it as the high split $H$ that corresponds to the state $s_{max}$, resulting in a new projection $P'$ with $k_H$ projection components, $P'_1, P'_2,\ldots, P'_{k_H}$. Each $P'_i$ is a single projection component  such that  $P_i'$ is disjoint from $P_j'$ for $i\not=j$. Let $c_i$ be the number of $n$-crossings in $P'_i$. For each $i$, $c_i\leq c_n$ since the projection initially had $c_n+1$ crossings. We
have $\sum_{i=1}^{k_H} c_i=c$.So we can apply the inductive hypothesis to each such component.

 Denote maximal and minimal states for
each $K'_i$ as $s'_{max,i}$ and $s'_{min,i}$. Because the $K'_i$ are
disjoint, resolving the crossings in one $K'_i$ does not affect the
number of components formed by resolving the crossings in some other
$K'_{i'}$. Therefore, $|s_{max}| = \sum_{i=1}^{k_H}|s'_{max,i}|$. 

We note that to achieve the decompositions that yield the maximum
(minimum) power of $A$, we choose at each crossing a split that is maximal
(minimal) with respect to the skein relation of that
crossing. However, within a skein relation, there may be multiple
maximal (minimal) splits. The split that yields the highest (lowest)
power of $A$ depends on the configuration of outside strands, which is
dictated, for some projection components, by the choice of splitting
performed at $S$. 

When we begin with $P$ and split $x$ as $L$, we obtain $k_L$
projection components, which we label $P''_1, P''_2,\ldots,
P''_{k_L}$. For the same reason as above, $|s_{min}| =
\sum_j^{k_L}|s''_{min,j}|$. We distinguish $|s'_{min,i}|$ from
$|s''_{min,i}|$ by the fact that the minimal splitting that realizes
the lowest power of $A$ depends on the splitting at $x$. The splits
chosen to realize $|s''_{min,i}|$ come from the minimal splits that realize
the lowest power of $A$ when $x$ is split as $L$. 

 Suppose $x$ is split as $H$ and each other crossing in the projection
 is split as the minimal split that realizes the lowest power of $A$
 when $x$ is split as $L$. We are left with a set of disjoint
 connected components. Changing the split of $x$ from $H$ to $L$
  corresponds to surgering the circles around that crossing to
 add $-k_H + k_L$ more circles to our state. 

It follows that  $$\sum_{j=1}^{k_L} |s''_{min, j}|=\sum_{i=1}^{k_H} |s''_{min, i}|-k_H+k_L.$$

By Lemma~\ref{sdvsnc}, $||H,T||+||L,T||=k_H+k_L \leq 2n-D$. Therefore,
$|s_{min}| = \sum_{i=1}^{k_H} |s''_{min, i}| - k_H+k_L\leq \sum_{i = 1}^{k_H} |s_{min,i}| - k_{H}+(2n-D-k_{H})$.

By combining our inductive hypothesis and Lemma~\ref{sdvsnc}, we
obtain the following:
\begin{align*}
|s_{max}| + |s_{min}| 
&= \sum_{i = 1}^{k_H} (|s'_{max, i}| + |s''_{min, i}|) - k_H+k_L \\
&\le  \sum_{i = 1}^{k_H} ((2n - 2-D)c_i + 2) - k_H+(2n-D-k_H) \\
&= (2n -D-2)c + 2n - D \\
&= (2n-D-2)(c + 1) + 2
\end{align*}

\end{proof}

\begin{remark}
When $k_H<n-D$, the equality $k_L= 2n-D-k_H$ cannot be realized because $k_L \le n <2n-D-k_H$. 

\end{remark}

In order to complete the proof of Theorem~\ref{componentsthm}, it suffices to show that for $n \geq 3$,  $D \geq 2$ for $H$ and $L$ and then apply Lemma~\ref{general_components_theorem}. In other words, we need to show that the split distance $d(H,L)$ between any high split and any low split is at least 2. In order to do so we must investigate how the high and low splits were generated when we constructed the $(n+1)$-skein relation.

  \begin{definition} Let $S$ be a term in an $n$-skein relation and $S'$ be a term in a corresponding $(n+1)$-skein relation. If the split of $S'$ was obtained by resolving intersections of an overstrand with the split of $S$, we call $S$ the \textit{mother} of $S'$ as illustrated in Figure \ref{mother}. We consider $S'$ to be an \textit{offspring} of $S$.
  \end{definition}

\begin{figure}[h]
\begin{center}
\includegraphics[width=0.4\columnwidth,angle=0] {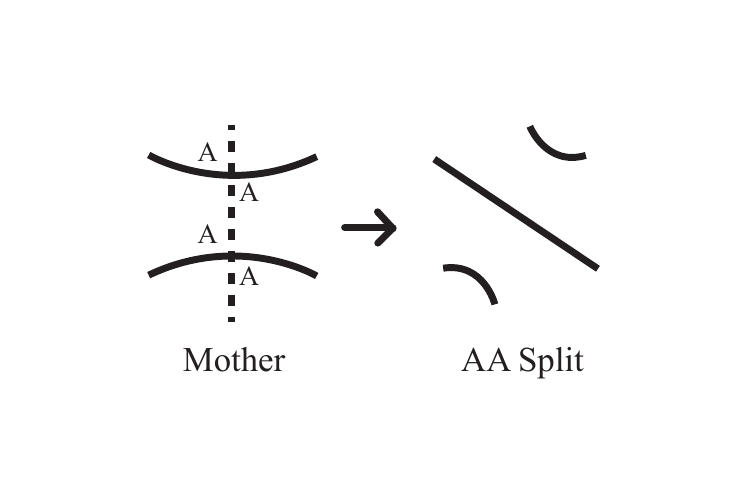} 
\caption{\label{mother} The AA split is the result of resolving the intersections between the overstrand (dashed line) and the mother as A splits.}
\end{center}
\end{figure}

 Let the term $S$ be the mother of the term $S'$. If $I(O,S)\neq 0$ then $S$ generates more than one offspring as we resolve the intersections with the overstrand as different combinations of $A$-splits and $B$-splits.  Let $S'$ and $R'$ be offspring obtained from the same mother. If the sequence of $A$-splits and $B$-splits to obtain them differs in only one place, then $d(S',R') = 1$. 
 
 \begin{definition}
 Consider the endpoints of the most recently added overstrand, assuming it is vertical,  after the double crossings have been resolved. Travel away from one of the new endpoints. If the arc turns in the clockwise direction away from the vertical with respect to the starting endpoint for both of the newest endpoints, then we call the split a \textit{clockwise split}. Similarly,  if the arc turns in the counterclockwise direction from the vertical  with respect to the starting endpoint for both of the newest endpoints, then we call the split a \textit{counterclockwise split}. Finally, if the endpoints of the most recently added strand are connected by a straight line, then we call the split a \textit{straight split}. This can only occur when the most recently added overstrand does not intersect its mother split. Note there are some splits that are neither counterclockwise, clockwise, nor straight. See Figure~\ref{clockwise}.
 \end{definition}

It is important to note that straight splits only occur as part of odd
multi-crossing skein relations.
   \begin{figure}[h]
\centering
\includegraphics[width=.5\columnwidth,angle=0] {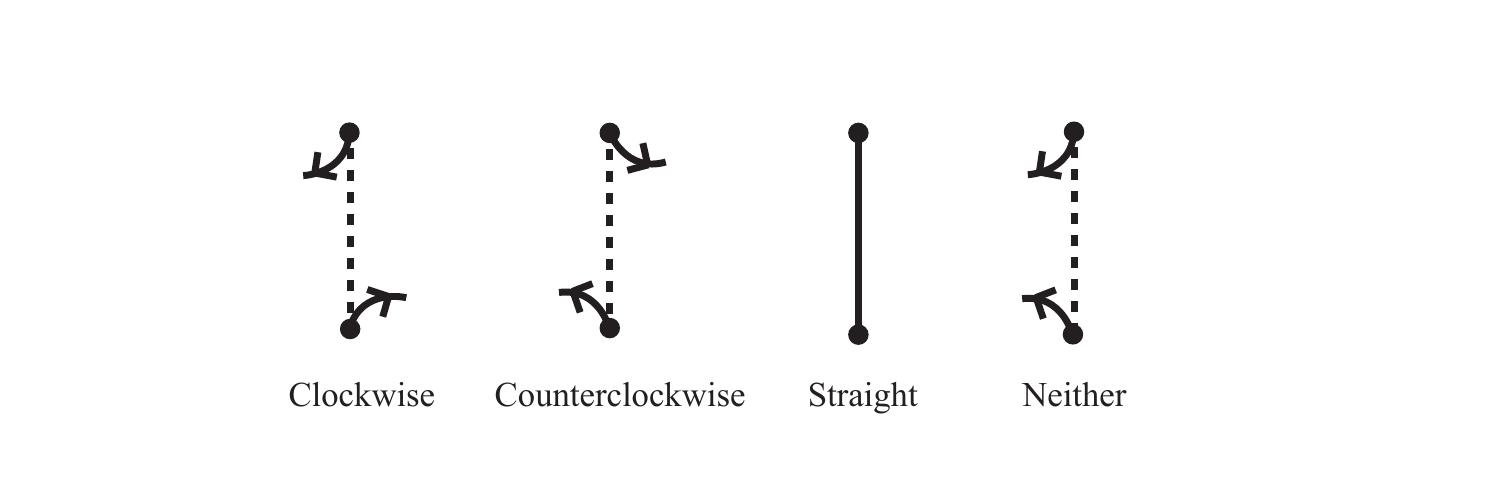}%
\caption{\label{clockwise} We can classify the $n+1$ splits into four distinct types by considering the arcs that emanate from the endpoints of the most recently layed overstrand.}
\end{figure}

 \begin{lemma} \label{rulesoflife} All high splits are either counterclockwise or straight. All low splits are either clockwise or straight.
 \end{lemma}
 
\begin{proof}
If $I(O,S)=0$, $S$ can only generate one term $S'$ that has a straight split. Therefore all straight splits are high or low splits. If  $S$ is a term in a skein relation of an odd multi-crossing, then $I(O,S)\geq 1$. Thus, terms with straight splits only occur in the skein relations from odd multi-crossings. 

Suppose now that $|I(O,S)|\neq 0$. If the first crossing encountered when traveling away from the endpoint of the newest overstrand  is resolved as an A-split, then an arc will extend counterclockwise from at least one endpoint of $O$. One splitting will create counterclockwise arcs at both endpoints if and only if $|I(O,S)| = 1$. Likewise, if the first crossing encountered when traveling away from the endpoint is resolved as a B-split, then an arc will extend clockwise from at least one endpoint of $O$. One splitting will create clockwise arcs at both endpoints if and only if $|I(O,S)| = 1$. A high split is generated by an all-A-split division and a low split is generated by an all-B-split division. Therefore, all high splits that are not straight are counterclockwise and all low splits that are not straight are clockwise.
\end{proof}

\begin{lemma} \label{special mom} 
Let $n \geq 2$, and let $S$ be an $n$-split with  offspring $S'$ from overstrand $O$.  If $|I(O,S)| \le 1$ and $S'$ is a high split (resp. low split), then $S$ is a high split (resp. low split).

\end{lemma}

\begin{proof}
First suppose that $S$ is not a high split and that $|I(O,S)| = 0$.  Then there  exists a term $T$ with $n$ arcs such that $S \prec T$ with respect to the partial ordering.  We call $T$ the aunt of $S'$.  Since $|I(O,S)| = 0$, either $|I(O,T)| = 0$ or $|I(O, T)| = 2$ since $T$ differs from $S$ by a single split move.

By resolving crossings, $T$ generates a term $T'$ that we will call the cousin. Suppose $|I(O,T)|=0$. Then $T'$ is a straight split. The cousin is clearly one split move from $S'$ because they differ only in the location where the mother and aunt differ. See Figure~\ref{fig:aunt1}. Furthermore, the cousin is one level above $S'$ because they each have the same coefficient as their respective mothers. Therefore $S'<T'$, contradicting our original assumption that $S'$ is a high split.

If $|I(O,T)|=2$, we can generate more than one term $T'$. In particular, consider the two terms generated by splitting both intersections with one A-split and one B-split. We will denote both terms as $T'$ for simplicity. Observe that the split of $T'$ and $S'$ are one split move apart. Moreover, the split of the mother $S$ is only one split move away from the split of the aunt $T$. See Figure~\ref{fig:aunt2}. Since $P(T)=P(T')$, the cousin $T'$ is one level above $S'$. Thus, $S'<T'$, contradicting our original assumption that $S'$ is a high split.

\begin{figure}[h]
        \centering
        \begin{subfigure}[b]{0.3\textwidth}
                \centering
                \includegraphics[width=\textwidth]{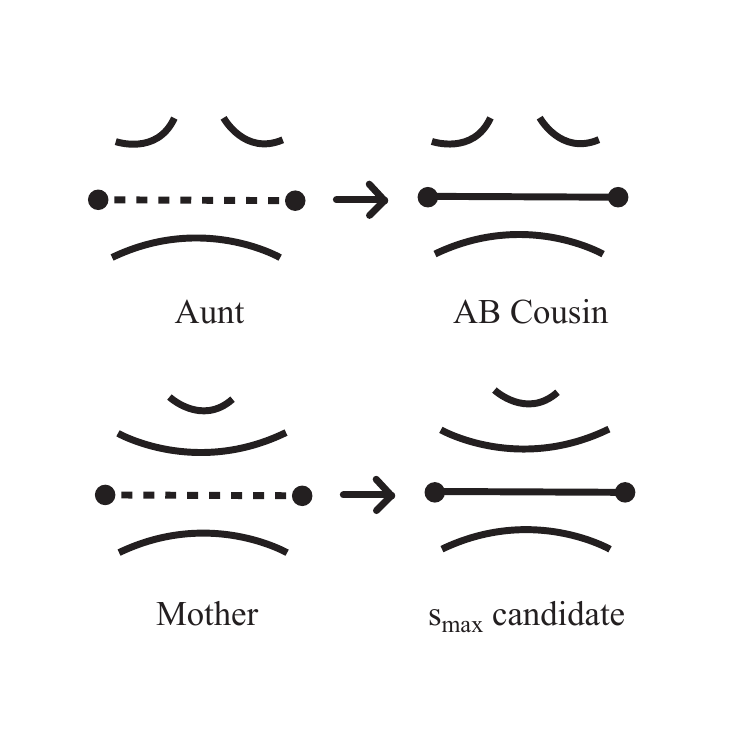}
                \caption{}
                \label{fig:aunt1}
        \end{subfigure}
        \begin{subfigure}[b]{0.3\textwidth}
                \centering
                \includegraphics[width=\textwidth]{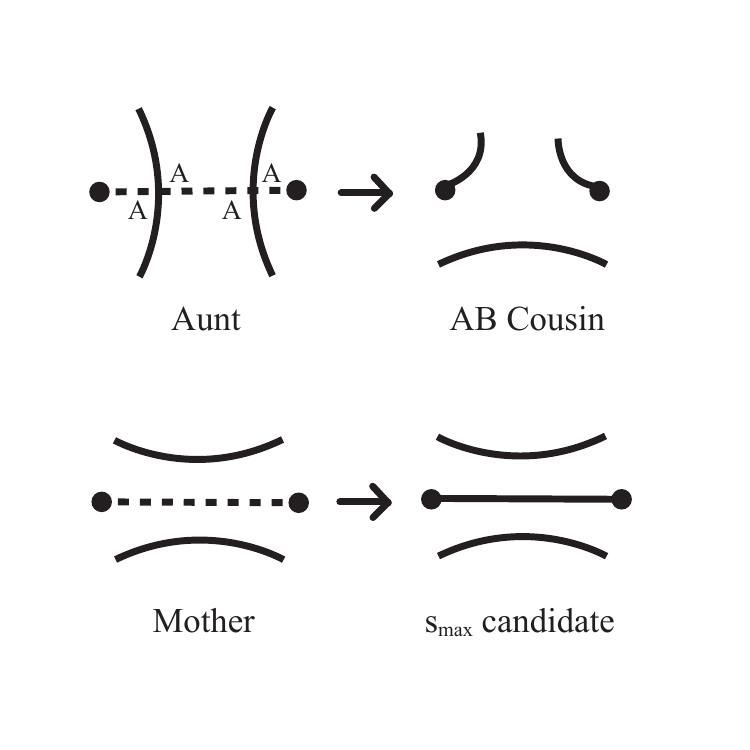}
                \caption{}
                \label{fig:aunt2}
        \end{subfigure}
        \caption{For cases satisfying condition (1), the cousin and
          the high split are one split move apart. These illustrations
          omit arcs that the mother and aunt share. (a) The added
          strand does not intersect the aunt or the mother. The
          resulting cousin and the high split are one split move
          apart. (b) The added strand intersects the aunt twice. The
          resulting cousin and the high split are one split move apart}\label{fig:aunt12}
\end{figure}

Now suppose $S$ is not a high split and  $|I(O,S)|=1$. Then there exists a term $T$ such that $S \prec T$ with respect to the partial ordering. Since the splits of $T$ and $S$ are one split move away, they differ only in two arcs. There are three cases for where  $T$ and $S$ differ. In case (a), the split move that separates the mother and aunt involves the arc in the mother that intersects the added strand, as shown in Figure~\ref{aunta}. In case (b), the split move that separates the mother and aunt involves two arcs on the same side of the added strand in the mother split, as shown in Figure~\ref{auntb}. Finally, in case (c), the split move that separates the mother and aunt involves an arc on each side of the added strand in the mother split. as shown in Figure~\ref{auntc}. The figures show that in each case there exists a cousin $T'$ such that $S'<T'$. This contradicts the assumption that $S'$ is a high split. 

A similar proof follows if we begin with the assumption that $S'$ is a low split.
\end{proof}

 \begin{figure}[h]
 
\begin{subfigure}[t]{0.3\textwidth}
\centering
\includegraphics[width=\textwidth,angle=0] {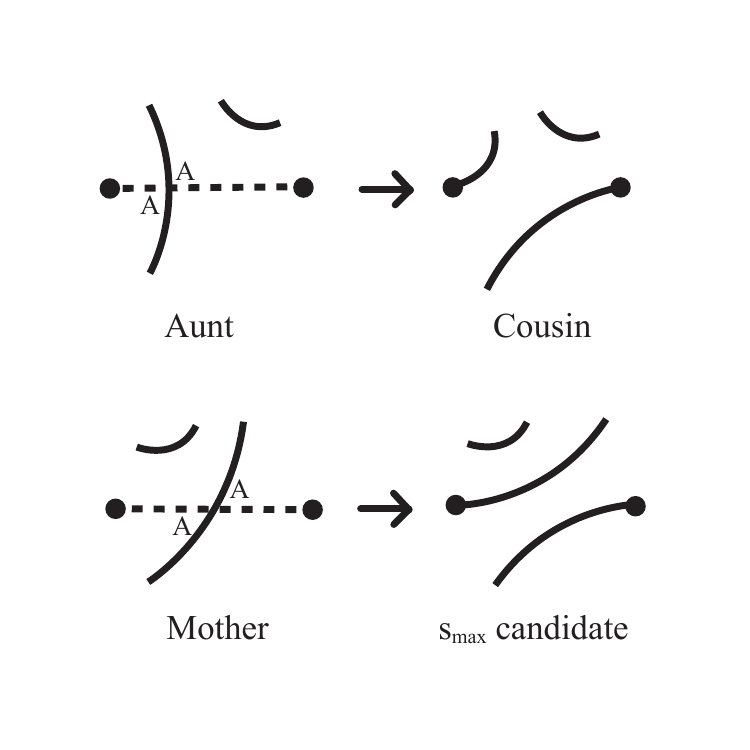}%
\caption{ \label{aunta} }
\end{subfigure}
\begin{subfigure}[t]{0.3\textwidth}
\centering
\includegraphics[width=\columnwidth,angle=0] {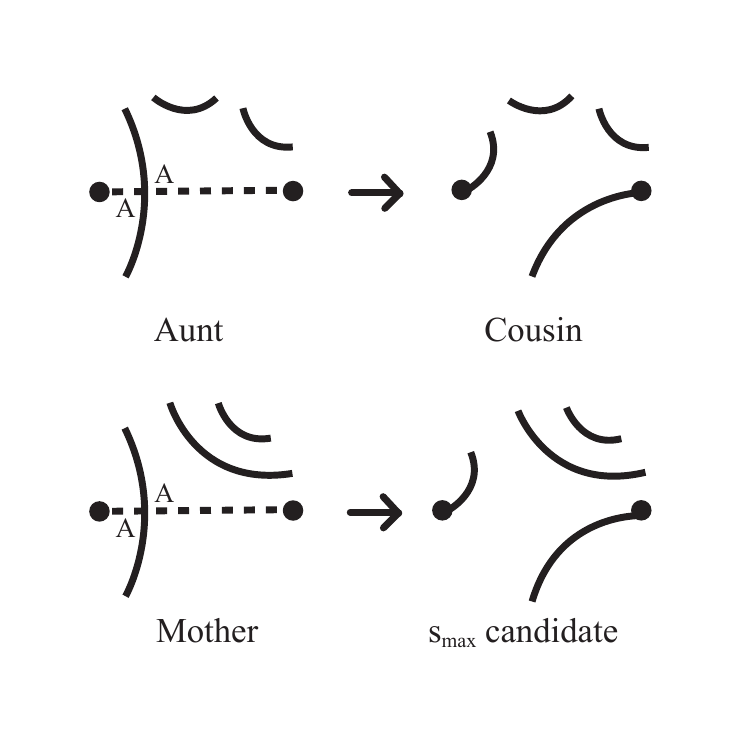}%
\caption{ \label{auntb} }
\end{subfigure}
\begin{subfigure}[t]{0.3\textwidth}
\centering
\includegraphics[width=\columnwidth,angle=0] {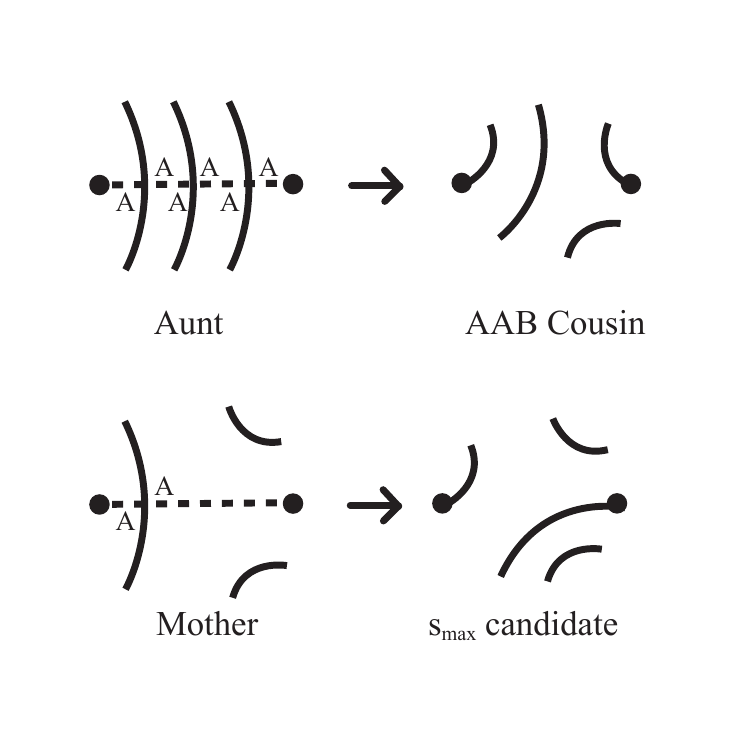}%
\caption{\label{auntc} }
\end{subfigure}
\caption{For cases satisfying condition (2), the cousin and the high split are one split move apart. These illustrations omit arcs that the mother and aunt share. (a)  The split move that separates the mother and aunt involves the arc in the mother that intersects the added strand. (b)  The split move that separates the mother and aunt involves two arcs on the same side of the added strand in the mother split. (c) The split move that separates the mother and aunt involves an arc on each side of the added strand in the mother split. }
\label{fig:4}
\end{figure}

\begin{lemma} \label{nomaxmin} A split cannot be both a high split and a low split. 
\end{lemma}
\begin{proof} Suppose the split of $S'$ were both a high split and a low split. By Lemma \ref{rulesoflife}, the split of $S'$ is straight, since it cannot be both clockwise and counterclockwise. Lemma \ref{special mom} implies that $S$, the the mother of $S'$, is both a high split and a low split and therefore must also be straight. This would imply that both $S$ and $S'$ are part of an even multi-crossing skein, which is a contradiction. \end{proof}

\begin{lemma} 
\label{greaterthanone}
Let $n \geq 3$. Given any high $n$-split $S$ and any low $n$-split $T$, $d(S,T)\geq 2$.
\end{lemma}

 \begin{proof} 

 It suffices to prove the statement for the following cases:
 \begin{enumerate}
\item $S$ is counterclockwise and $T$ is clockwise.
\item $S$ is straight and $T$ is clockwise.
\item $S$ is is counterclockwise and $T$ is straight.
\item $S$ is straight and $T$ is straight.
\end{enumerate}

 \begin{enumerate}
\item
Suppose a counterclockwise high split $S$ were one split move away from a clockwise low split $T$. Then, one split move to a high split will change the orientation of both the counterclockwise arcs that extend from a new endpoint. This can only occur if there are no other arcs or lines passing between the new endpoints. For this to occur, the high split and the low split must be the products of resolving one intersection between a mother and the overstrand as an A-split and one as a B-split.  Since the high split and the low split are one split move away, they differ by only 2 arcs. The four endpoints of these arcs are the same in both splits.  One pair of endpoints is a result of the overstrand, and the other pair of endpoints is from the mother split.  Consequently the mothers of the two splits are identical. We apply Lemma \ref{special mom} and conclude that the mother of the high split is maximal and the mother of the low split is minimal. Since their mothers are identical, this implies the mother split is both maximal and minimal, contradicting Lemma \ref{nomaxmin}. Therefore, a counterclockwise high split cannot be one split move from a clockwise low split.

\item Consider the endpoints of the new strand on the straight split (i.e. the straight strand). In order to make an arc extend clockwise from the top endpoint,
a split move must be performed between the straight strand and an arc on its left. This will yield an arc extending clockwise from the top endpoint and counterclockwise from the bottom endpoint. We will need at least one more split move to construct an arc extending clockwise from the bottom endpoint make this split into a clockwise split. Thus, clockwise and straight splits are more than one split move apart.

\item The argument for part (2) holds.

\item By Lemma \ref{special mom}, the mother of the high split is maximal and the mother the low split is minimal. Neither mother is a straight split because the mothers are part of even multi-crossing skein. Consider the endpoints of the most recently added strand on the mother split. Arcs extend from these points counterclockwise and clockwise respectively on the high split and the low split candidate. There is no way to change the direction of both the top and bottom counterclockwise arcs in a single split move because a straight line rests betweens the two arcs. Therefore, the high and low splits are more than one split move apart. \end{enumerate}
\end{proof}

\medskip

\noindent \textbf{Proof of Theorem 3.1:} By applying Lemma~\ref{general_components_theorem} and Lemma~\ref{greaterthanone}, we have now proved that \\ $|s_{max}| + |s_{min}| \leq (2n - 4)c_n + 2.$



\section{Bound on multi-crossing number in terms of the Span of the Bracket Polynomial} \label{together}

We now use the results proved in Sections \ref{Floor Proof} and \ref{components} to prove our main result. 
\begin{theorem}
\label{we-did-it}
For any projection of a knot or link, $K$, with $c_n$ $n$-crossings with $n \geq 3$, $$\text{span}\langle K \rangle\le \left( \left\lfloor \frac{n^2}{2}\right\rfloor + 4n -8 \right) c_n.$$
\end{theorem}
\begin{proof}
Recall that \[
\text{span}\langle K \rangle \le M-m\leq \sum^{c_n}_{i=1}w(R_{i})+2(|s_{max}|+|s_{min}|-2).\]
By Theorem \ref{Upper Bound}, $w(R_i) \leq \left\lfloor \frac{n^2}{2}\right\rfloor$, and by Theorem \ref{componentsthm}, $|s_{max}| + |s_{min}| \leq (2n - 4)c_n + 2$.  Therefore,
\begin{align*}
\text{span}\langle K \rangle &\leq \sum^{c_n}_{i=1}w(R_{i})+2(|s_{max}|+|s_{min}|-2) \\
	&\leq \left\lfloor \frac{n^2}{2}\right\rfloor c_n + 2((2n - 4)c_n + 2 - 2) \\
    &= \left(\left\lfloor \frac{n^2}{2}\right\rfloor + 4n - 8\right)c_n.
\end{align*}
\end{proof}

The proof of the theorem goes through in the case of $n=2$, except for Lemma~\ref{greaterthanone}. In the case of $n=2$, for high split $T$ and low split $S$, one has instead that $d(S,T) \geq 1$.  This yields the result that $\text{span}\langle K \rangle \le 4 c_2.$

This bound assumes that all the crossings in the knot have skein relations that realize the bound $\lfloor \frac{n^2}{2}\rfloor$. The bound can be improved by considering the width of the different skein relations that correspond to different types of $n$-crossings. For example, in the case of $n=5$, the skein relations of the crossing types $13524$ and $14253$, which are reflections of one another, have width $8$, whereas the 22 other types of $5$-crossings have a width of $12$. 

Therefore, given a 5-crossing projection of a knot or link $K$, if we let $c_{5,8}$ denote the number of crossings of type $13524$ and $14253$, which have skein relations of width 8,  and $c_{5,12}$ denote the number of remaining crossings, each of which has a skein relation of width 12, and  we let $c_{5,P}$ denote the number of crossings in the projection 
$P$, we obtain 

$$\text{span}\langle K \rangle \leq 8c_{5,8} + 12 c_{5,12} +(4(5)-8)c_{5,P} =  20 c_{5,P} + 4 c_{5,12}$$

 In particular, if a link $K$ has a minimal 5-crossing projection such that all of its crossings are of type $13524$ and $14253$, then $\text{span}\langle K \rangle\leq 20c_5(K).$ If instead, we consider all types of $5$-crossings, we obtain the weaker bound $\text{span}\langle K \rangle \leq 24c_{5}$. In a similar manner, we can create tighter bounds for $n>5$ by considering the widths of the specific types of $n$-crossings found in the knot projection.

We can also find a bound on the span of the bracket polynomial in terms of petal number. A petal projection is defined in \citep{Multi} as the projection of a knot K with a single multi-crossing and no nested loops. The petal number of a knot $K$, denoted $p(K)$, is the least number of loops in a petal projection. This  is  equivalent to the number of strands passing through the single $n$-crossing. Every petal projection has a pre-petal projection that can be obtained by pulling the top strand off the crossing, creating one nesting loop.

\begin{corollary}
For any projection of a knot $K$, $ \text{span}\langle K \rangle \leq \left\lfloor \frac{(p(K)-1)^2}{2} \right\rfloor + 4p(K)-12$,  where $p(K)$ is the petal number of $K$.
\end{corollary}

\begin{proof}
Every petal projection has a corresponding pre-petal projection with a single $(p(K)-1)$-crossing, meaning $c_{p(K)-1} = 1$.  Letting $n = p(K) - 1$, we obtain from Theorem \ref{we-did-it}, $ \text{span}\langle K \rangle \leq \left\lfloor \frac{(p(K)-1)^2}{2} \right\rfloor + 4(p(K)-1) - 8$, which is equivalent to $ \text{span}\langle K \rangle \leq \left\lfloor \frac{(p(K)-1)^2}{2} \right\rfloor + 4p(K)-12$.
\end{proof}

\section{The Crossing Spectrum}\label{spectrum}

Recall that every knot has an $n$-crossing number for all $n$ by \citep{Triple}. We can therefore define the \textit{crossing spectrum} of a knot $K$ to be the sequence of multi-crossing numbers $\{c_2(K), c_3(K), \ldots \}$. The Table in Section~\ref{chart} illustrates the crossing spectrum of various low crossing knots. We can consider the crossing spectrum as a knot invariant. It does not distinguish the knots $6_2$, $6_3$, $3_1\# 3_1$ and $3_1\# m3_1$. However, the crossing spectrum distinguishes other knots such as $3_1\#5_1$ from $3_1\#m5_1$, which are often difficult to distinguish.

Adams previously asked whether this crossing spectrum is monotonic. More precisely, is it true that $c_n(K) \ge c_{n+1}(K)$ for all $n\geq 2$? We have a few results that may help us answer this question.  In  \citep{Triple} and \citep{Quadruple}, proofs of the following appear:

\begin{align}
\label{odd-ineq}
 c_2(K) &> c_3(K)\ge c_5(K) \ge \cdots\\ 
\label{even-ineq}
 c_2(K)&> c_4(K) \ge c_6(K) \ge \cdots 
\end{align} 

Further,  it was shown in \citep{Multi} that for any knot or link, K, there exists an $n$, such that $c_n(K) = 1$.  

As shown in the Table in Section~\ref{chart}, there are no known counterexamples to the monotonicity conjecture.  To this collection of known relations we add the following result, which serves to somewhat interlace the inequalities \ref{odd-ineq} and \ref{even-ineq}:

\begin{lemma}\label{2n}
For all $n\geq 2$ and all links $K$, $c_n(K) \ge c_{2n}(K)$ .
\end{lemma}

\begin{proof}
Assume for simplicity that $K$ is a knot.   Let $P_0$ be a projection of $K$  with $c_{n}$ $n$-crossings.  Pick a point $p_0 \in P_0$ and choose an orientation on the knot.  We first construct a projection $P_1$ of $K \amalg 0_1$ with all crossings of order $2n$ save possibly a single double crossing. 

Take a point $p_1$ just to the right of $p_0$ such that $p_1$ is not on the knot.  Draw a path $\gamma$ following $P_0$ along the orientation, intersecting  each crossing transversely and intersecting $P_0$ nowhere else. Initially, this $\gamma$ is on the right of $p_0$; after the first crossing, it will be to the left, as depicted in Figure \ref{fig:unknot}.  Continuing along $P_0$, we see that $\gamma$ switches sides at each crossing, as depicted in Figure \ref{fig:unknot}. 


Hence, after we have followed $P_0$ around all the way back to $p$, we see that $\gamma$ will return to $p$ either on the same side or on the opposite side.  If $\gamma$ returns on the same side, then we simply join $\gamma$ to its starting point; otherwise we cross over $P_0$ at $p$ and join $\gamma$ to its starting point, creating a single double crossing. We depict the scenario in which we create a double crossing in Figure \ref{fig:unknot}.  Note that, we have generated another projection $P_0$ with the now closed path $\gamma$. Therefore, all crossings have valence $2n$, except the single double crossing that was possibly created.


Now we need to assign heights to the strands in the crossings of $\gamma$ so that our projection is in fact a projection of $K \amalg 0_1$.  We may simply have $\gamma$ pass over all crossings of $P_0$, and further choose the height of the strands in the crossings of $\gamma$ by order of first traversal, so that we obtain a projection $P_1$ of $K \amalg 0_1$.

Now we may compose $K$ with $0_1$ along disks as shown in Figure \ref{fig:Composition}, depending on whether or not we had to introduce a double crossing.  It is clear that if we had to introduce a double crossing, we can still compose $K$ with $0_1$ and then remove the double crossing with a type I Reidemeister move.  Hence we obtain a projection of $K \# 0_1 \simeq K$ with exactly $c_n$ $2n$-crossings, proving our claim.

If $K$ is a link, we must choose several basepoints $p_i$ and draw several knots, being sure to compose along sufficiently small disks; the conclusion follows just as in the case that $K$ is a knot.
\end{proof}

\begin{figure}
        \centering
        \begin{subfigure}[b]{0.3\textwidth}
                \centering
                \includegraphics[width=\textwidth]{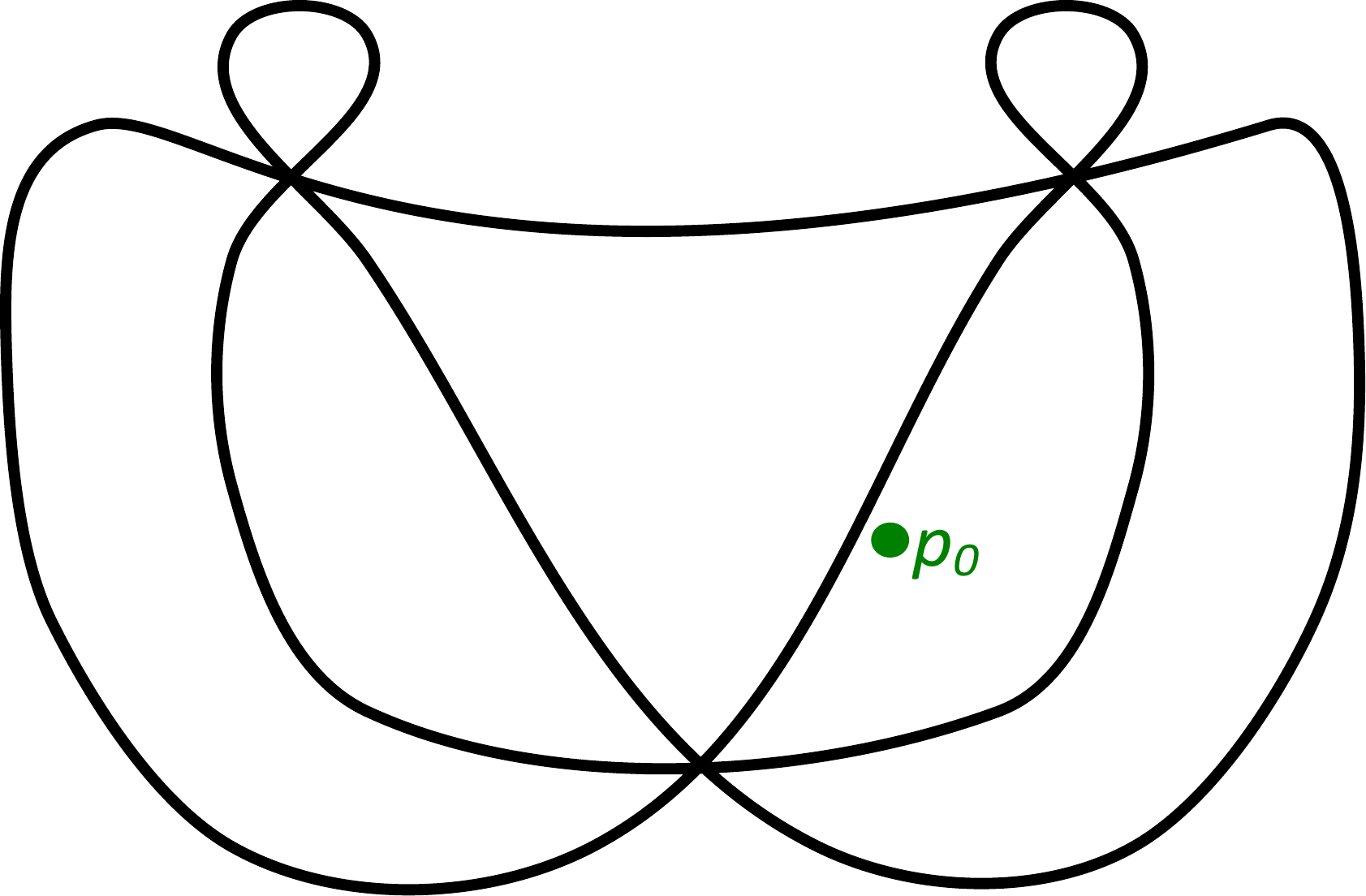}
                \caption{ }
                \label{fig:po}
        \end{subfigure}%
        \begin{subfigure}[b]{0.3\textwidth}
                \centering
                \includegraphics[width=\textwidth]{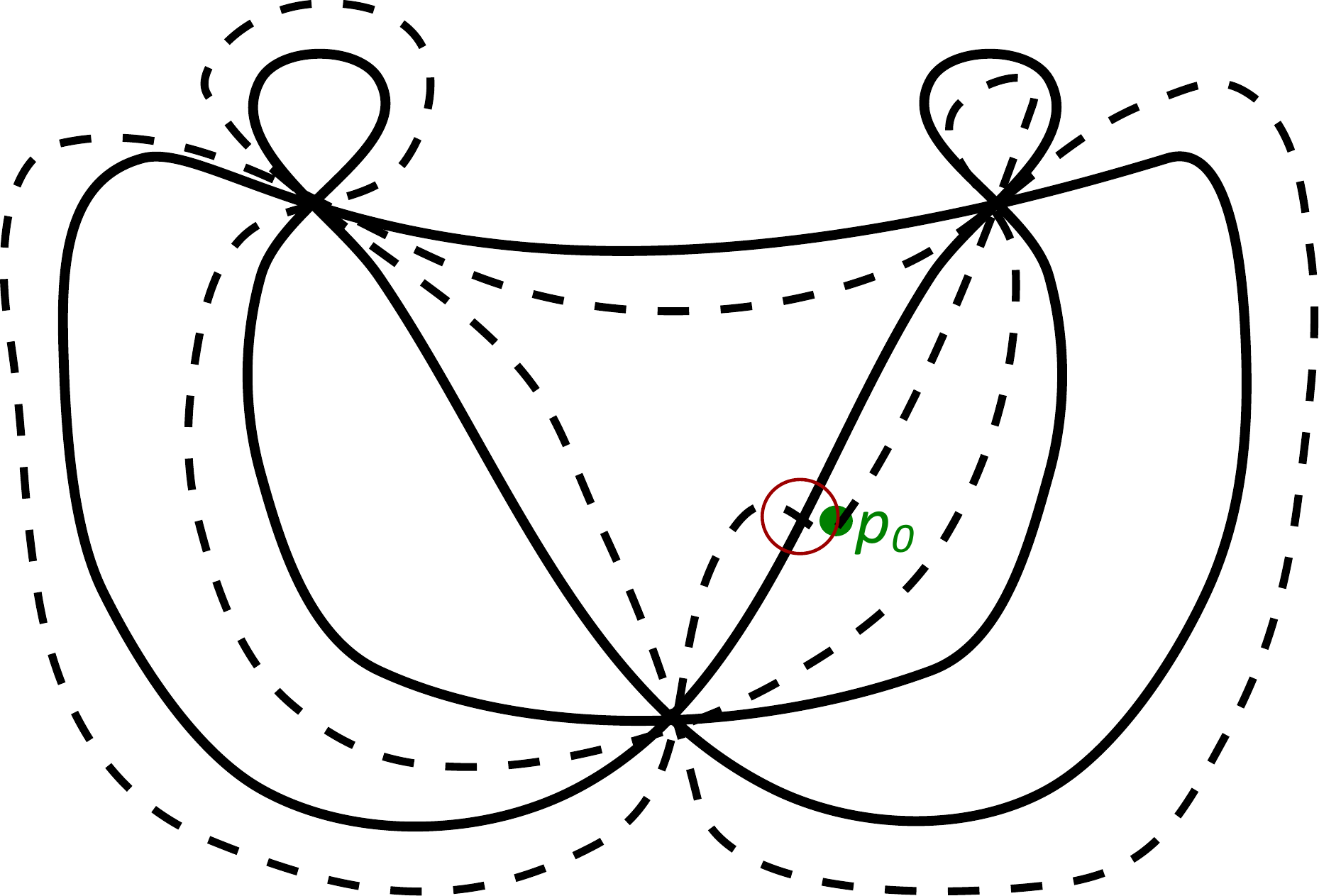}
                \caption{ }
                \label{fig:unknot}
        \end{subfigure}
        \begin{subfigure}[b]{0.3\textwidth}
                \centering
                \includegraphics[width=\textwidth]{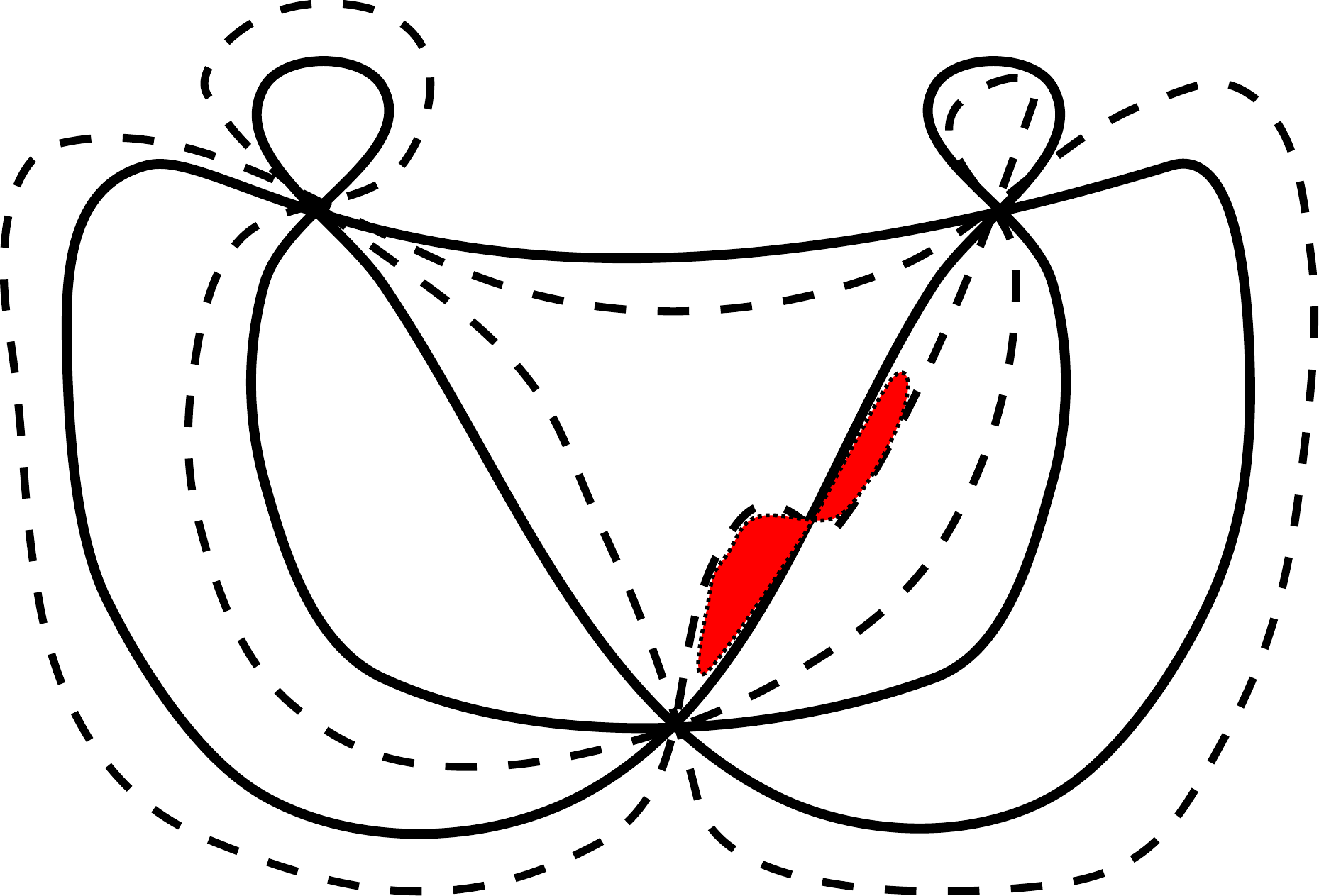}
                \caption{}
                \label{fig:Composition}
        \end{subfigure}
        \caption{(a) Knot diagram with $p_)$ specified. (b) Placing an unkot with the same planar diagram as $K$ on top of $K$. The added double-crossing is circled. (c) Composing $K$ with the unknot along the disk shown.}\label{fig:vitek}
\end{figure}


\section{Additivity of Crossing Number}\label{additivity}

It is known that $c_2(K_1 \# K_2) = c_2(K_1) + c_2(K_2)$ for  wide classes of knots $K_1$ and $K_2$.  So far there  is no counterexample that shows that the minimal double crossing number of any composite knot is not strictly additive. One might hope that the minimal multi-crossing number of any composite knot is also strictly additive.  However, for $n\ge 4$ there are counterexamples to the $c_n$-additivity under composition. 

We first consider the case of $n=3$. No counterexamples have yet been discovered showing that the triple crossing number of composite knots is not strictly additive. There do exist  classes of knots such that triple crossing number is  known to be strictly additive under composition.

In \citep{Triple} Adams introduced two types of triple crossing knots. $K$ is a Type I knot if $K$ is an alternating (in the classical double-crossing context) knot or link with a reduced alternating projection such that 

\begin{enumerate}
\item Every crossing is on the boundary of a complementary region that is a bigon. We refer to a sequence of such bigons touching one another end-to-end as a \textit{bigon chain}.
\item Every bigon chain of maximal length contains an even number of crossings.\
\end{enumerate}

Any collection of crossings in a projection that are the crossings of  the union of maximal bigon chains, each containing an even number of crossings, is said to satisfy the \textit{even bigon chain condition}. 
\\
\\
$K$ is a Type II knot if $K_2$ is an alternating knot or link with a reduced alternating projection where

\begin{enumerate}
\item There exists a circle that bisects exactly 3 crossings.
\item The rest of the crossings in $K$ satisfy the even bigon chain condition. 
\end{enumerate}

\begin{theorem} Let $K_1$ and $K_2$ be either Type I or Type II knots or links. \\

If either $K_1$ or $K_2$ is a Type I knot, then 
$ c_3(K_1\#K_2) = c_3(K_1) + c_3(K_2).$ 

If $K_1$ and $K_2$ are both Type II knots, then
$c_3(K_1) + c_3(K_2) - 1 \le  c_3(K_1\#K_2) \leq c_3(K_1) + c_3(K_2)$. 
\end{theorem}

\begin{proof}
In both cases, the natural composition of minimal triple-crossing diagrams of $K_1$ and $K_2$ yields $c_3(K_1\#K_2) \leq c_3(K_1) + c_3(K_2)$.  Note that as $K_1, K_2$ are alternating, we have $\bspn{K_1} = 4c_2(K_1)$ and $\bspn{K_2} = 4c_2(K_2)$.  A previous result of Adams \citep{Triple} shows that for knots and links of of Type I, $\bspn{K} = 8c_3(K)$ and for knots and links of Type II, $\bspn{K} = 8c_3(K) - 4 \,$.

If both $K_1$ and $K_2$ are Type I knots,  then $K_1\#K_2$ is a Type I knot as well. Hence:

$$8c_{3}(K_1\#K_2) =\bspn{K_1\#K_2}=\bspn{K_1}+\bspn{K_2} = 8c_{3}(K_1)+8c_{3}(K_2)$$

Therefore, $c_3(K_1\#K_2)=c_{3}(K_1)+c_{3}(K_2)$. If $K_1$ is a Type I knot and $K_2$ is a Type II knot then $K_1\#K_2$ is a Type II knot as well. Hence:

$$8c_{3}(K_1\#K_2) - 4 =\bspn{K_1\#K_2}=\bspn{K_1}+\bspn{K_2} = 8c_{3}(K_1)+8c_{3}(K_2)-4$$

Therefore, $c_3(K_1\#K_2)=c_{3}(K_1)+c_{3}(K_2)$. Suppose $K_1$ and $K_2$ are both Type II knots, then 

$$8c_{3}(K_1\#K_2) \geq \bspn{K_1\#K_2}=\bspn{K_1}+\bspn{K_2} = 8c_{3}(K_1) -4+8c_{3}(K_2)-4.$$

Therefore $c_{3}(K_1)+c_{3}(K_2)\geq c_3(K_1\#K_2) \geq c_{3}(K_1)+c_{3}(K_2)-1$.

\end{proof}

\medskip

A counterexample to additivity for $n = 4$ may be constructed as follows. We know from \citep{Quadruple} that $c_4(5_2) = c_4(6_2) = 2$, and claim that $c_4(5_2 \# \overline{6_2}) = 3$.  This is illustrated pictorially in Figure \ref{fig:quad-comp}.  We can perform a move III, defined in \citep{Quadruple} and illustrated in Figure~\ref{fig:quad-comp}, that can turn three double crossings into a quadruple crossing.  Note that the diagrams given for $5_2$ and $\overline{6_2}$ are not quadruple crossing diagrams.

\begin{figure}[H]
\begin{center}
\includegraphics[scale=0.4]{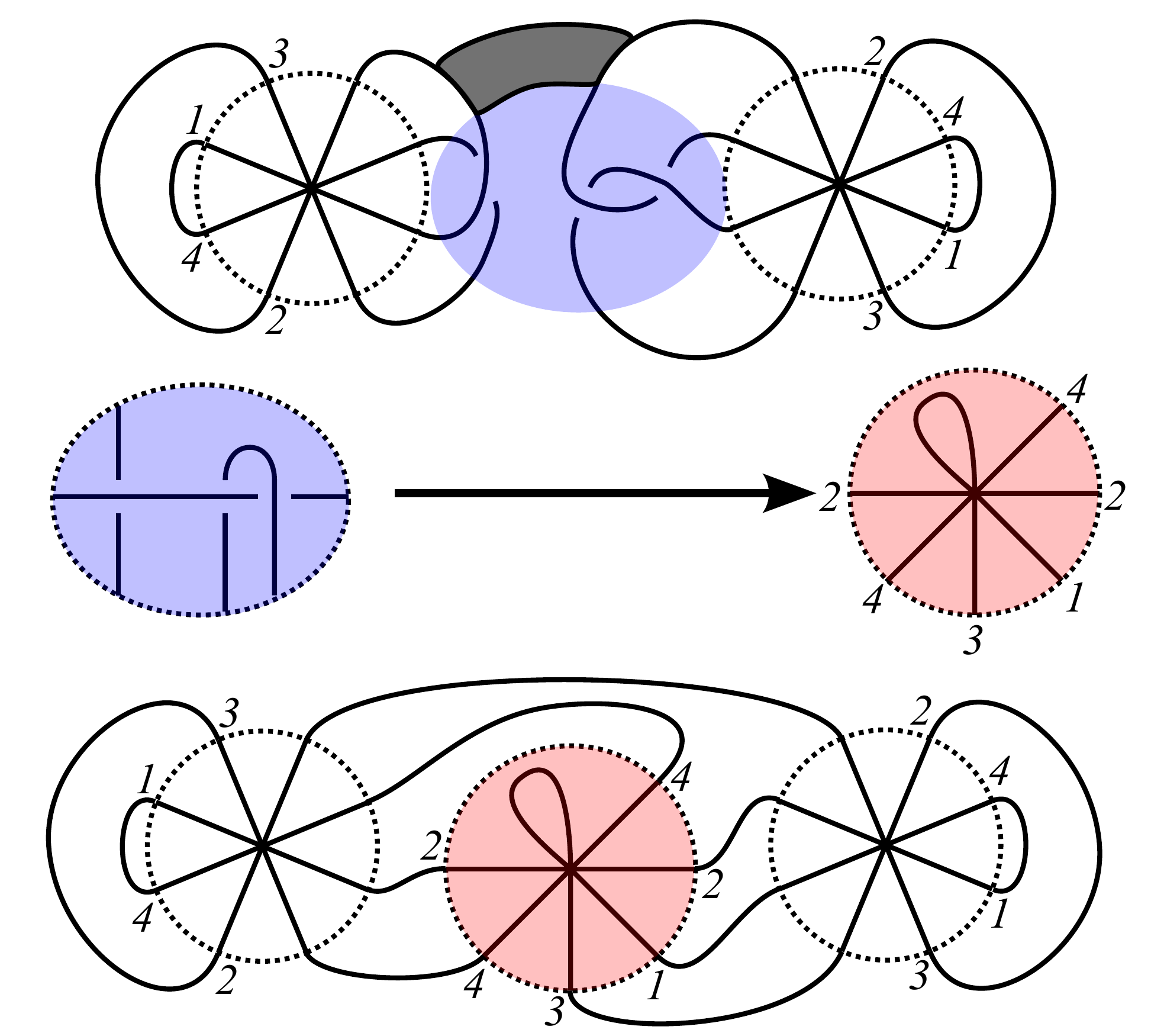}
\caption{An illustration of the composition of $5_2$ and $\overline{6_2}$ violating additivity of quadruple crossing number.}
\label{fig:quad-comp}
\end{center}
\end{figure}

The counterexample to $c_5$-additivity is almost analogous.  We know that $c_5(6_2)=2$ because only the trefoil and figure-eight knots have 5-crossing number 1, and $6_2$ can be realized with 5-crossing number 2.
However, $c_5(6_2 \# 6_2) = 3$.  To see this, we perform a similar construction as above with $6_2$ and $6_2$, starting from a projection with one 5-crossing and two double crossings as in  in Figure \ref{fig:quintcomp}. The required move after composition of the projections is also illustrated in that figure.

\begin{figure}[H]
\begin{center}
\includegraphics[scale=0.8]{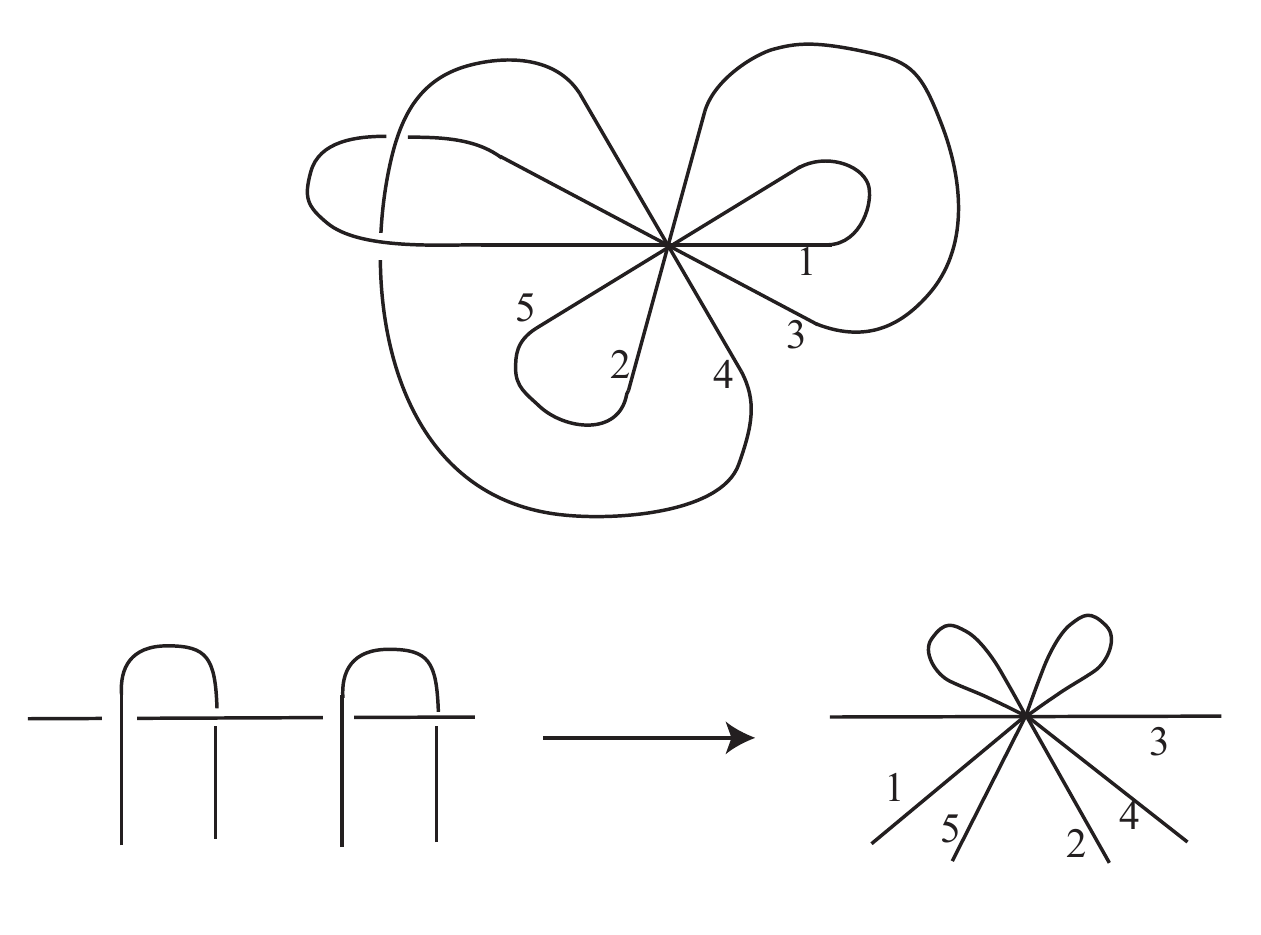}
\caption{Composing $6_2$ with itself using a move that generates three 5-crossings for the result. }
\label{fig:quintcomp}
\end{center}
\end{figure}

For $n \ge 6$ we claim that $c_n(3_1) = c_n(3_1 \# 3_1) = 1$.  From the table we know that $c_6(3_1) = c_7(3_1) = 1$ and $c_6(3_1 \# 3_1) = c_7(3_1 \# 3_1) = 1$. From Lemma \ref{2n}, we know $c_n(K) \ge c_{n+2}(K)$, thereby proving the result.

More generally, let $T_{r,r+1}$ be the $(r, r+1)$-torus knot for $r \geq 2$. Corollary 2.1 of \citep{Geometry} and Corollary 4.16 of \citep{Multi} imply that $c_n(T_{r,r+1}) > 1$ for all $n < 2r$ and $c_n(T_{r,r+1}) = 1$ for all $n \geq 2r$. Corollary 2.3 of \citep{Geometry} implies that $c_{4r}(T_{r,r+1} \# T_{r,r+1}) = 1.$ Hence, $c_n(T_{r,r+1}) =c_{n}(T_{r,r+1} \# T_{r,r+1})$ for all $n \geq 4r$.

We remark that each of these counterexamples satisfies the inequality $c_n(K_1) + c_n(K_2) - 1 \le c_n(K_1 \# K_2) \le c(K_1) + c(K_2)$ for $n\geq4$.  Currently, there are no examples of knots $K_1$ and $K_2$ such that $c_n(K_1 \# K_2)<c_n(K_1) + c_n(K_2) - 1.$


\section{Table}\label{chart}

The following Table illustrates the first extensive list of calculations of $n$-crossing number knots, and consequently the crossing spectrum of certain knots. Note that non-alternating knots tend to have a smaller $n$-crossing number for some $n$ when compared to alternating knots that have the same $c_2$ number. 

We now detail how of each number in the Table was obtained. All numbers with $^\dagger$ came from \citep{Quadruple}  All numbers with $^*$ came by permuting all triple crossing diagrams with $n$ crossings and then identifying them with SnapPy (see \citep{Snap}). If the knot was not previously identified by an exhaustive search of $n-1, n-2 \dots$ triple crossing diagrams, then the knot must have triple-crossing number of $n$. 
All the numbers with $^{**}$ were obtained by using SnapPy to identify diagrams with two quintuple crossings.  If the knot was not previously identified by an exhaustive search of  diagrams with one quintuple crossing, then the knot must have a quintuple-crossing number of $2$. We explored two crossing quintuple diagrams with only two types of quintuple crossings: 14203 and 13024. However, we have not exhausted all $c_5=2$ diagrams with these two types of crossings.
 All numbers with $^\star$ were obtained through our exhaustive method search of  diagrams with a single  $n$-crossing and their subsequent identification by a computer program. 
All numbers with $^\ddagger$ have two even bigon chains of length 4. Since each such bigon chain can be isotoped into a single  quintuple crossing, we were able to find their quintuple crossing number. 
 All numbers with $^\diamond$ have a crossing spectrum with two ones in a row, and by the inequalities  5.1 and 5.2, it must be a one.  
 All numbers with $^\Box$ have  $c_n=2$ because $c_{n-2} = 2$ and the knot did not show up in the exhaustive method of search of diagrams with a single $n$-crossing, so $c_n \neq 1$.  
 All numbers with $^\rhd$ belong to composite knots only, denoted $K_1 \# K_2$.  We know that $c_n(K_1) = c_n(K_2) = 1$ and we know that $c_n(K_1 \# K_2) \leq c_n(K_1) + c_n(K_2)$. Since $K_1 \# K_2$ did not show up on the exhaustive method with $c_n = 1$, then $c_n(K_1 \# K_2) = 2$.  
All numbers with $^\bullet$ were obtained in a similar manner to those numbers obtained as in $^\rhd$ but with compositions of three knots rather than two.

\newpage
\[
\begin{array}[h]{ | l | l | l | l | l | l | l | l | l| | l | l | l | l | l | l | l | l | l | }
\hline
	Knot 	& c_2 	& c_3 	& c_4 		& c_5 		& c_6 		& c_7 & c_8 & c_9 
& 	Knot 	& c_2 	& c_3 	& c_4 		& c_5 		& c_6 		& c_7 & c_8 & c_9 \\ \hline
	3\_1 	& 3 	& 2^*	& 1^\dagger	& 1^\star	& 1^\star 	& 1^\diamond & 1^\diamond & 1^\diamond 
&	9\_10 	& 9 	&  	& 3^\dagger	&  		&  		&  		& 1^\star &  \\ \hline
	4\_1 	& 4 	& 2^*	& 2^\dagger	& 1^\star	& 1^\star 	& 1^\diamond & 1^\diamond & 1^\diamond 
& 	9\_11 	& 9 	&  	& 3^\dagger	&  		&  		&  		& 1 &  \\ \hline
	5\_1 	& 5 	& 4^*	& 2^\dagger	& 2^{**} 	& 1^\star	& 1^\star 	& 1^\diamond & 1^\diamond 
& 	9\_12 	& 9 	& 5^* & 3^\dagger	&  		&  		&  		& 1^\star &  \\ \hline
	5\_2 	& 5 	& 3^* & 2^\dagger	& 2^{**} 	& 1^\star 	& 1^\star 	& 1^\diamond & 1^\diamond 
& 	9\_13 	& 9 	&  	& 3^\dagger	&  		&  		&  		& 1^\star &  \\ \hline
	6\_1 	& 6 	& 3^* & 2^\dagger	& 2^{**} 	& 1^\star	& 1^\star 	& 1^\diamond & 1^\diamond 
& 	9\_14 	& 9 	& 5^* & 3^\dagger	&  		&  		&  		& 1^\star &  \\ \hline
	6\_2 	& 6 	& 4^* & 2^\dagger	& 2^{**} 	& 1^\star 	& 1^\star 	& 1^\diamond & 1^\diamond 
& 	9\_15 	& 9 	& 5^* & 3^\dagger	&  		&  		&  		& 1^\star &  \\ \hline
	6\_3 	& 6 	& 4^* & 2^\dagger	& 2^{**} 	& 1^\star 	& 1^\star 	& 1^\diamond & 1^\diamond 
& 	9\_16 	& 9 	&  	& 3^\dagger	&  		& 	 	&  		& 1 &  \\ \hline
	3\_1\#3\_1 
		& 6 	& 4^* & 2^\dagger	& 2^{**} 	& 1^\star 	& 1^\star 	& 1^\diamond & 1^\diamond 
& 	9\_17 	& 9 	&  & 3 &  		&  		&  		& 1^\star 	&  \\ \hline
	3\_1\#m3\_1 
		& 6 	& 4^* & 2^\dagger	& 2^{**} 	& 1^\star 	& 1^\star 	& 1^\diamond & 1^\diamond 
& 	9\_18 	& 9 	&  	& 3^\dagger	&  		&  		&  		& 1^\star &  \\ \hline
	7\_1 	& 7 	&  	& 2^\dagger	& 2^{**} 	& 2^\Box	& 1^\star 	& 1^\star & 1^\diamond 
& 	9\_19 	& 9 	& 5^* & 3^\dagger	&  		&  		&  		& 1^\star &  \\ \hline
	7\_2 	& 7 	& 4^* & 2^\dagger	&  		& 2^\Box	& 1^\star 	& 1^\star & 1^\diamond 
& 	9\_20 	& 9 	&  	& 3^\dagger	&  		&  		&  		& 1^\star &  \\ \hline
	7\_3 	& 7 	& 5^* & 2^\dagger	& 2^{**} 	& 2^\Box	& 1^\star 	& 1^\star & 1^\diamond 
& 	9\_21 	& 9 	& 5^* & 3^\dagger	&  		&  		&  		& 1^\star &  \\ \hline
	7\_4 	& 7 	& 4^* & 2^\dagger	& 2^\ddagger& 2^\Box	& 1^\star 	& 1^\star & 1^\diamond 
& 	9\_22 	& 9 	&  	& 3^\dagger	&  		&  		&  		& 1^\star &  \\ \hline
	7\_5 	& 7 	& 5^* & 2^\dagger	& 2^{**} 	& 2 ^\Box 	& 1^\star 	& 1^\star & 1^\diamond 
& 	9\_23 	& 9 	& 	& 3^\dagger	&  		&  		&  		&  &  \\ \hline
	7\_6 	& 7 	& 4^* & 3^\dagger	& 	 	&  		& 1^\star 	& 1^\star & 1^\diamond 
& 	9\_24 	& 9 	&  	& 3^\dagger	&  		&  		&  		& 1^\star &  \\ \hline
	7\_7 	& 7 	& 4^* & 3^\dagger	&  		&  		& 1^\star 	& 1^\star & 1^\diamond 
& 	9\_25 	& 9 	& 5^* & 3^\dagger	&  		&  		&  		& 1^\star &  \\ \hline
	3\_1\#4\_1 
		& 7 	& 4^* &  		& 2^\rhd 	& 2^\rhd 	& 1^\star 	& 1^\star & 1^\diamond 
& 	9\_26 	& 9 	&  	&  		&  		&  		&  		& 1^\star &  \\ \hline
	8\_1 	& 8 	& 4^* & 3^\dagger	& 2^\ddagger&  		& 1^\star 	& 1^\star & 1^\diamond 
& 	9\_27 	& 9 	&  	& 3^\dagger	&  		&  		&  		& 1^\star &  \\ \hline
	8\_2 	& 8 	&  	& 3^\dagger	&  		&  		&  		& 1^\star &  
& 	9\_28 	& 9 	&  	& 3^\dagger	&  		&  		&  		& 1^\star &  \\ \hline
	8\_3 	& 8 	& 4^*	& 2^\dagger	& 2^\ddagger& 2^\Box	& 1^\star 	& 1^\star & 1^\diamond 
& 	9\_29 	& 9 	&  	& 3^\dagger	&  		&  		&  		&  	&  \\ \hline
	8\_4 	& 8 	& 5^* & 2^\dagger	&  		& 2^\Box	& 1^\star 	& 1^\star & 1^\diamond 
& 	9\_30 	& 9 	&  	& 3^\dagger	&  		&  		&  		& 1^\star &  \\ \hline
	8\_5 	& 8 	&  	& 3^\dagger	&  		&  		&  		& 1^\star &  
& 	9\_31 	& 9 	&  	&  		&  		&  		&  		& 1^\star &  \\ \hline
	8\_6 	& 8 	& 5^* & 3^\dagger	&  		&  		& 1^\star 	& 1^\star & 1^\diamond 
& 	9\_32 	& 9 	&  	&  		&  		&  		&  		& 1^\star &  \\ \hline
	8\_7 	& 8 	&  	& 3^\dagger	&  		&  		& 1^\star 	& 1^\star & 1^\diamond 
& 	9\_33 	& 9 	&  	&  		&  		&  		&  		& 1^\star &  \\ \hline
	8\_8 	& 8 	& 5^* & 3^\dagger	&  		&  		& 1^\star 	& 1^\star & 1^\diamond 
& 	9\_34 	& 9 	&  	&  		&  		&  		& 		& 	 &  \\ \hline
	8\_9 	& 8 	&  	& 2^\dagger	& 2^{**} 	& 2^\Box	& 1^\star 	& 1^\star & 1^\diamond 
& 	9\_35 	& 9 	& 5^* & 3^\dagger	&  		&  		&  		& 1^\star &  \\ \hline
	8\_10 	& 8 	&  	& 3^\dagger	&  		&  		&  		& 1^\star &  
& 	9\_36 	& 9 	&  	& 3^\dagger	&  		&  		&  		& 1^\star &  \\ \hline
	8\_11 	& 8 	& 5^* & 3^\dagger	&  		&  		&  		& 1^\star &  
& 	9\_37 	& 9 	& 5^* & 3^\dagger	&  		&  		&  		& 1^\star &  \\ \hline
	8\_12 	& 8 	& 4^* & 3^\dagger	&  		&  		&  		& 1^\star &  
& 	9\_38 	& 9 	&  	& 3^\dagger	&  		&  		&  		&  	&  \\ \hline
	8\_13 	& 8 	& 5^*	&  		&  		&  		&  		& 1^\star &  
& 	9\_39 	& 9 	& 5^* & 3^\dagger	&  		&  		&  		& 1^\star &  \\ \hline
	8\_14 	& 8 	& 5^* & 3^\dagger	&  		&  		&  		& 1^\star &  
& 	9\_40 	& 9 	&  	&  		&  		&  		&  		&  	&  \\ \hline
	8\_15 	& 8 	& 5^* & 3^\dagger	& 2^{**} 	&  		& 2^\Box	& 1^\star &  
& 	9\_41 	& 9 	& 5^* & 3^\dagger	&  		&  		&  		& 1^\star &  \\ \hline
	8\_16 	& 8 	&  	&  		&  		&  		&  		& 1^\star &  
& 	9\_42 	& 9 	& 4^* & 3^\dagger	& 2^{**} 	&  		& 1^\star 	& 1^\star & 1^\diamond \\ \hline
	8\_17 	& 8 	&  	&  		& 2^{**} 	&  		& 2^\Box	& 1^\star &  
& 	9\_43 	& 9 	&  	& 3^\dagger	&  		&  		& 1^\star 	& 1^\star & 1^\diamond \\ \hline
	8\_18 	& 8 	&  	&  		& 2^{**} 	&  		&  		& 	 &  
& 	9\_44 	& 9 	& 4^* & 3^\dagger	&  		&  		& 1^\star 	& 1^\star & 1^\diamond \\ \hline
	8\_19 	& 8 	&  	& 2^\dagger	& 2^{**} 	& 1^\star 	& 1^\star 	& 1^\diamond & 1^\diamond 
& 	9\_45 	& 9 	& 4^* & 3^\dagger	&  		&  		& 1^\star 	& 1^\star & 1^\diamond \\ \hline
	8\_20 	& 8 	& 4^* & 2^\dagger	&  		& 1^\star 	& 1^\star 	& 1^\diamond & 1^\diamond 
& 	9\_46 	& 9 	& 4^* & 2^\dagger	&  		& 2^\Box	& 1^\star 	& 1^\star & 1^\diamond \\ \hline
	8\_21 	& 8 	& 4^* & 2^\dagger	& 2^{**} 	& 1^\star 	& 1^\star 	& 1^\diamond & 1^\diamond 
& 	9\_47 	& 9 	&  	&  		& 2^{**} 	&  		& 2^\Box	& 1^\star &  \\ \hline
	3\_1\#5\_1 
		& 8 	&  	&  		&  		& 2^\rhd 	& 2^\rhd 	& 1^\star &  
& 	9\_48 	& 9 	& 4^* & 3^\dagger	&  		&  		&  		& 1^\star &  \\ \hline
	3\_1\#m5\_1 
		& 8 	&  	&  		&  		& 2^\rhd 	& 1^\star 	& 1^\star & 1^\diamond 
& 	9\_49 	& 9 	&  	& 3^\dagger	& 2^{**} 	&  		& 2^\Box	& 1^\star &  \\ \hline
	3\_1\#5\_2 
		& 8 	& 5^* &  		&  		& 2^\rhd 	& 2^\rhd 	& 1^\star &  
& 	3\_1\#3\_1\#3\_1 
		& 9 	& 6^* &  		& 		& 2^\bullet	& 2^\bullet	& 1^\star &  \\ \hline
	3\_1\#m5\_2 
		& 8 	& 5^* &  		&  		& 2^\rhd 	& 1^\star 	& 1^\star & 1^\diamond 
& 	3\_1\#3\_1\#m3\_1 
		& 9 	& 6^* &  		&  		& 2^\bullet 	& 2^\bullet 	& 1^\star &  \\ \hline
	4\_1\#4\_1 
		& 8 	& 4^* &  		& 2^\rhd 	& 2^\rhd 	& 2 		& 1^\star &  
& 	3\_1\#6\_1 
		& 9	& 5^* &  		&  		& 2^\rhd 	& 2^\rhd 	& 1^\star &  \\ \hline
	9\_1 	& 9 	&  	& 3^\dagger	&  		&  		&  		& 1^\star &  
& 	3\_1\#m6\_1 
		& 9 	& 5^* &  		&  		& 2^\rhd 	& 2^\rhd 	& 1^\star &  \\ \hline
	9\_2 	& 9 	& 5^* & 3^\dagger	&  		&  		&  		& 1^\star &  
& 	3\_1\#6\_2 
		& 9 	&  	&  		&  		& 2^\rhd 	& 2^\rhd 	& 1^\star &  \\ \hline
	9\_3 	& 9 	& 	& 3^\dagger	&  		&  		&  		& 1^\star &  
& 	3\_1\#m6\_2 
		& 9 	& 	&  		&  		& 2^\rhd 	& 2^\rhd 	& 1^\star &  \\ \hline
	9\_4 	& 9 	&  	& 3^\dagger	&  		&  		&  		& 1^\star &  
& 	3\_1\#6\_3 
		& 9 	&  	&  		&  		& 2^\rhd 	& 2^\rhd 	& 1^\star &  \\ \hline
	9\_5 	& 9 	& 5^* & 3^\dagger	&  		&  		&  		& 1^\star &  
& 	4\_1\#5\_1 
		& 9 	&  	&  		&  		& 2^\rhd 	& 2^\rhd 	& 1^\star &  \\ \hline
	9\_6 	& 9 	&  	& 3^\dagger	&  		&  		&  		& 1^\star &  
& 	4\_1\#m5\_1 
		& 9 	&  	&  		&  		& 2^\rhd 	& 2^\rhd 	& 1^\star &  \\ \hline
	9\_7 	& 9 	&  	& 3^\dagger	&  		&  		&  		& 1^\star &  
& 	4\_1\#5\_2 
		& 9 	& 5^* &  		&  		& 2^\rhd 	& 2^\rhd 	& 1^\star &  \\ \hline
	9\_8 	& 9 	& 5^* & 3^\dagger	&  		&  		&  		& 1^\star &  
& 	4\_1\#m5\_2 
		& 9 	& 5^* &  		&  		& 2^\rhd 	& 2^\rhd 	& 1^\star &  \\ \hline
	9\_9 	& 9 	&  	& 3^\dagger	&  		&  		&  		& 1^\star &  
&  &  &  &  &  &  &  &  &  \\ \hline
\end{array}\]


\bibliographystyle{alpha}
\bibliography{algebra-references}

\end{document}